\documentclass[12pt,oneside]{amsart}

\usepackage{amsmath,amssymb,amsthm,xcolor}
\usepackage[a4paper, total={5.5in, 9in}]{geometry}

\usepackage{stmaryrd}

\usepackage{enumitem}

\usepackage[backref]{hyperref}

\theoremstyle{plain}
\newtheorem{theorem}{Theorem}[section]
\newtheorem{corollary}[theorem]{Corollary}
\newtheorem{proposition}[theorem]{Proposition}
\newtheorem{lemma}[theorem]{Lemma}

\newtheorem{definition}[theorem]{Definition} 

\theoremstyle{definition}
\newtheorem{example}[theorem]{Example}
\newtheorem{remark}[theorem]{Remark}

\newcommand{\Bl}{\mathrm{Bl}}
\newcommand{\Lich}{\mathcal{D}}

\newcommand{\Tr}{\mathrm{Tr}}
\newcommand{\Hess}{\mathrm{Hess}}
\newcommand{\Ric}{\mathrm{Ric}}
\renewcommand{\d}{\partial}
\newcommand{\db}{\overline{\partial}}

\renewcommand{\O}{\mathrm{O}}
\newcommand{\Id}{\mathrm{Id}}
\newcommand{\Riem}{\mathrm{Rm}}

\newcommand{\pr}{\mathrm{pr}}

\begin{document}
	
	\title{Weighted extremal metrics on blowups}
	\author{Michael Hallam}
	\address{Department of Mathematics, Aarhus University, Ny Munkegade 118, DK-8000 Aarhus C, Denmark.}
	\email{hallam@math.au.dk}
	\maketitle

	\begin{abstract}
		We show that if a compact K{\"a}hler manifold admits a weighted extremal metric for the action of a torus, so too does its blowup at a relatively stable point that is fixed by both the torus action and the extremal field. This generalises previous results on extremal metrics by Arezzo--Pacard--Singer and Sz{\'e}kelyhidi to many other canonical metrics, including extremal Sasaki metrics, deformations of K{\"a}hler--Ricci solitons and \(\mu\)-cscK metrics. In a sequel to this paper, we use this result to study the weighted K-stability of weighted extremal manifolds.
	\end{abstract}

	\tableofcontents
	
	\section{Introduction}
	
		Ever since the seminal work of Calabi \cite{Cal54}, the problem of whether a compact K{\"a}hler manifold admits a canonical K{\"a}hler metric has been a driving force in the field of K{\"a}hler geometry. The most studied among canonical metrics are the K{\"a}hler--Einstein metrics, but much research in the last thirty years has been motivated by the constant scalar curvature K{\"a}hler (cscK) or, more generally, extremal problem \cite{Cal82}. Alongside these, there are many other notions of canonical metric available, for example K{\"a}hler--Ricci solitons which arise as (possibly singular) limits of the K{\"a}hler--Ricci flow \cite{CW20,CSW18}. Other important examples include extremal Sasaki metrics \cite{BGS08}, which provide a notion of canonical metric on manifolds of odd real dimension, and conformally K{\"a}hler--Einstein--Maxwell (cKEM) metrics \cite{LeB10, AM19}.
		
		It was shown recently that all of these examples can be treated under the same framework through the notion of a \emph{weighted extremal metric}, due to Lahdili \cite{Lah19}. Around the same time, Inoue independently introduced a closely related generalisation of cscK metrics and K{\"a}hler--Ricci solitons called \emph{\(\mu\)-cscK} metrics, which form a special subclass of Lahdili's metrics \cite{Ino22}. 
		
		To define the weighted extremal equation, one considers a compact real torus \(T\) acting on a compact K{\"a}hler manifold \((M,\omega)\) by hamiltonian isometries. This has a moment map \(\mu:M\to\mathfrak{t}^*\) with image \(P\subset\mathfrak{t}^*\) a convex polytope. The ``weights" of the weighted cscK equation are then smooth positive functions \(v,w:P\to\mathbb{R}_{>0}\). Using these functions, one can deform the scalar curvature \(S(\omega)\) to the \emph{\((v,w)\)-weighted scalar curvature}: \[S_{v,w}(\omega):=\frac{1}{w(\mu)}\left(v(\mu)S(\omega)-2\Delta(v(\mu))+\frac{1}{2}\Tr(g\circ\Hess(v)(\mu))\right).\] A \emph{\((v,w)\)-weighted extremal metric} is a metric \(\omega\) such that the gradient of \(S_{v,w}(\omega)\) is a real holomorphic vector field, called the \emph{extremal field}. For the special case in which \(S_{v,w}(\omega)\) is constant, we call \(\omega\) a \emph{\((v,w)\)-weighted cscK metric}. Through various choices of the weight functions \(v\) and \(w\), one can recover all of the examples of canonical metrics described above; we refer to Section \ref{sec:background} for more details and precise definitions.
		
		In the setting of cscK metrics, an important general construction was given by Arezzo--Pacard, who showed that if a manifold with discrete automorphism group admits a cscK metric, so too does its blowup at a point \cite{AP06}. They also showed that one can drop the assumption of discrete automorphisms, provided one instead blows up a suitable sufficiently large collection of points \cite{AP09}. The results were later extended by Arezzo--Pacard--Singer and Sz{\'e}kelyhidi to the extremal setting \cite{APS11,Sze12}, where one blows up points fixed by the extremal field. In \cite{APS11} the points satisfy collections of conditions concerning automorphism groups; in \cite{Sze12} it was clarified that the points need only satisfy a suitable stability condition. In the cscK case, Sz{\'e}kelyhidi considered the case of blowing up a single point, characterising existence of a cscK metric on the blowup in terms of K-stability of the blowup in complex dimension \(n>2\) \cite{Sze15}. Recently, Dervan--Sektnan extended this result of Sz{\'e}kelyhidi to the extremal case (and extended the cscK result to complex dimension \(n=2\)), classifying when the blowup at a point admits an extremal metric in terms of relative K-stability of the blowup \cite{DS21}.
		
		Not only do these results furnish many examples of extremal metrics, but they are also historically important for their use in partially proving various formulations of the Yau--Tian--Donaldson conjecture. For instance, after Donaldson proved that a polarised manifold admitting a cscK metric is K-semistable \cite{Don05}, Stoppa used Donaldson's result along with the theorem of Arezzo--Pacard \cite{AP06} to prove that polarised manifolds with discrete automorphisms are K-stable \cite{Sto09,Sto11}. Later, Stoppa--Sz{\'e}kelyhidi used the blowup result of Arezzo--Pacard--Singer for extremal metrics \cite{APS11} to prove relative K-polystability of extremal manifolds \cite{SS11}.
		
		In this paper, we generalise the results \cite{APS11,Sze12} of Arezzo--Pacard--Singer and Sz{\'e}kelyhidi to the weighted extremal setting. Let \((M,\omega)\) be a compact K{\"a}hler manifold, and let \(T\) be a real torus acting effectively on \(M\) by hamiltonian isometries, with moment map \(\mu:M\to\mathfrak{t}^*\). Given a \(T\)-invariant K{\"a}hler potential \(\varphi\) with respect to \(\omega\), one can write down an explicit moment map \(\mu_\varphi\) for \(\omega_\varphi:=\omega+i\d\db\varphi\) such that \(\mu_\varphi(M)=P=\mu(M)\). Thus, modifying the moment map in this way, one can define the weighted scalar curvature \(S_{v,w}(\omega_\varphi)\) for any such \(\varphi\), and seek a \(T\)-invariant solution of the weighted extremal equation in the class \([\omega]\).
		
		Let \(p\in M\) be a fixed point of the \(T\)-action, and denote by \(\pi:\Bl_pM\to M\) the blowup of \(M\) at \(p\), with exceptional divisor \(E\subset\Bl_pM\). The \(T\)-action on \(M\) lifts to \(\Bl_pM\), and for \(\epsilon>0\) sufficiently small the cohomology class \([\pi^*\omega-\epsilon E]\) contains a \(T\)-invariant K{\"a}hler metric \(\omega_\epsilon\). In Lemma \ref{lem:epsilon_moment_map} below, we show that there exists a moment map \(\mu_\epsilon:\Bl_pM\to\mathfrak{t}^*\) for \(\omega_\epsilon\) whose moment polytope \(P_\epsilon:=\mu_\epsilon(\Bl_pM)\) is contained in the moment polytope \(P:=\mu(M)\). Thus, given a choice of weight functions \(v\) and \(w\) on \(P\), we can restrict these to \(P_\epsilon\), and search for a \((v,w)\)-weighted extremal metric in the class \([\pi^*\omega-\epsilon E]\). 
		
		In order for the blowup \(\Bl_pM\) to admit a weighted extremal metric, we will need the point \(p\) to satisfy a stability condition. Roughly, we first construct a certain subgroup \(H\) of the group of \(T\)-commuting hamiltonian isometries \(G\) of \((M,\omega)\). Choosing an invariant inner product on the Lie algebra \(\mathfrak{h}\) of \(H\), we may identify \(\mathfrak{h}\) with its dual \(\mathfrak{h}^*\), and thus consider the moment map \(\mu_H:M\to\mathfrak{h}^*\) for the \(H\)-action as a map \(\mu_H^{\#}:M\to\mathfrak{h}\). The point \(p\) is \emph{relatively stable} if \(\mu_H^{\#}(p)\in\mathfrak{h}_p\), i.e. the vector field generated by \(\mu_H^{\#}(p)\) fixes \(p\). We refer to Section \ref{sec:deformation} for the full definition of \(H\), but for now we remark that if the torus \(T\) is maximal then \(H=T\), and any fixed point of \(T\) will automatically be relatively stable.
		
		\begin{theorem}\label{thm:main}
			Let \((M,\omega)\) be a \((v,w)\)-weighted extremal manifold, and let \(p\in M\) be a relatively stable point that is fixed by both the \(T\)-action and the extremal field. Denote by \(\pi:\Bl_pM\to M\) the blowup of \(M\) at \(p\) with exceptional divisor \(E\subset\Bl_pM\). Then for all \(\epsilon>0\) sufficiently small, the class \([\pi^*\omega-\epsilon^2 E]\) contains a \((v,w)\)-weighted extremal metric.
		\end{theorem}
		
		This result can be generalised to the following:
	
	\begin{theorem}\label{thm:extension}
		Let \((M,\omega)\) be a \((v,w)\)-weighted extremal manifold of dimension \(n\), and let \(p_1,\ldots,p_N\) be a collection of \(T\)-fixed points in \(M\) that are also fixed by the extremal field. Let \(a_1,\ldots,a_N\in\mathbb{R}_{>0}\) be such that the vector field generated by \[\sum_{j=1}^Na_j^{n-1}\mu_H^{\#}(p_j)\in \mathfrak{h}\] vanishes at each of the \(p_j\). Denote by \(\pi:\Bl_{p_1,\ldots,p_N}M\to M\) the blowup of \(M\) at the points \(p_1,\ldots,p_N\) with exceptional divisors \(E_j:=\pi^{-1}(p_j)\). Then for all \(\epsilon>0\) sufficiently small, the class \([\pi^*\omega-\epsilon^2(a_1E_1+\cdots+a_NE_N)]\) contains a \((v,w)\)-weighted extremal metric.
	\end{theorem} 
	
	\begin{remark}
		Although we do not explicitly prove Theorem \ref{thm:extension} in this paper, as in \cite{Sze12}, we note that the techniques of the proof of Theorem \ref{thm:main} can be extended to prove Theorem \ref{thm:extension} in a straightforward manner.
	\end{remark}

		The structure of the proof follows closely the work \cite{Sze12}. Namely, we first construct an approximate solution \(\omega_\epsilon\) of the weighted extremal equation by gluing \(\omega\) to a rescaling of a model metric near the exceptional divisor \(E\), called the \emph{Burns--Simanca metric}. The approximate solution \(\omega_\epsilon\) is then deformed to a genuine solution via the contraction mapping theorem. 
		
		Despite the structural similarities, there are many new technical obstacles that arise in the course of the proof. First among these is the construction of the moment map \(\mu_\epsilon\), whose image \(P_\epsilon\) is contained in \(P\). It is not enough to know that \(\mu_\epsilon\) exists, but we also require a careful understanding of \(\mu_\epsilon\) near the exceptional divisor. Furthermore, there are many new terms in the weighted extremal equation that must be estimated in order to apply the contraction mapping theorem. The bulk of the groundwork towards these estimates is carried out in Section \ref{sec:estimates_moment_maps}, where we use the explicit description of \(\mu_\epsilon\) from Lemma \ref{lem:epsilon_moment_map}.
		
		There is one conceptual point of interest that arises in the proof, namely why it is possible to glue in the Burns--Simanca metric---a scalar flat metric---rather than some \((v,w)\)-weighted analogue which has vanishing \((v,w)\)-weighted scalar curvature. To give a rough justification, near the exceptional divisor \(E\), the image of \(\mu_\epsilon\) is a small region of \(P_\epsilon\), on which the weight functions \(v\) and \(w\) appear approximately constant. The Burns--Simanca metric can be considered as a weighted cscK metric with \emph{constant weight functions}. Thus, it is plausible that on this region we can deform the Burns--Simanca metric to a weighted cscK metric with the required weight functions. This is partly justified by deformation results on weighted extremal metrics found in \cite[Chapter 6.1]{Hal22}. Rather than taking this path of deforming the Burns--Simanca metric to some other metric which we glue in, we simply deform the approximate metric \(\omega_\epsilon\) built from the Burns--Simanca metric directly to a weighted extremal metric in one clean stroke.

		Taking different choices of weight functions \(v\) and \(w\), we obtain analogues of \cite{APS11,Sze12} for other kinds of canonical K{\"a}hler metrics. For example, suppose that \([\omega]\) is the first Chern class \(c_1(L)\) of an ample line bundle \(L\). For certain choices of \(v,w\) depending on an element \(\xi\in\mathfrak{t}\), \((v,w)\)-weighted extremal metrics on \(M\) correspond to extremal Sasaki metrics on the unit circle bundle \(S\) of \(L^*\) with Sasaki--Reeb vector field \(\xi\) \cite{AC21}.
		
		\begin{corollary}\label{cor:Sasaki}
			Let \((M,L)\) be a \(T\)-equivariant polarised manifold, and suppose that \(\omega\in c_1(L)\) induces an extremal Sasaki metric on \(S\subset L^*\) with Sasaki--Reeb field \(\xi\in\mathfrak{t}\). Let \(p\in M\) be a relatively stable point fixed by \(T\) and the extremal field. Then for all rational \(\epsilon>0\) sufficiently small there exists a K{\"a}hler metric \(\omega_\epsilon\) on the blowup \(\Bl_pM\) in the class \(c_1(\pi^*L-\epsilon^2E)\) satisfying the following: for all \(k>0\) such that \(k\epsilon\in\mathbb{Z}\), \(k\omega_\epsilon\) induces an extremal Sasaki metric on the unit circle bundle of \(k(-\pi^*L+\epsilon^2E)\) with Sasaki--Reeb field \(\xi\).
		\end{corollary}
		
		We have stated the application to extremal Sasaki metrics because this is the main instance in which the interpretation of the theorem is clear cut. For example, we cannot conclude that the blowup of a manifold with a K{\"a}hler--Ricci soliton also admits a K{\"a}hler--Ricci soliton. For starters, the blowup may not be Fano, and even if it is, the class \([\pi^*\omega-\epsilon^2E]\) only admits a weighted extremal metric for \(\epsilon>0\) sufficiently small, so cannot be taken to be the canonical class in general. What's more, the K{\"a}hler--Ricci soliton equation corresponds to a weighted extremal metric with weight functions \(v=w=e^{\langle\xi,-\rangle}\) \emph{and} extremal field \(\xi\in\mathfrak{t}\). Our blowup result leaves the weight functions \(v\) and \(w\) unchanged, however we do not have control over the extremal field---in general it will only be a small deformation of the original extremal field \(\xi\), so will not match the weight functions. Thus, the metric we produce is only a small deformation of a K{\"a}hler--Ricci soliton, rather than a genuine K{\"a}hler--Ricci soliton.
		
		Nonetheless, it is still useful to know of the existence of a weighted extremal metric on the blowup. For example, this carries interesting ramifications about the automorphism group of the blowup, via the Matsushima--Lichnerowicz theorem for weighted extremal manifolds \cite[Theorem B.1]{Lah19}. Furthermore, the blowup result can be used to prove relative weighted K-polystability of manifolds admitting weighted extremal metrics. Indeed, in the sequel \cite{sequel} to this paper, we prove that a weighted extremal manifold is relatively weighted K-polystable with respect to a maximal torus. The proof is along the lines of Stoppa \cite{Sto09, Sto11}, Stoppa--Sz{\'e}kelyhidi \cite{SS11}, and Dervan \cite{Der18}, and directly uses Theorem \ref{thm:main} along with the weighted K-semistability of weighted cscK manifolds proven by Lahdili \cite{Lah19} and Inoue \cite{Ino20}.
		
		\renewcommand{\abstractname}{Acknowledgements}
		\begin{abstract}
		I thank Eveline Legendre for a helpful remark on moment polytopes, and Lars Sektnan for valuable advice on weighted H{\"o}lder spaces and comments on the manuscript. I also thank Zakarias Sj{\"o}str{\"o}m Dyrefelt and Ruadha{\'i} Dervan for their interest and comments, and the referee for their helpful remarks.
		\end{abstract}
	
	\section{Background}\label{sec:background}
	
		\subsection{K{\"a}hler geometry}
		
			We briefly lay out our notation and terminology for K{\"a}hler metrics. Let \((M,\omega)\) be a compact \(n\)-dimensional K{\"a}hler manifold. The Ricci curvature of \(\omega\) is \[\Ric(\omega):=-\frac{i}{2\pi}\d\db\log\omega^n,\] which is shorthand for choosing a system of local holomorphic coordinates \((z_1,\ldots, z_n)\), and defining \(\Ric(\omega):=-\frac{i}{2\pi}\d\db \log\det(g_{\j\bar{k}})\); the definition is then independent of the choice of local holomorphic coordinates. The scalar curvature is \[S(\omega):=\Lambda_\omega\Ric(\omega)=\frac{n\,\Ric(\omega)\wedge\omega^{n-1}}{\omega^n}.\]  
			
			For a smooth function \(f\in C^\infty(M,\mathbb{C})\), we write \(\nabla^{1,0}f\) for the projection of the gradient \(\nabla f\) to \(T^{1,0}M\). In local coordinates, \(\nabla^{1,0}f=g^{j\bar{k}}\d_{\bar{k}}f\). The operator \(\Lich:C^\infty(M,\mathbb{C})\to\Gamma(M,T^{1,0}\otimes\Omega^{0,1})\) is defined by \[\Lich f:=\db\nabla^{1,0}f,\] where \(\db\) is the del-bar operator of the holomorphic vector bundle \(T^{1,0}M\). The \emph{Lichnerowicz operator} is \[f\mapsto\Lich^*\Lich f,\] where the adjoint of \(\Lich\) is taken with respect to the \(L^2\)-metric on tensors determined by \(\omega\). It is straightforward but tedious to derive the following product rule for the adjoint of \(\Lich\): \begin{equation*}\label{eq:Lich_product}
				\Lich^*(fA)=f\Lich^*A-(\db^*A,\nabla^{1,0}f)-(PA,\db f)+(\Lich f,A),
			\end{equation*} for \(f\in C^\infty(M,\mathbb{C})\) and \(A\in\Gamma(M,T^{1,0}\otimes\Omega^{0,1})\). Here \((\,,)\) is the pointwise hermitian inner product on tensors, and \(P:\Gamma(M,T^{1,0}\otimes\Omega^{0,1})\to\Gamma(M,\Omega^{0,1})\) is the first order linear differential operator \[PA:=-g^{\bar{k}m}\d_m(g_{\bar{k}\ell}A^\ell_{\bar{j}})d\overline{z}^j.\] 
		
			Note the adjoint operator \(\db^*\) is given by the formula \begin{equation}\label{eq:db_adjoint}
				\db^*A=-g^{\bar{k}m}\d_m(g_{\bar{j}\ell}A^\ell_{\bar{k}})g^{\bar{j}i}\frac{\d}{\d z^i}.
			\end{equation} This appears similar to the operator \(P\), but it is not the quite same since the indices \(\bar{j}\) and \(\bar{k}\) are swapped. However, if \(A\) has the symmetry \[g_{\bar{j}\ell}A^\ell_{\bar{k}}=g_{\bar{k}\ell}A^\ell_{\bar{j}}\] then we will indeed have \(\db^*A=(PA)^\#\), where \(\#\) is conversion from a \((0,1)\)-form to a \((1,0)\)-vector field via the metric. This relation is satisfied in the important situation \(A=\Lich h\), in which case \(g_{\bar{j}\ell}A^\ell_{\bar{k}}=\d_{\bar{j}}\d_{\bar{k}}h\) at the centre of a normal coordinate system. For later use, we therefore record the following: \begin{lemma}\label{lem:Lich_product}
				For \(f,g\in C^\infty(M,\mathbb{C})\), \[\Lich^*(f\Lich g)=f\Lich^*\Lich g-2(\db^*\Lich g,\nabla^{1,0}f)+(\Lich f,\Lich g).\]
			\end{lemma}
			
			If \(\nabla^{1,0}f\) is a holomorphic vector field, we call \(f\) a \emph{holomorphy potential}. The Lichnerowicz operator is a self-adjoint elliptic operator whose kernel consists of the holomorphy potentials. Although we will not need this, we remark that a holomorphic vector field has a holomorphy potential precisely if the vector field vanishes somewhere \cite{LS94}, and so the set of holomorphic vector fields arising from holomorphy potentials is independent of the choice of K{\"a}hler metric.
			
			Let \(\mathcal{H}\) be the space of smooth K{\"a}hler potentials with respect to \(\omega\). For \(\varphi\in\mathcal{H}\), we write \(\omega_\varphi:=\omega+i\d\db\varphi\) for the corresponding K{\"a}hler metric. The scalar curvature determines an operator \(S:\mathcal{H}\to C^\infty(M,\mathbb{R})\), \(\varphi\mapsto S(\omega_\varphi)\). The linearisation \(L_\varphi\) of the scalar curvature operator at \(\varphi\in\mathcal{H}\) is given by \[L_\varphi\psi=\Lich_\varphi^*\Lich_\varphi\psi+\frac{1}{2}\nabla_\varphi S(\omega_\varphi)\cdot\nabla_\varphi\psi,\] for \(\psi\in C^\infty(M,\mathbb{R})=T_\varphi\mathcal{H}\), where \(\Lich_\varphi\) and \(\nabla_\varphi\) denote the operators defined by the K{\"a}hler metric \(\omega_\varphi\).

		\subsection{Weighted extremal metrics}
		
			In this section, we review the weighted cscK metrics introduced by Lahdili \cite{Lah19}.
		
			Take: \begin{enumerate}
				\item \((M,\omega)\) a compact K{\"a}hler manifold,
				\item \(T\) a real torus acting effectively on \((M,\omega)\) by hamiltonian isometries,
				\item \(\mu:M\to \mathfrak{t}^*\) a moment map for the \(T\)-action,
				\item \(P:=\mu(M)\subset\mathfrak{t}^*\) the moment polytope,
				\item \(v,w:M\to\mathbb{R}_{>0}\) positive smooth functions.
			\end{enumerate} Note \(\mu(M)\) is indeed a convex polytope by a theorem of Atiyah \cite{Ati82} and Guillemin--Sternberg \cite{GS82}. Our convention for moment maps is the following: given an element \(\xi\in\mathfrak{t}\), \[\langle d\mu,\xi\rangle = -\omega(\xi,-),\] where we abuse notation by conflating the element \(\xi\) of \(\mathfrak{t}\) with the real holomorphic vector field it generates on \(M\); here \(\langle-,-\rangle\) denotes the natural pairing between \(\mathfrak{t}^*\) and \(\mathfrak{t}\). \begin{definition}
				Given the above data, we define the \emph{\(v\)-weighted scalar curvature} of \(\omega\) to be \[S_v(\omega):=v(\mu)S(\omega)-2\Delta(v(\mu))+\frac{1}{2}\Tr(g\circ\Hess(v)(\mu)).\] Here  \(S(\omega)=\Lambda_\omega\Ric(\omega)\) is the scalar curvature, \(\Delta=-\d^*\d\) is the K{\"a}hler Laplacian of \(\omega\), and \(g\) is the Riemannian metric determined by \(\omega\).
			\end{definition} Concretely, the term \(\Tr(g\circ\Hess(v)(\mu))\) may be written \[\sum_{a,b}v_{,ab}(\mu)g(\xi_a,\xi_b),\] where \(\xi_1,\ldots,\xi_r\) is a basis of \(\mathfrak{t}\), and \(v_{,ab}\) denotes the \(ab\)-partial derivative of \(v\) with respect to the dual basis of \(\mathfrak{t}^*\).
		
			\begin{remark}
				While the definition of \(S_v(\omega)\) may seem arbitrary at first, the formula arises naturally as an infinite-dimensional moment map on the space \(\mathcal{J}^T\) of \(T\)-invariant almost complex structures compatible with \(\omega\), when one perturbs the metric on this space using the weight function \(v\) \cite[Section 4]{Lah19}. What makes this curvature worth studying is that it can further recover many well-known and important examples of canonical metrics in K{\"a}hler geometry---see Example \ref{ex:weighted_cscK_metrics} below.
			\end{remark}
		
			\begin{remark}
				In \cite{Lah19}, the \(v\)-weighted scalar curvature is instead written \[S_v(\omega)=v(\mu)S(\omega)+2\Delta(v(\mu))+\Tr(g\circ\Hess(v)(\mu)).\] Our weighted scalar curvature is equal to half of this, and the differences in signs and constants are due to the differences between the Riemannian and K{\"a}hler curvatures and Laplacians.
			\end{remark}
		
			\begin{definition}[{\cite{Lah19}}]\label{def:weighted_extremal}
				The metric \(\omega\) is: \begin{enumerate}
					\item a \emph{\((v,w)\)-weighted cscK metric}, if \[S_v(\omega)=c_{v,w}w(\mu),\] where \(c_{v,w}\) is a constant;
					\item a \emph{\((v,w)\)-weighted extremal metric} if the function \[S_{v,w}(\omega):=S_v(\omega)/w(\mu)\] is a holomorphy potential with respect to \(\omega\).
				\end{enumerate}  
			\end{definition}
		
			Sometimes we shorten the full name to just a \((v,w)\)-cscK metric, or a \((v,w)\)-extremal metric. If the weight functions \(v\) and \(w\) are understood or irrelevant, we may also refer to such a metric simply as a weighted cscK metric, or a weighted extremal metric.
		
			\begin{remark}\label{rem:weighted_extremal}
				Our definition of weighted extremal metric is slightly different to that in \cite{Lah19}. Namely, in \cite{Lah19} it is required that the function \(S_v(\omega)/w(\mu)\) is of the form \(w_{\text{ext}}(\mu)\), where \(w_{\text{ext}}:P\to\mathbb{R}\) is an affine linear function. Here we do not require that \(S_v(\omega)/w(\mu)\) can be described as such, but note that since this function is a \(T\)-invariant holomorphy potential, we can enlarge the torus to \(T'\supset T\) by taking the torus generated by a basis \(\xi_1,\ldots,\xi_r\) for \(\mathfrak{t}\) together with the holomorphic vector field \(\xi\) determined by \(S_v(\omega)/w(\mu)\). If \(P'\) is the moment polytope of \(T'\), \(S_v(\omega)/w(\mu)\) is then the composition of the affine linear function \(\langle\xi,-\rangle:P'\to\mathbb{R}\) with \(\mu_{\mathfrak{t}'}:M\to\mathfrak{t}'\). When the torus \(T\) is maximal, our two definitions therefore coincide.
			\end{remark}
		
			\begin{example}\label{ex:weighted_cscK_metrics}
				Fix an element \(\xi\in\mathfrak{t}\), and denote by \(\ell_\xi:\mathfrak{t}^*\to\mathbb{R}\) the corresponding element of \((\mathfrak{t}^*)^*\). Let \(a\) be a constant such that \(a+\ell_\xi>0\) on \(P\). Many standard canonical metrics can be obtained from certain choices of the functions \(v,w\) \cite[Section 3]{Lah19}: \begin{enumerate}
					\item {\bf CscK:} Taking \(v\) and \(w\) constant, the weighted cscK equation reduces to \[S(\omega)=c,\] which is the usual cscK equation.
					
					\item {\bf Extremal:} Taking \(v\) and \(w\) constant again, a weighted extremal metric is precisely an extremal metric in the usual sense, meaning \(\nabla^{1,0}S(\omega)\) is a holomorphic vector field.
					
					\item {\bf K{\"a}hler--Ricci soliton:} For \(M\) Fano and \(\omega\in c_1(X)\), the K{\"a}hler--Ricci soliton equation is \[\Ric(\omega)-\omega=-\frac{1}{2}\mathcal{L}_{J\xi}\omega,\] where \(\mathcal{L}_{J\xi}\) is the Lie derivative with respect to the real holomorphic vector field \(J\xi\). A weighted extremal metric in \(c_1(X)\) with weights \(v=w=e^{\ell_\xi}\) is a K{\"a}hler--Ricci soliton provided the extremal field is \(J\xi\). That is, the K{\"a}hler--Ricci soliton equation may be written \[\nabla S_{v,w}(\omega)=J\xi\] for the choice of weights \(v=w=e^{\ell_\xi}\). 
					
					The K{\"a}hler--Ricci soliton equation has an extensive literature. One important context in which they arise is as the Gromov--Hausdorff limits of solutions to the K{\"a}hler--Ricci flow, which is a powerful theorem of Chen--Wang \cite{CW20}.
					
					\item {\bf Extremal Sasaki:} Suppose that \([\omega]\) is the first Chern class \(c_1(L)\) of an ample line bundle \(L\to M\). A choice of K{\"a}hler metric \(\omega_\varphi\in[\omega]\) then corresponds to a Sasaki metric on the unit circle bundle \(S\) of \(L^*\). Letting \[v:=(a+\ell_\xi)^{-n-1},\quad w:=(a+\ell_\xi)^{-n-3},\] a \((v,w)\)-extremal metric on \(M\) then corresponds to an extremal Sasaki metric on \(S\) with Sasaki--Reeb field \(\xi\) \cite{AC21,ACL21}.
					
					\item {\bf Conformally K{\"a}hler--Einstein Maxwell:}  Letting \[v=(a+\ell_\xi)^{-2n+1},\quad w=(a+\ell_\xi)^{-2n-1},\] a \((v,w)\)-cscK metric on \(M\) then corresponds to a conformally K{\"a}hler Einstein--Maxwell metric introduced in \cite{LeB10, AM19}; see \cite{Lah20}.
					
					\item {\bf \(v\)-soliton:} Suppose \(M\) is Fano and \([\omega]=c_1(X)\). Taking \(v\) arbitrary and defining \[w(p)=2v(p)(n+\langle d\log v(p),p\rangle),\] the \((v,w)\)-cscK equation then becomes the \(v\)-soliton equation \[\Ric(\omega)-\omega=i\d\db\log v(\mu),\] introduced in \cite{Mab03} and studied in \cite{BWN14,HL20,AJL21}.
					
					\item {\bf \(\mu\)-cscK:} In \cite{Ino22,Ino20}, Inoue has introduced and studied a class of \emph{\(\mu\)-cscK metrics}. These are a special class of weighted extremal metrics, given by the same weight functions \(v=w=e^{\ell_\xi}\) and extremal field \(\xi\) as for K{\"a}hler--Ricci solitons, only one drops the condition of \(M\) being Fano \cite[Section 2.1.6]{Ino22}.
				\end{enumerate}
			\end{example}

			For an element \(\xi\in\mathfrak{t}\), we will write \(\mu^\xi:=\langle\mu,\xi\rangle=\ell_\xi\circ\mu\), where the pairing \(\langle-,-\rangle\) is the natural one on \(\mathfrak{t}^*\otimes\mathfrak{t}\). When we have chosen a basis \(\{\xi_a\}\) for \(\mathfrak{t}\), we will also write \(\mu^a\) in place of \(\mu^{\xi_a}\). The function \(\mu^\xi\) is then a hamiltonian for the infinitesimal action of \(\xi\) on \(M\). 
			
			We are interested in finding a weighted extremal metric in the class \([\omega]\). Given a \(T\)-invariant K{\"a}hler potential \(\varphi\in\mathcal{H}^T\), let \[\mu_\varphi:=\mu+d^c\varphi.\] That is, for any \(\xi\in\mathfrak{t}\), \[\mu_\varphi^\xi:=\mu^\xi+d^c\varphi(\xi),\] where we abuse notation by writing \(\xi\) for the vector field it generates on \(M\). Our convention is that \(d^c:=\frac{i}{2}(\db-\d)\), so that \(dd^c=i\d\db\).
		
			\begin{lemma}[{\cite[Lemma 1]{Lah19}}]\label{lem:moment_map_change}
				With the above definition, \(\mu_\varphi\) is a moment map for the \(T\)-action with respect to \(\omega_\varphi\), and \(\mu_\varphi(M)=P\), where \(P\) is the moment polytope for \(\mu=\mu_0\).
			\end{lemma}
		
			With this lemma in mind, it then makes sense to search for a \((v,w)\)-extremal metric in the class \([\omega]\), which is \(T\)-invariant by fiat.
		
			\begin{lemma}[{\cite[Lemma 2]{Lah19}}]\label{lem:const_cvw}
				With the above choice of moment map \(\mu_\varphi\), the following quantities are independent of the choice of \(\varphi\in\mathcal{H}^T\): \begin{enumerate}
					\item \(\int_Mv(\mu_\varphi)\,\omega_\varphi^n\),
					\item \(\int_Mv(\mu_\varphi)\,\Ric(\omega_\varphi)\wedge\omega_\varphi^{n-1}+\int_M\langle dv(\mu_\varphi),-\Delta_{\varphi}\mu_\varphi\rangle\,\omega_\varphi^n\),
					\item \(\int_M S_v(\omega_\varphi)\,\omega_\varphi^n\).
				\end{enumerate} It follows that the constant \(c_{v,w}\) of Definition \ref{def:weighted_extremal} is fixed, given by \[c_{v,w}=\frac{\int_MS_v(\omega)\,\omega^n}{\int_Mw(\mu)\,\omega^n}.\]
			\end{lemma} 
		
			\begin{remark}\label{rem:Ricci_moment_map}
				The significance of \(-\Delta_\varphi\mu_\varphi\) in \emph{(2)} of Lemma \ref{lem:const_cvw} is that it is a moment map for the Ricci curvature \(\Ric(\omega_\varphi)\), see \cite[Lemma 5]{Lah19}. That is, for any \(\xi\in\mathfrak{t}\) and \(x\in M\), \[\langle d(-\Delta_\varphi\mu_\varphi)(x),\xi\rangle = \Ric(\omega_\varphi)(x)(-,\xi_x).\]
			\end{remark}
		
			We will also need to understand well the linearisation of the weighted scalar curvature operator. Recall for a K{\"a}hler metric \(\omega\), the linearisation of the usual scalar curvature operator \(S:\mathcal{H}\to C^\infty(M,\mathbb{R})\) at \(\varphi\in\mathcal{H}\) is \begin{equation}\label{eq:unweighted_linearisation}
			L_\varphi\psi=\Lich_\varphi^*\Lich_\varphi\psi+\frac{1}{2}\nabla_\varphi S(\omega_\varphi)\cdot\nabla_\varphi\psi.
			\end{equation}
			
			In the weighted setting, a very similar formula holds:
			
			\begin{proposition}[{\cite[Lemma B.1]{Lah19}}]
				The linearisation of the weighted scalar curvature operator \(S_{v,w}:\mathcal{H}^T\to C^\infty(M,\mathbb{R})^T\) at \(\varphi\in\mathcal{H}^T\) is given by \[\check{L}_{\varphi}(\psi)=\frac{v(\mu_\varphi)}{w(\mu_\varphi)}\Lich^*_{v,\varphi}\Lich_\varphi\psi+\frac{1}{2}\nabla_\varphi S_{v,w}(\omega_\varphi)\cdot\nabla_\varphi\psi\] for \(\psi\in C^\infty(M,\mathbb{R})^T=T_\varphi\mathcal{H}^T\). Here \[\Lich^*_{v,\varphi}A:=\frac{1}{v(\mu_\varphi)}\Lich_\varphi(v(\mu_\varphi)A)\] for \(A\in\Gamma(T^{1,0}\otimes\Omega^{0,1})^T\).
			\end{proposition}
		
			We now show how to rewrite \(\check{L}_\varphi\) in terms of \(L_\varphi\); for simplicity of notation we will drop the omnipresent subscript \(\varphi\). First, for any metric \(\omega\) we will write \[S_{v,w}(\omega)=\frac{v(\mu)}{w(\mu)}S(\omega)+\Phi_{v,w}(\omega),\] where \begin{equation}\label{eq:Phi}
				\Phi_{v,w}(\omega):=-\frac{2}{w(\mu)}\Delta(v(\mu))+\frac{1}{2w(\mu)}\Tr(g\circ\Hess(v)(\mu)).
			\end{equation} It follows that \[\nabla S_{v,w}(\omega)\cdot\nabla\psi=\frac{v(\mu)}{w(\mu)}\nabla S(\omega)\cdot\nabla\psi+S(\omega)\nabla\left(\frac{v(\mu)}{w(\mu)}\right)\cdot\nabla\psi+\nabla\Phi_{v,w}(\omega)\cdot\nabla\psi.\] Applying Lemma \ref{lem:Lich_product},\begin{align*}
				v(\mu)\Lich^*_v\Lich\psi &= \Lich^*(v(\mu)\Lich\psi) \\
				&= v(\mu)\Lich^*\Lich\psi-2(\db^*\Lich\psi,\nabla^{1,0}(v(\mu)))+(\Lich\psi,\Lich(v(\mu))).
			\end{align*} Putting this all together:
			
			\begin{lemma}\label{lem:weighted_linearisation}
				The linearisation \(\check{L}\) of the weighted scalar curvature operator \(S_{v,w}\) can be written \begin{align*}
					\check{L}(\psi) &= \frac{v(\mu)}{w(\mu)}L(\psi)-\frac{2}{w(\mu)}(\db^*\Lich\psi,\nabla^{1,0}(v(\mu)))+\frac{1}{w(\mu)}(\Lich\psi,\Lich(v(\mu))) \\
					&+ \frac{1}{2}S(\omega)\nabla\left(\frac{v(\mu)}{w(\mu)}\right)\cdot\nabla\psi+\frac{1}{2}\nabla\Phi_{v,w}(\omega)\cdot\nabla\psi,
				\end{align*} where \(L\) is the linearisation \eqref{eq:unweighted_linearisation} of the usual scalar curvature operator \(S\).
			\end{lemma}
		
			It will also be important to understand the extra term \(\Phi_{v,w}(\omega)\). For this, let us pick a normal Riemannian coordinate system \(x_1,\ldots,x_{2n}\) and compute at the centre: \begin{align*}
				&2\Delta(v(\mu)) \\
				= &\sum_k\frac{\d^2}{\d x_k^2}(v(\mu))\\
				=& \sum_k\frac{\d}{\d x_k}\left(\sum_a v_{,a}(\mu)\frac{\d\mu^a}{\d x_k}\right) \\
				=&\sum_{k}\sum_a v_{,a}(\mu)\frac{\d^2\mu^a}{\d x_k^2}+\sum_{k}\sum_{a,b}v_{,ab}(\mu)\frac{\d\mu^a}{\d x_k}\frac{\d\mu^b}{\d x_k} \\
				=&2\sum_a v_{,a}(\mu)\Delta\mu^a+\sum_{a,b}v_{,ab}(\mu)g(\nabla \mu^a,\nabla\mu^b) \\
				=&2\sum_a v_{,a}(\mu)\Delta\mu^a+\sum_{a,b}v_{,ab}(\mu)g(\xi_a,\xi_b).
			\end{align*} Of course, the last term here is just \(\Tr(g\circ\Hess(v)(\mu))\). From this calculation and \eqref{eq:Phi} we conclude: \begin{lemma}\label{lem:Phi}
			The term \(\Phi_{v,w}(\omega)\) is a linear combination of functions of the form \(u_a(\mu)\Delta\mu^a\) and \(u_{ab}(\mu)g(\xi_a,\xi_b)\), where the \(u_a\) and \(u_{ab}\) are among finitely many fixed smooth functions on the moment polytope \(P\) depending only on \(v\), \(w\) and the basis \(\{\xi_a\}\).
		\end{lemma}

	\section{Setting up the problem}
			
			Now that we have reviewed the relevant background material, we can proceed with setting up the proof of Theorem \ref{thm:main}. Structurally this will largely follow \cite{Sze12}, although the technicalities differ as there are many new terms that arise in the weighted setting. 
			
			Let \((M,\omega)\) be a weighted extremal manifold, and let \(p\in M\) be a fixed point of the \(T\)-action and the extremal field. We wish to show that, under a certain stability condition on \(p\), the blowup \(\Bl_pM\) admits a weighted extremal metric in the class \([\pi^*\omega-\epsilon^2E]\) for all \(\epsilon>0\) sufficiently small, where \(\pi:\Bl_pM\to M\) is the blowup map and \(E\) is the exceptional divisor of the blowup. 
			
			We begin by defining an approximate solution \(\omega_\epsilon\) to the weighted extremal equation on \(\Bl_pM\). This approximate solution is constructed by gluing the given weighted extremal metric \(\omega\) on \(M\) to a suitable rescaling of a model metric \(\eta\) on \(\Bl_0\mathbb{C}^n\) over a small neighbourhood of \(E\). This model metric \(\eta\) is called the \emph{Burns--Simanca metric}, and we cover its properties in Section \ref{sec:Burns-Simanca}. The gluing construction is then described in Section \ref{sec:approx_soln}.
			
			Given the approximate solution \(\omega_\epsilon\), we then seek to deform it to a metric \(\widetilde{\omega}_\epsilon\) that solves the weighted extremal equation up to a finite dimensional obstruction. A result of Sz{\'e}kelyhidi shows that if \(p\) is relatively stable then this obstruction can be overcome, and hence a weighted extremal metric on \(\Bl_pM\) exists. We set up the deformation problem in Section \ref{sec:deformation}.

			The main technical tool used in deforming \(\omega_\epsilon\) to \(\widetilde{\omega}_\epsilon\) is a family of weighted H{\"o}lder norms on \(\Bl_pM\), depending on \(\epsilon\). These are introduced in Section \ref{sec:weighted_norms}, where we cover some of their basic properties.

		\subsection{Burns--Simanca metric}\label{sec:Burns-Simanca}
		
			In this section we describe the Burns--Simanca metric \(\eta\) on \(\Bl_0\mathbb{C}^n\), which is a scalar-flat and asymptotically Euclidean K{\"a}hler metric. On \(\Bl_0\mathbb{C}^2\), it can be written explicitly:  \begin{equation}\label{eq:BS_metric_dim2}
				\eta:=i\d\db (|\zeta|^2+\log|\zeta|),
			\end{equation} where \(\zeta=(\zeta_1,\zeta_2)\) are the standard coordinates on \(\mathbb{C}^2\backslash\{0\}\cong\Bl_0\mathbb{C}^2\backslash\mathbb{P}^1\). This metric was first shown to be scalar-flat by Burns (see \cite[p. 594]{LeB88} and \cite[Remark 1]{Sim91}). 
			
			On \(\Bl_0\mathbb{C}^n\) for \(n>2\), the metric \(\eta\) was constructed by Simanca \cite{Sim91}. In this case there is no explicit formula available, however there is an asymptotic expansion of the metric. To describe this, first on \(\Bl_0\mathbb{C}^n\backslash E\) write \(\zeta=(\zeta_1,\ldots,\zeta_n)\) for the standard complex coordinates pulled back from \(\mathbb{C}^n\backslash\{0\}\). The metric \(\eta\) satisfies \begin{equation}\label{eq:BS_metric}
				\eta=i\d\db(|\zeta|^{2}+g(\zeta))
			\end{equation} on \(\Bl_0\mathbb{C}^n\backslash E\), where \[g(\zeta)=-|\zeta|^{4-2n}+\O(|\zeta|^{3-2n})\] as \(|\zeta|\to\infty\). Here a smooth real-valued function \(h\) is declared to be \(\O(|\zeta|^\ell)\) if it lies in the weighted H{\"o}lder space \(C^{k,\alpha}_{\ell}(\Bl_0\mathbb{C}^n)\) for all \(k\) and \(\alpha\in(0,1)\). We will go over the weighted H{\"o}lder norms in Section \ref{sec:weighted_norms}, but for now an equivalent definition is that for any multi-index \(I=(i_1,\ldots,i_n,j_1,\ldots,j_n)\), there exists a constant \(C_I>0\) such that \[\d_Ih:=\frac{\d^{|I|}h}{\d x_1^{i_1}\cdots\d x_n^{i_n}\d y_1^{j_1}\cdots\d y_n^{j_n}}\] satisfies \(|\d_I h|\leq C_I|\zeta|^{\ell-|I|}\) for all \(|\zeta|\gg0\), where \(\zeta_k=x_k+iy_k\) and \(|I|:=i_1+\cdots +i_n+j_1+\cdots+j_n\).
			
			In our situation, we will have a compact torus \(T\) acting on \(\Bl_0\mathbb{C}^n\), lifting a linear \(T\)-action on \(\mathbb{C}^n\). Changing the basis of \(\mathbb{C}^n\), this action can be assumed to be diagonal, in which case it is straightforward from \eqref{eq:BS_metric_dim2} and the construction in \cite{Sim91} that \(\eta\) is \(T\)-invariant. We claim that there exists a moment map \(\mu_\eta:\Bl_0\mathbb{C}^n\to\mathfrak{t}^*\) for the  \(T\)-action. To see this, note from \cite[Proposition 1]{Sim91} that on \(\Bl_0\mathbb{C}^n\backslash E\), \(\eta\) can be written \(i\d\db(s(u)+\log(u))\), where \(u=|\zeta|^2\), and \(s:[0,\infty)\to\mathbb{R}\) is a smooth function with \(s'(0)>0\). The component \(i\d\db\log u\) simply is the pullback of the Fubini--Study metric on \(\mathbb{P}^{n-1}\) to \(\mathcal{O}_{\mathbb{P}^{n-1}}(-1)\cong \Bl_0\mathbb{C}^n\), for which we already have a moment map. So it suffices to construct a moment map for the remaining term \(i\d\db s(u)\), but such a moment map is given by \(d^c(s(u))\), since \(s(u)\) is \(T\)-invariant.

			By \eqref{eq:BS_metric} and Lemma \ref{lem:moment_map_change}, on \(\Bl_0\mathbb{C}^n\backslash E\) a moment map for the \(T\)-action is \[\mu_\eta^\xi=\sum_{j=1}^nA^\xi_{j}|\zeta_j|^2+d^cg(\xi),\] where \(\xi\in\mathfrak{t}\) and \(\mathrm{diag}(A_1^\xi,\ldots,A_n^\xi)\in GL(\mathbb{C}^n)\) is the infinitesimal generator of the action of \(\xi\) on \(\mathbb{C}^n\).\footnote{Recall we use the shorthand \(\mu^\xi:=\langle\mu,\xi\rangle\) for moment maps, and conflate \(\xi\in\mathfrak{t}\) with the vector field it generates.} In the case \(n=2\) we take \(g:=\log|\zeta|\) and the same formula holds. By uniqueness of moment maps up to addition of constants, this formula extends over the exceptional divisor \(E\) to a well-defined moment map on \(\Bl_0\mathbb{C}^n\). We note the first term of \(\mu_\eta^\xi\) is \(\O(|\zeta|^2)\), and claim the remaining term is \(\O(|\zeta|^{4-2n})\) when \(n>2\). To see this, note that \(g\) is \(\O(|\zeta|^{4-2n})\), and so \(d^cg\) is \(\O(|\zeta|^{3-2n})\). But \(\xi\) is \(\O(|\zeta|)\), so \(d^cg(\xi)\) is \(\O(|\zeta|^{4-2n})\) by the Leibniz rule. In the case \(n=2\), \(d^cg\) is \(\O(|\zeta|^{-1})\) and \(d^cg(\xi)\) is \(\O(1)\). We record this for future use:
			
			\begin{lemma}\label{lem:BS_moment_map}
				There exists a moment map \(\mu_\eta\) for the \(T\)-action on \((\Bl_0\mathbb{C}^n,\eta)\), which on \(\Bl_0\mathbb{C}^n\backslash E\) can be written \[\mu_\eta^\xi=\sum_{j=1}^nA^\xi_{j}|\zeta_j|^2+d^cg(\xi),\] where \(g\) is defined by \[\eta=i\d\db(|\zeta|^2+g(\zeta))\] and satisfies \(g=\O(|\zeta|^{4-2n})\) when \(n>2\) and \(g=\log|\zeta|\) when \(n=2\). Furthermore, \(d^cg(\xi)\) is \(\O(|\zeta|^{4-2n})\) for \(n\geq2\).
			\end{lemma}

		\subsection{The approximate solution}\label{sec:approx_soln}
		
			Let \((M,\omega)\) be a weighted extremal manifold, and let \(p\in M\) be a common fixed point for both the \(T\)-action and the extremal field. Let \((z_1,\ldots,z_n)\) be a system of normal coordinates centred at \(p\), with respect to which the action of \(T\) is linear and diagonal. Such coordinates exist by the Bochner linearisation theorem, which says we may choose holomorphic coordinates about \(p\) with respect to which the \(T\)-action is linear. The \(T\)-action is unitary on \(T_pM\), so we may further find a linear change of coordinates such that the \(T\)-action is diagonal in these coordinates, and such that the induced linear coordinates on \(T_pM\) are orthonormal. Lastly, taking a Taylor expansion of the metric \(\omega\) about \(p\), one sees that the linear terms in \(\omega\) can only be non-zero in the directions where the torus acts trivially. Performing a quadratic change of coordinates as in \cite[Proposition 1.14]{Sze14}, we produce normal coordinates for which the \(T\)-action is diagonal.
			
			Without loss of generality, we can assume the coordinates \(z_1,\ldots,z_n\) are well-defined for \(|z|<2\). For \(r<2\), let \[B_r:=\{z\in M:|z|<r\}\] be the ball of radius \(r\) centred at \(p\), and define \[\widetilde{B}_r:=\pi^{-1}(B_r)\subset\Bl_pM.\] Similarly, for the standard complex coordinates \(\zeta\) on \(\mathbb{C}^n\) and \(R>0\), we will write \[D_R:=\{\zeta\in\mathbb{C}^n:|\zeta|<R\},\] and define \[\widetilde{D}_R:=\pi^{-1}(D_R)\subset \Bl_0\mathbb{C}^n.\]
			
			For \(\epsilon>0\) sufficiently small, let \begin{equation}\label{eq:r_epsilon}
			r_\epsilon:=\epsilon^{\frac{2n-1}{2n+1}},\quad R_\epsilon:=\epsilon^{-1}r_\epsilon.
			\end{equation} We then have \(r_\epsilon\to0\) as \(\epsilon\to0\), and \(R_\epsilon\to\infty\) as \(\epsilon\to0\). We will identify \(\widetilde{B}_{r_\epsilon}\subset \Bl_pM\) with the subset \(\widetilde{D}_{R_\epsilon}\subset\Bl_0\mathbb{C}^n\) via \(\iota_\epsilon:\widetilde{B}_{r_\epsilon}\to\widetilde{D}_{R_\epsilon}\), which is the unique lift of the map \(\iota_\epsilon:B_{r_\epsilon}\to D_{R_\epsilon}\), \(\iota_\epsilon(z):=\epsilon^{-1}z\).
			
			Let \(\rho:\mathbb{R}_{\geq0}\to[0,1]\) be a smooth function such that \(\rho(x)=0\) for \(x<1\) and \(\rho(x)=1\) for \(x>2\). Given \(\epsilon>0\) sufficiently small, on \(\widetilde{B}_1\backslash \widetilde{B}_\epsilon\) we define \begin{equation}\label{eq:gammas}
			\gamma_{1,\epsilon}(z):=\rho(|z|/r_\epsilon),\quad\quad\quad\gamma_{2,\epsilon}:=1-\gamma_{1,\epsilon},
			\end{equation} where \(r_\epsilon\) is defined in \eqref{eq:r_epsilon} and the coordinates \(z\) are pulled back from \(M\). We extend the \(\gamma_{j,\epsilon}\) to smooth functions on all of \(\Bl_pM\) by taking \(\gamma_{1,\epsilon}|_{\widetilde{B}_\epsilon}:=0\), \(\gamma_{1,\epsilon}|_{\Bl_pM\backslash \widetilde{B}_1}:=1\), and \(\gamma_{2,\epsilon}:=1-\gamma_{1,\epsilon}\). 
			
			The metric \(\omega\) has an expansion \begin{equation}\label{eq:omega_about_p}
				\omega=i\d\db(|z|^2+f(z))
			\end{equation} about \(p\), where \(f(z)\) is \(\O(|z|^4)\), by definition of normal coordinates. Note that \[\hat{f}(z):= \int_{T}f(tz)dt\] is \(T\)-invariant, \(\O(|z|^4)\), and also satisfies \eqref{eq:omega_about_p} since all other terms and operators in \eqref{eq:omega_about_p} are \(T\)-invariant or equivariant, so we will assume henceforth that \(f\) in \eqref{eq:omega_about_p} is \(T\)-invariant. Since \(\pi:\widetilde{B}_1\backslash E\to B_1\backslash\{p\}\) is a biholomorphism, the coordinates \(z_1,\ldots,z_n\) on \(B_1\backslash\{p\}\) lift to coordinates on \(\widetilde{B}_1\backslash E\), which we denote by the same symbols. We define the approximate solution \(\omega_\epsilon\) on three separate regions as follows: \begin{enumerate}
				\item On \(\Bl_pM\backslash\widetilde{B}_1\), \[\omega_\epsilon:=\pi^*\omega.\]
				\item On \(\widetilde{B}_1\backslash\widetilde{B}_\epsilon\), \[\omega_\epsilon:=i\d\db(|z|^2+\gamma_{1,\epsilon}(z)f(z)+\gamma_{2,\epsilon}(z)\epsilon^2g(\epsilon^{-1}z)),\] where \(f\) is defined in \eqref{eq:omega_about_p} and \(g\) is defined in \eqref{eq:BS_metric}.
				\item  On \(\widetilde{B}_\epsilon\), \[\omega_\epsilon:=\iota_\epsilon^*(\epsilon^2\eta),\] where \(\iota_\epsilon:\widetilde{B}_{\epsilon}\to\widetilde{D}_1\) is the biholomorphism lifting the map \(B_\epsilon\to D_1\), \(z\mapsto \epsilon^{-1}z\).
			\end{enumerate} It is easy to see these constructions give a well-defined real closed \((1,1)\)-form on \(\Bl_pM\). Furthermore, for \(\epsilon>0\) sufficiently small, the growth conditions on \(f\) and \(g\) imply \(\omega_\epsilon\) is a K{\"a}hler metric. Lastly, \(\omega_\epsilon\) is equal to \(\omega\) outside \(\widetilde{B}_{2r_\epsilon}\), and equal to \(\iota_\epsilon^*(\epsilon^2\eta)\) on \(\widetilde{B}_{r_\epsilon}\).
		
			We now describe an explicit moment map \(\mu_\epsilon:\Bl_pM\to\mathfrak{t}^*\) for \(\omega_\epsilon\). \begin{lemma}\label{lem:epsilon_moment_map}
				There exists a moment map \(\mu_\epsilon:\Bl_pM\to\mathfrak{t}^*\) for \(\omega_\epsilon\), satisfying: \begin{enumerate}
					\item On \(\Bl_pM\backslash\widetilde{B}_{2r_\epsilon}\), \[\mu_\epsilon=\pi^*\mu,\]
					\item On \(\widetilde{B}_1\backslash\widetilde{B}_\epsilon\),  \[\mu_\epsilon=\mu+d^c(\gamma_{2,\epsilon}(z)(\epsilon^2g(\epsilon^{-1}z)-f(z))),\] 
					\item On \(\widetilde{B}_{r_\epsilon}\backslash E\), \[\mu_\epsilon=\mu(p)+\epsilon^2\sum_{j=1}^nA_{j}|\epsilon^{-1}z_j|^2+d^c(\epsilon^2g(\epsilon^{-1}z)),\] 
				\end{enumerate} where \(\mu:M\to\mathfrak{t}^*\) is the moment map for \(\omega\), the functions \(f\) and \(g\) are defined in \eqref{eq:omega_about_p} and \eqref{eq:BS_metric} respectively, and \(A_{j}\in\mathfrak{t}^*\) is the diagonal matrix representation of the \(\mathfrak{t}\)-action on \(T_pM\cong\mathbb{C}^n\). For \(\epsilon>0\) sufficiently small, the image of \(\mu_\epsilon\) is a convex polytope \(P_\epsilon\) contained in the moment polytope \(P:=\mu(M)\).
			\end{lemma}	\begin{proof} On \(\Bl_pM\backslash\widetilde{B}_1\), we take \(\mu_\epsilon:=\pi^*\mu\), where \(\mu\) is the moment map for \(\omega\). This fixes a normalisation for the moment map. On \(\widetilde{B}_1\backslash\widetilde{B}_\epsilon\), \begin{align*}
				\omega_\epsilon-\omega &= 	i\d\db(\gamma_{1,\epsilon}(z)f(z)+\gamma_{2,\epsilon}(z)\epsilon^2g(\epsilon^{-1}z)-f(z))\\
				&= i\d\db(\gamma_{2,\epsilon}(z)(\epsilon^2g(\epsilon^{-1}z)-f(z))).
			\end{align*} Hence by Lemma \ref{lem:moment_map_change}, the moment map on \(\widetilde{B}_1\backslash\widetilde{B}_\epsilon\) is \[\mu_\epsilon:=\mu+d^c(\gamma_{2,\epsilon}(z)(\epsilon^2g(\epsilon^{-1}z)-f(z)));\] note this agrees with the chosen normalisation on \(\Bl_p M\backslash\widetilde{B}_1\), since \(\gamma_{2,\epsilon}=0\) outside of \(\widetilde{B}_{2r_\epsilon}\).  On the region \(\widetilde{B}_{r_\epsilon}\backslash\widetilde{B}_\epsilon\), we have \(\gamma_{2,\epsilon}=1\) and so \[\mu_\epsilon=(\mu-d^cf)+d^c(\epsilon^2g(\epsilon^{-1}z)).\] Note by \eqref{eq:omega_about_p} and Lemma \ref{lem:moment_map_change}, the first term \(\mu-d^c f\) is the Euclidean moment map on \(B_1\) with normalisation \(\mu(p)-d^cf(p)=\mu(p)\). It follows that the moment map on \(\widetilde{B}_\epsilon\backslash E\) is \[\mu_\epsilon=\mu(p)+\epsilon^2\sum_{j=1}^nA_{j}|\epsilon^{-1}z_j|^2+d^c(\epsilon^2g(\epsilon^{-1}z)),\] where \(A_{j}\in\mathfrak{t}^*\) is the diagonal matrix representation of the \(\mathfrak{t}\)-action on \(T_pM\cong\mathbb{C}^n\). We know from Lemma \ref{lem:BS_moment_map} that this formula extends smoothly to a moment map on \(\widetilde{B}_\epsilon\), hence we have a moment map \(\mu_\epsilon\) for \(\omega_\epsilon\). 
			
			It remains to show that the image \(P_\epsilon:=\mu_\epsilon(\Bl_pM)\) is a subset of \(P:=\mu(M)\); by the Atiyah--Guillemin--Sternberg theorem  \cite{Ati82,GS82}, both \(P_\epsilon\) and \(P\) are convex polytopes. By observing that \(\mu_\epsilon\) is a small modification of \(\pi^*\mu\) near the exceptional locus,  we will show that convexity itself implies \(P_\epsilon\subset P\).
			
			To this end, we first show that for any open neighbourhood \(U\) of \(\mu(p)\) in \(\mathfrak{t}^*\), there exists \(\epsilon_0>0\) such that for all \(\epsilon>0\) satisfying \(\epsilon <\epsilon_0\), we have \(\mu_\epsilon(\tilde{B}_{2r_\epsilon})\subset U\). To see this, first note by \cite[p. 158]{Sze12}, on \(\widetilde{B}_{2r_\epsilon}\backslash\widetilde{B}_{r_\epsilon}\), \[\gamma_{1,\epsilon}(z)f(z)+\gamma_{2,\epsilon}(z)\epsilon^2g(\epsilon^{-1}z)=\O(|z|^4),\] and so \begin{align*}
				\mu_\epsilon - \mu &= d^c(\gamma_{1,\epsilon}(z)f(z) + \gamma_{2,\epsilon}(z)\epsilon^2g(\epsilon^{-1}z) - f(z)) \\
				&= \O(|z|^4),
			\end{align*} noting that although \(d^c\) of something \(\O(|z|^4)\) is \(\O(|z|^3)\), we implicitly evaluate at holomorphic vector fields that vanish at \(p\), which are \(\O(|z|)\). Since \(z\) here is constrained to \(\widetilde{B}_{2r_\epsilon}\), the above term is then \(\O(r_\epsilon^4)\) as \(\epsilon\to0\), and so tends to \(0\) as \(\epsilon\to0\). Similarly by the estimates on \(g\) in Lemma \ref{lem:BS_moment_map}, \(\mu_\epsilon-\mu(p)\) is \(\O(r_\epsilon^2 + \epsilon R_\epsilon^{4-2n})\) on \(\widetilde{B}_{r_\epsilon}\), and also tends to \(0\) as \(\epsilon\to0\).  These estimates on \(\mu_\epsilon\) in terms of \(\mu\) and \(\mu(p)\) imply the claim in the first sentence of this paragraph.
			
			Since \(\mu(M\backslash B_{2r_\epsilon}) = \mu_\epsilon(\Bl_p M\backslash \widetilde{B}_{2r_\epsilon})\), we need only consider the shape of the moment polytopes \(P\) and \(P_\epsilon\) near \(\mu(p)\). If \(\mu(p)\) is an interior point of \(P\) then we take \(U\) above to be contained in \(P\) and we are done. Otherwise, \(\mu(p)\) lies on the boundary of \(P\); let us write \(P=H_1\cap\cdots\cap H_k\) as the intersection of closed half-spaces in \(\mathfrak{t}^*\), where \(k\) is minimal so each \(H_j\) corresponds to a face of \(P\). Then \(\mu(p)\in\d H_{j_1}\cap\cdots\cap \d H_{j_\ell}\) for some \(1\leq j_1<\cdots<j_\ell\leq k\) where \(\ell\geq 1\) is maximal. We choose an open neighbourhood \(U\) of \(\mu(p)\) sufficiently small so that: (a) every vertex (except for \(\mu(p)\) if \(\mu(p)\) is a vertex) has an open neighbourhood disjoint with \(U\), and (b) \(P\cap U = H_{j_1}\cap\cdots\cap H_{j_\ell}\cap U\). We then take \(\epsilon_0\) so that for \(0<\epsilon<\epsilon_0\), \(\mu_\epsilon(\widetilde{B}_{2r_\epsilon})\subset U\). Now suppose that for some \(0<\epsilon<\epsilon_0\), there exists \(\xi\in P_\epsilon\) such that \(\xi\notin U\backslash P\). Then \(\xi\notin H_{j_1}\), without loss of generality. Let \(\xi'\in P\) be a vertex of \(P\) distinct from \(\mu(p)\) such that the straight line segment from \(\mu(p)\) to \(\xi'\) lies on \(\d P\cap \d H_{j_1}\). By our choice of \(U\), \(\xi'\) has an open neighbourhood \(V\) such that \(V\cap P = V\cap P_\epsilon\). Denote by \(L\) the half-open line segment from \(\xi\) to \(\xi'\), not including \(\xi'\). Since \(P_\epsilon\) is convex, \(L\) must lie within \(P_\epsilon\). However, note that no point of \(L\) is within \(H_{j_1}\), and hence no point of \(L\) lies within \(P\). Since \(P\cap V = P_\epsilon\cap V\), there exist points in \(L\) that do not lie in \(P_\epsilon\); a contradiction.
			\end{proof}

			\begin{remark}
				The inclusion \(P_\epsilon\subset P\) allows us to restrict \(v\) and \(w\) to \(P_\epsilon\), so the weighted scalar curvature \(S_{v,w}(\omega_\epsilon)\) is well-defined and it makes sense to search for a \((v,w)\)-extremal metric on \(\Bl_pM\). We remark that the inclusion \(P_\epsilon\subset P\) may not be strict, in particular if \(\mu(p)\) lies in the interior of \(P\) then we will have \(P_\epsilon=P\) for all \(\epsilon\) (I thank Eveline Legendre for pointing out this possibility).
			\end{remark}

		\subsection{The deformation problem}\label{sec:deformation}
		
			Let \((M,\omega)\) be a weighted extremal manifold. Then \(X:=\nabla S_{v,w}(\omega)\) is a \(T\)-invariant real holomorphic vector field and \(JX\) preserves \(\omega\), where \(J\) is the integrable almost complex structure of \(X\). Let \(p\) be a fixed point of both \(T\) and the extremal field \(X\). We make the following definitions: \begin{enumerate}
			\item \(G\) is the group of \(T\)-commuting hamiltonian isometries of \((M,\omega)\),
			\item \(G_p\) is the subgroup of \(G\) fixing \(p\),
			\item \(T'\subset G_p\) is a maximal torus, and
			\item \(H\subset G\) is the subgroup of automorphisms commuting with \(T'\).
			\end{enumerate} We write the Lie algebras of \(T'\) and \(H\) as \(\mathfrak{t}'\) and \(\mathfrak{h}\) respectively, and note the inclusions \(\mathfrak{t}\subset\mathfrak{t}'\subset\mathfrak{h}\). If \(T\) was maximal in the hamiltonian isometry group to begin with, these inclusions are all equalities.
			
			\begin{remark}
			In the previous section, we constructed a \(T\)-invariant metric \(\omega_\epsilon\) using \(T\)-invariant coordinates \(z\) near the fixed point \(p\). Since \(T'\) also acts by hamiltonian isometries and fixes the point \(p\), we can assume that the \(z\)-coordinates are in fact \(T'\)-invariant, and thus all the constructions from the previous section, including \(\omega_\epsilon\), are \(T'\)-invariant as well.
			
			In addition, we claim that if \(\varphi\) is a \(T'\)-invariant K{\"a}hler potential, then the weighted scalar curvature \(S_{v,w}(\omega_\epsilon+i\d\db\varphi)\) is also \(T'\)-invariant. To see this, by the chain rule it suffices to show the moment map \(\mu_\epsilon\) for \(T\) is \(T'\)-invariant. Let \(\mu_\epsilon^a\) be a component function of \(\mu_\epsilon\) generating \(\xi_a\in\mathfrak{t}\), and let \(\xi\in\mathfrak{t}'\). We write \(\widetilde{\xi}_a\) and \(\widetilde{\xi}\) for the corresponding vector fields on \(\Bl_pM\). Then \(\widetilde{\xi}(\mu_\epsilon^a)\) is the hamiltonian generator of \([\widetilde{\xi},\widetilde{\xi}_a]=0\), hence this function is constant. It vanishes at a maximum of \(\mu_\epsilon^a\), and is therefore identically zero.
			\end{remark}
			
			We may assume that the moment maps \(\mu\) for \(T\) and \(\mu_H\) for \(H\) satisfy \[\overline{\mu}:=\int_M\mu\,w(\mu)\omega^n=0\in\mathfrak{t}^*,\quad\overline{\mu}_H:=\int_M\mu_H\,w(\mu)\omega^n=0\in\mathfrak{h}^*.\] This is achieved by replacing \(\mu\) with \(\mu-\overline{\mu}\) and \(\mu_H\) with \(\mu_H-\overline{\mu}_H\); this clearly preserves the moment map equation, and to see equivariance is maintained for \(\mu_H\), note \begingroup
			\allowdisplaybreaks \begin{align*}
			\mathrm{ad}(\xi)^*(\overline{\mu}_H) &= \int_M\mathrm{ad}(\xi)^*(\mu_H)w(\mu)\omega^n \\
			&= \int_M\mathcal{L}_\xi(\mu_H)w(\mu)\omega^n \\
			&= \int_M\mathcal{L}_\xi(\mu_H w(\mu)\omega^n) \\
			&= 0 \\
			&= \mathcal{L}_\xi(\overline{\mu}_H).
			\end{align*} \endgroup Here \(\mathrm{ad}(\xi)^*\) denotes the coadjoint action of an element \(\xi\in\mathfrak{h}\) on \(\mathfrak{h}^*\); in the second line we used equivariance of \(\mu_H\), and in the third we used \(\mathcal{L}_\xi\omega=0\) as well as \(\xi(\mu)=0\), which follows since \(\mathfrak{t}\) is central in \(\mathfrak{h}\). Note this adjustment of moment maps shifts the moment polytope \(P\), however we can also translate the weight functions \(v,w\) by \(\overline{\mu}\) so that they are well defined on this shift, and this preserves the weighted extremal condition. 
			
			On the compact Lie algebra \(\mathfrak{h}\), we now fix the \(H\)-invariant inner product \begin{equation}\label{eq:inner_product}
			\langle\xi,\xi'\rangle_{\mathfrak{h}}:=\int_M\langle\mu_H,\xi\rangle\langle\mu_H,\xi'\rangle w(\mu)\omega^n,
			\end{equation} where the pairing \(\langle-,-\rangle\) inside the integral is the natural dual pairing between \(\mathfrak{h}^*\) and \(\mathfrak{h}\). Via this inner product, we identify the Lie algebra and its dual \(\mathfrak{h}\cong\mathfrak{h}^*\). Under this identification, the moment map \(\mu_H\) for the \(H\)-action can then be considered to take values in \(\mathfrak{h}\), rather than \(\mathfrak{h}^*\); we will write \(\mu_H^{\#}\) when we do this. \begin{definition}\label{def:relatively_stable}
			We say the point \(p\) is \emph{relatively stable} if \(\mu_H^{\#}(p)\in\mathfrak{h}_p\), that is, the vector field generated by \(\mu_H^{\#}(p)\in\mathfrak{h}\) fixes the point \(p\).
			\end{definition}  We remark this notion does not depend on the particular choice of invariant inner product; any other invariant product will differ from the chosen one by an equivariant isomorphism \(\mathfrak{h}\to\mathfrak{h}\), and any such isomorphism preserves \(\mathfrak{h}_p\).
			
			For a Lie subalgebra \(\mathfrak{s}\) of \(\mathfrak{h}\), we will write \[\overline{\mathfrak{s}}:=\{h\in C^\infty(M,\mathbb{R}):dh=\omega(-,Y)\text{ for some } Y\in\mathfrak{s}\},\] so that \(J\nabla:\overline{\mathfrak{s}}\to\mathfrak{s}\) is a surjection with kernel the constant functions.
			
			For each \(\epsilon>0\) sufficiently small, we will define a lifting function \(\ell_\epsilon:\overline{\mathfrak{h}}\to C^\infty(\Bl_pM,\mathbb{R})^{T'}\) in terms of the metric \(\omega_\epsilon\) constructed in \ref{sec:approx_soln}. Write \(\mathfrak{h}'\) for the orthogonal complement of \(\mathfrak{t}'\) in \(\mathfrak{h}\) with respect to the fixed invariant metric, so that \(\mathfrak{h}=\mathfrak{t}'\oplus\mathfrak{h}'\). This yields a decomposition \(\overline{\mathfrak{h}}=\overline{\mathfrak{t}'}\oplus\mathfrak{h}'\), where we have identified elements of \(\mathfrak{h}'\) with their generators in \(\overline{\mathfrak{h}}\) that are normalised to vanish at \(p\). 
			
			Any element \(h\in\overline{\mathfrak{t}'}\) generates a real holomorphic vector field \(Y\) on \(M\) that vanishes at \(p\) via the hamiltonian equation \(dh=\omega(-,Y)\), and \(Y\) lifts to a real holomorphic vector field \(\widetilde{Y}\) on \(\Bl_pM\). We define \(\ell_\epsilon(h)\) to be the hamiltonian potential for \(\widetilde{Y}\) with respect to \(\omega_\epsilon\), normalised so that \(\ell_\epsilon(h)=\pi^*h\) outside of \(\widetilde{B}_{2r_\epsilon}\) (recall that \(\omega_\epsilon=\omega\) outside \(\widetilde{B}_{2r_\epsilon}\)). For \(h\in\mathfrak{h}'\), we define \(\ell_\epsilon(h)=\pi^*(\gamma_{1,\epsilon}h)\), where \(\gamma_{1,\epsilon}\) was defined in Section \ref{sec:approx_soln}. This defines \(\ell_\epsilon\) uniquely on \(\overline{\mathfrak{h}}=\overline{\mathfrak{t}'}\oplus\mathfrak{h}'\).
			
			On \(\Bl_pM\), the weighted extremal problem can be written \begin{equation}\label{eq:weighted_extremal}
				S_{v,w}(\omega_\epsilon+i\d\db\varphi)-\frac{1}{2}\nabla_{\epsilon} h\cdot\nabla_{\epsilon}\varphi = h
			\end{equation} for \(\varphi\in\mathcal{H}_{\Bl_pM}^T\), where \(\nabla_\epsilon\) is the gradient operator of \(\omega_\epsilon\) and \(h\) is a \(T\)-invariant holomorphy potential with respect to \(\omega_\epsilon\) \cite[Lemma 4.10]{Sze14}. We will not attempt to deform \(\omega_\epsilon\) to a solution of \eqref{eq:weighted_extremal} directly, but instead prove the following direct analogue of \cite[Proposition 14]{Sze12}.
		
			\begin{proposition}\label{prop:approx_equation}
				Let \((M,\omega)\) be a \((v,w)\)-weighted extremal manifold, and let \(p\in M\) be fixed by both \(T\) and the extremal field \(X\). For all sufficiently small \(\epsilon>0\), there exist \(\varphi_\epsilon\in C^\infty(\Bl_pM)^{T'}\) and \(h_{p,\epsilon}\in\overline{\mathfrak{h}}\) such that \(\omega_\epsilon+i\d\db \varphi_\epsilon>0\) and \[S_{v,w}(\omega_\epsilon+i\d\db \varphi_\epsilon)-\frac{1}{2}\nabla_\epsilon\ell_\epsilon(h_{p,\epsilon})\cdot\nabla_\epsilon \varphi_\epsilon=\ell_\epsilon(h_{p,\epsilon}).\] Moreover, for \(\epsilon>0\) sufficiently small the following expansion holds in \(\overline{\mathfrak{h}}\): \[h_{p,\epsilon} = s + \epsilon^{2n-2}c_n\overline{\mu_H^{\#}(p)}+h_{p,\epsilon}',\] where \(s:=S_{v,w}(\omega)\in\overline{\mathfrak{h}}\) is the weighted scalar curvature generating the extremal field on \(M\), \(c_n\) is a constant depending only on \(n\),  \(\overline{\mu_H^{\#}(p)}\) is a fixed lift of \(\mu_H^{\#}(p)\) to \(\overline{\mathfrak{h}}\), and the \(h_{p,\epsilon}\) satisfy \(|h_{p,\epsilon}'|\leq C\epsilon^\kappa\) for some \(\kappa>2n-2\) and \(C>0\) independent of \(\epsilon\).
			\end{proposition}
		
			Note that if \(h_{p,\epsilon}\in \overline{\mathfrak{t}'}\) then \(\ell_\epsilon(h_{p,\epsilon})\) is a \(T\)-invariant holomorphy potential on \(\Bl_pM\), so equation \eqref{eq:weighted_extremal} is satisfied and \(\omega_\epsilon+i\d\db \varphi_\epsilon\) is weighted extremal. Denote by \(H^c\) the complexification of the group \(H\), which acts on \(M\). Suppose there exists a point \(q\) in the \(H^c\)-orbit of \(p\) for which the condition \(h_{q,\epsilon}\in\overline{\mathfrak{t}'}\) is satisfied. Then \(\Bl_qM\) admits a \((v,w)\)-weighted extremal metric, but since the manifolds \(\Bl_pM\) and \(\Bl_qM\) are \(T\)-equivariantly biholomorphic, this implies \(\Bl_pM\) admits a \((v,w)\)-weighted extremal metric. The exact same argument as \cite[p. 11 -- Proof of Theorem 1]{Sze12} shows that if \(p\) is relatively stable, then such a point \(q\) exists:

			\begin{proposition}[{\cite{Sze12}}]\label{prop:point_q}
				If Proposition \ref{prop:approx_equation} holds, and the point \(p\) is relatively stable in the sense of Definition \ref{def:relatively_stable}, then there exists a point \(q\) in the \(H^c\)-orbit of \(p\) such that \(\Bl_qM\) admits a weighted extremal metric. Since \(\Bl_qM\) and \(\Bl_pM\) are \(T\)-equivariantly diffeomorphic, there exists a weighted extremal metric on \(\Bl_pM\).
			\end{proposition}
		
		\begin{remark}
		We have taken care to make our constructions invariant under the maximal torus \(T'\) in \(H\). It will seem that we will never use this condition, however it is an essential ingredient of \cite[p. 11 -- Proof of Theorem 1]{Sze12}, so must be included in the present work.
		\end{remark}
		
			Thus, our only goal now is to prove Proposition \ref{prop:approx_equation}, as Proposition \ref{prop:point_q} will then imply Theorem \ref{thm:main}. Before launching into the proof, we must introduce the weighted H{\"o}lder norms.
		
		\subsection{Weighted norms}\label{sec:weighted_norms}
		
			We will define the weighted H{\"o}lder norms on three manifolds: \(M_p:=M\backslash\{p\}\), \(\Bl_0\mathbb{C}^n\) and \(\Bl_pM\). These are modifications of the \(C^{k,\alpha}\) norms that depend on an extra parameter \(\delta\in\mathbb{R}\), and in the case of \(\Bl_pM\) a further parameter \(\epsilon>0\). On the non-compact manifolds \(M_p\) and \(\Bl_0\mathbb{C}^n\) they allow or enforce certain growth or decay conditions at the ends, depending on the sign of \(\delta\).
			
			For \(f:M_p\to\mathbb{R}\), define \[\|f\|_{C^{k,\alpha}_\delta(M_p)}:=\|f\|_{C^{k,\alpha}(M\backslash B_{1})}+\sup_{0<r<1/2}r^{-\delta}\|f(rz)\|_{C^{k,\alpha}(B_2\backslash B_1)}.\] On \(M\backslash B_{1}\) we calculate the norm with respect to the metric \(\omega\), and on \(B_2\backslash B_1\) with respect to the fixed Euclidean metric defined by the coordinates \((z_1,\ldots,z_n)\) from Section \ref{sec:approx_soln}. The space \(C^{k,\alpha}_\delta(M_p)\) is the set of locally \(C^{k,\alpha}\)-functions on \(M_p\) with finite \(\|\cdot\|_{C^{k,\alpha}_\delta(M_p)}\)-norm.
			
			For \(f:\Bl_0\mathbb{C}^n\to\mathbb{R}\), let \[\|f\|_{C^{k,\alpha}_\delta(\Bl_0\mathbb{C}^n)}:=\|f\|_{C^{k,\alpha}(\widetilde{D}_1)}+\sup_{r>1}r^{-\delta}\|f(r\zeta)\|_{C^{k,\alpha}(\widetilde{D}_2\backslash \widetilde{D}_1)}.\] On \(\widetilde{D}_1\) we compute the norm with respect to the fixed metric \(\eta\), and on \(\widetilde{D}_2\backslash\widetilde{D}_1\) we use the Euclidean metric defined by the \(\zeta\)-coordinates. The space \(C^{k,\alpha}_\delta(\Bl_0\mathbb{C}^n)\) is the set of locally \(C^{k,\alpha}\)-functions on \(\Bl_0\mathbb{C}^n\) with finite \(\|\cdot\|_{C^{k,\alpha}_\delta(\Bl_0\mathbb{C}^n)}\)-norm.
			
			Let \(f\in C^{k,\alpha}(\Bl_pM,\mathbb{R})\). The \(C^{k,\alpha}_{\delta,\epsilon}\)-weighted norm of \(f\) is defined as \[\|f\|_{C^{k,\alpha}_{\delta,\epsilon}}:=\|f\|_{C^{k,\alpha}(\Bl_pM\backslash\widetilde{B}_1)}+\sup_{\epsilon\leq r\leq 1/2}r^{-\delta}\|f(rz)\|_{C^{k,\alpha}(\widetilde{B}_2\backslash\widetilde{B}_1)}+\epsilon^{-\delta}\|(\iota_\epsilon)_*(f)\|_{C^{k,\alpha}(\widetilde{D}_1)}.\] Here the norms are computed with respect to fixed background metrics, and we recall the map \(\iota_\epsilon:\widetilde{B}_\epsilon\to\widetilde{D}_1\) is the biholomorphism lifting the map \(B_\epsilon\to D_1\), \(z\mapsto\epsilon^{-1}z\). Note we do not include \(\Bl_pM\) in the notation for this norm; whenever the manifold is not specified, we take it that the norm is on \(\Bl_pM\).
			
			We must also define the weighted H{\"o}lder norms of a tensor \(T\) on \(\Bl_pM\). On \(\Bl_pM\backslash\widetilde{B}_1\) this is done as normal, with respect to the fixed  metric \(\omega\): \[\|T\|_{C^{k,\alpha}(\Bl_pM\backslash\widetilde{B}_1)}.\] Suppose that \(T\) is a section of \((T^*\widetilde{M})^m\otimes (T\widetilde{M})^\ell\), where \(\widetilde{M}:=\Bl_pM\). We define \(\sigma(T):=\ell-m\). On \(\widetilde{B}_2\backslash E\) we have the Euclidean coordinates \(z_1,\ldots,z_n\); we define \(\iota_r:\widetilde{B}_{2r}\backslash\widetilde{B}_r\to\widetilde{B}_2\backslash\widetilde{B}_1\) by \(\iota_r(z):=r^{-1}z\) for \(\epsilon\leq r\leq 1/2\). On \(\widetilde{B}_1\backslash\widetilde{B}_\epsilon\), the weighted norm is then \[\sup_{\epsilon\leq r\leq 1/2}r^{-\delta}\|r^{\sigma(T)}(\iota_r)_*T\|_{C^{k,\alpha}(\widetilde{B}_2\backslash\widetilde{B}_1)}.\] Finally, on the region \(\widetilde{B}_\epsilon\) we identify \(\widetilde{B}_\epsilon\) with \(\widetilde{D}_1\) via \(\iota_\epsilon:\widetilde{B}_\epsilon\to\widetilde{D}_1\) and take: \[\epsilon^{-\delta}\|\epsilon^{\sigma(T)}(\iota_\epsilon)_*T\|_{C^{k,\alpha}(\widetilde{D}_1)}.\] Thus, overall: \[\|T\|_{C^{k,\alpha}_{\delta,\epsilon}}:=\|T\|_{C^{k,\alpha}(\Bl_pM\backslash\widetilde{B}_1)}+\sup_{\epsilon\leq r\leq 1/2}r^{-\delta}\|r^{\sigma(T)}(\iota_r)_*T\|_{C^{k,\alpha}(\widetilde{B}_2\backslash\widetilde{B}_1)}+\epsilon^{-\delta}\|\epsilon^{\sigma(T)}(\iota_\epsilon)_*T\|_{C^{k,\alpha}(\widetilde{D}_1)}.\] This agrees with our definition of the \(C^{k,\alpha}_{\delta,\epsilon}\)-norm in the case \(T\) is a function. Note the central term in the norm on \(\widetilde{B}_2\backslash\widetilde{B}_1\) is equivalent to pulling back the components of the tensor \(r^{-\delta}T\) in the \(z\)-coordinates by \(\iota_r^{-1}\) and summing the \(C^{k,\alpha}\)-norms of these. However, the final term on \(\widetilde{D}_1\) does not have a similar description as the rescaling \(\iota_\epsilon\) is only in one direction, namely the fibre coordinate of \(\mathcal{O}_{\mathbb{P}^{n-1}}(-1)\). Equivalently, we could consider pulling back \(T\) to \(\widetilde{D}_1\), and then measuring its \(C^{k,\alpha}\)-norm with respect a fixed metric on \(\widetilde{D}_1\), but using the metric \((\epsilon^2\eta)^{-m}\otimes(\epsilon^2\eta)^{\ell}\) on the vector bundle \((T^*\widetilde{D}_1)^m\otimes(T\widetilde{D}_1)^\ell\).

			We similarly define the weighted \(C^k\)-norms \(\|\cdot\|_{C^k_{\delta,\epsilon}}\), without the H{\"o}lder coefficient \(\alpha\). The following properties will be useful: \begin{lemma}\label{lem:weighted_norm_properties} Let \(\epsilon>0\) and \(\delta,\delta'\in\mathbb{R}\). Then:
				 \begin{enumerate}
					\item If \(\delta\leq\delta'\) then \(\|T\|_{C^{k,\alpha}_{\delta,\epsilon}}\leq \|T\|_{C^{k,\alpha}_{\delta',\epsilon}}\) for all tensors \(T\).
					\item If \(\delta>\delta'\) then \(\|T\|_{C^{k,\alpha}_{\delta,\epsilon}}\leq\epsilon^{\delta'-\delta} \|T\|_{C^{k,\alpha}_{\delta',\epsilon}}\) for all tensors \(T\).
					\item There is a constant \(C>0\), independent of \(\delta\), \(\delta'\) and \(\epsilon\), such that \(\|ST\|_{C^{k,\alpha}_{\delta+\delta',\epsilon}}\leq C\|S\|_{C^{k,\alpha}_{\delta,\epsilon}}\|T\|_{C^{k,\alpha}_{\delta',\epsilon}}\) for all tensors \(S,T\). Here \(ST\) can mean either the tensor product \(S\otimes T\), or a contraction of any number of dual pairs in \(S\otimes T\).
					\item There is a constant \(C>0\) independent of \(\epsilon>0\) such that \(\|T\|_{C^0}\leq C\|T\|_{C^0_{0,\epsilon}}\) for all tensors \(T\in\Gamma(\widetilde{M},(T^*\widetilde{M})^k)\) with \(k\geq0\), where the \(C^0\)-norm is fixed independent of \(\epsilon\). In the case \(k=0\), i.e. \(T=f\) is a function, this is an equivalence of norms.
					\item There is a uniform equivalence of norms on functions \[\|f\|_{C^{k,\alpha}_{\delta,\epsilon}}\sim\|\gamma_{1,\epsilon}f\|_{C^{k,\alpha}_{\delta,\epsilon}\left(M_p\right)}+\epsilon^{-\delta}\|\gamma_{2,\epsilon}(\iota_\epsilon^{-1}(\zeta))f(\iota_\epsilon^{-1}( \zeta))\|_{C^{k,\alpha}_{\delta}\left(\Bl_0\mathbb{C}^n\right)}\] independent of \(\epsilon\).
					\item There is a constant \(C>0\) independent of \(\epsilon\) such that \(\|f\|_{C^{0,\alpha}_{0,\epsilon}}\leq C\|f\|_{C^1}\) for all \(f\in C^\infty(\Bl_pM;\mathbb{R})\), where we take a fixed \(C^1\)-norm on \(\Bl_pM\) independent of \(\epsilon\).
				\end{enumerate} 
			\end{lemma} 
			
			Most of these are already known and straightforward; the most difficult perhaps is (6), so we shall prove this as an example.
		
			\begin{proof}[Proof of (6)]
				Outside of \(\widetilde{B}_1\) this follows from the usual inequality \[\|f\|_{C^{0,\alpha}(\Bl_pM\backslash\widetilde{B}_1)}\leq C\|f\|_{C^1(\Bl_pM\backslash\widetilde{B}_1)}.\] For the component of the norm on \(\widetilde{B}_1\backslash\widetilde{B}_\epsilon\), for \(\epsilon\leq r\leq 1/2\) we must estimate \[\|f(rz)\|_{C^{0,\alpha}(\widetilde{B}_2\backslash\widetilde{B}_1)},\] which is bounded above by the \(C^1\)-norm of \(f(rz)\) over \(\widetilde{B}_2\backslash\widetilde{B}_1\). The \(C^0\)-norm of \(f(rz)\) is bounded by \(\|f\|_{C^0(\Bl_pM)}\). For the derivatives, consider \[\left|\frac{\d}{\d z_j}(f(rz))\right|=r\left|\frac{\d f}{\d z_j}(rz)\right|.\] In this form, we cannot immediately bound the right hand side by the \(C^1\)-norm, since the coordinates \(z_j\) do not extend to the blowup and we require a uniform bound as \(\epsilon\to0\). Instead choose coordinates \((w_1,\ldots,w_n)=(z_1,\frac{z_2}{z_1},\ldots,\frac{z_n}{z_1})\) on \(\widetilde{B}_2\) such that \(|z_1|>a|(z_2,\ldots,z_n)|\) for a fixed small \(a>0\). Note this implies \(|z_j|/|z_1|\leq 1/a\) for all \(j\).  For \(z\) in the intersection of this coordinate domain and \(\widetilde{B}_2\backslash \widetilde{B}_1\), we also have \(1/|z_1|\leq C\) for some \(C>0\). Hence \begin{align*}
					r\left|\frac{\d f}{\d z_1}(rz)\right| &= 	\left|r\frac{\d f}{\d w_1}(rz)-\frac{z_2}{z_1^2}\frac{\d f}{\d w_2}(rz)-\cdots-\frac{z_n}{z_1^2}\frac{\d f}{\d w_n}(rz)\right| \\
					&\leq C\sum_{j=1}^n\left|\frac{\d f}{\d w_j}(rz)\right|,
				\end{align*} where \(C>0\) depends only on \(a\). Similarly \[r\left|\frac{\d f}{\d z_j}(rz)\right|=\left|\frac{1}{z_1}\frac{\d f}{\d w_j}(rz)\right|\leq C\left|\frac{\d f}{\d w_j}(rz)\right|\] for \(2\leq j\leq n\). Since the coordinates \(w_j\) are well-defined on the blowup, this gives a uniform bound on this coordinate domain by \(C\|f\|_{C^1}\). Covering \(\widetilde{B}_2\backslash\widetilde{B}_1\) by the analogous coordinate domains where \(z_2\neq0,\ldots,z_n\neq0\), we get a uniform bound by \(C\|f\|_{C^1}\) on this region.
				
				It remains to prove a bound on \(\widetilde{B}_\epsilon\). We need to estimate \[\|(\iota_\epsilon)_*f\|_{C^{0,\alpha}(\widetilde{D}_1)}.\] Once again we can reduce to estimating the \(C^1\)-norm of \((\iota_\epsilon)_*f\) on \(\widetilde{D}_1\). Similarly to above, we choose coordinates \((\nu_1,\ldots,\nu_n)=(\zeta_1,\frac{\zeta_2}{\zeta_1},\ldots,\frac{\zeta_n}{\zeta_n})\) on \(\widetilde{D}_1\) for the region \(|\zeta_1|>a|(\zeta_2,\ldots,\zeta_n)|\). The \(C^0\)-norm is bounded by \(\|f\|_{C^1}\), and the \(\nu\)-derivatives satisfy \[\left|\frac{\d}{\d \nu_j}((\iota_\epsilon)_*f)\right|=\left|\frac{\d}{\d \nu_j}(f(\epsilon \nu_1,\nu_2,\ldots,\nu_n))\right|\leq\left|\frac{\d f}{\d \nu_j}(\epsilon \nu_1,\nu_2,\ldots,\nu_n)\right|\leq\|f\|_{C^1}.\] Covering \(\widetilde{D}_1\) by similarly defined coordinate charts, we produce a uniform bound by \(\|f\|_{C^1}\) on \(\widetilde{D}_1\), which finishes the proof.
			\end{proof}
			
			\begin{remark}\label{rem:coordinates}
			Using the coordinates \(w_1,\ldots,w_n\) from the above proof, we can give an interpretation of how the weighted norms on \(\widetilde{B}_\epsilon\) can be computed in coordinates. For example, if \(\xi\) is a section of the \(T^{1,0}\)-bundle of \(\Bl_pM\), writing it as \[\xi=\xi^1\frac{\d}{\d w^1}+\xi^2\frac{\d}{\d w^2}+\cdots+\xi^n\frac{\d}{\d w^n}\] on \(\widetilde{B}_\epsilon\), its pushforward to \(\widetilde{D}_1\) is \[\xi^1(\iota_\epsilon^{-1}(\nu))\epsilon^{-1}\frac{\d}{\d\nu_1}+\xi^2(\iota_\epsilon^{-1}(\nu))\frac{\d}{\d\nu_2}+\cdots+\xi^n(\iota_\epsilon^{-1}(\nu))\frac{\d}{\d\nu_n}.\] Multiplying this all by \(\epsilon=\epsilon^{\sigma(\xi)}\), we are computing the H{\"o}lder norms of \[\xi^1(\iota_\epsilon^{-1}(\nu)),\quad\epsilon\xi^2(\iota_\epsilon^{-1}(\nu)),\quad\ldots\quad\epsilon\xi^n(\iota_\epsilon^{-1}(\nu))\] on the appropriate coordinate domain on \(\widetilde{D}_1\). In particular, from this it is easy to see there is a uniform bound on \(\|\xi\|_{C^{k,\alpha}_{0,\epsilon}}\) independent of \(\epsilon\), for any \(k\) and \(\alpha\). Furthermore, if \(\xi\) is the lift of a vector field on \(M\) that vanishes at \(p\), note the coefficient \(\xi^1\) vanishes along the exceptional locus. In this case we can even produce a uniform bound on \(\|\xi\|_{C^{k,\alpha}_{1,\epsilon}}\) independent of \(\epsilon\).
			\end{remark}

			We finish by collecting some useful estimates. Let \(g_\epsilon\) be the Riemannian metric on \(\Bl_pM\) corresponding to \(\omega_\epsilon\) defined in Section \ref{sec:approx_soln}, and recall the functions \(\gamma_{1,\epsilon}\) and \(\gamma_{2,\epsilon}\) defined in \eqref{eq:gammas}. Given a K{\"a}hler potential \(\varphi\), we will write \(g_{\epsilon,\varphi}\) for the Riemannian metric corresponding to \(\omega_{\epsilon,\varphi}:=\omega_\epsilon+i\d\db\varphi\), where \(\omega_\epsilon\) is the metric from Section \ref{sec:approx_soln}. For a Riemannian metric \(g\), we write \(\Riem(g)\) for the full Riemann curvature tensor of \(g\).
			
			\begin{lemma}[{\cite[pp. 166--167]{Sze14}}]\label{lem:metric_and_gamma_estimates}
				The norms \[\|g_\epsilon\|_{C^{2,\alpha}_{0,\epsilon}},\, \|g_\epsilon^{-1}\|_{C^{2,\alpha}_{0,\epsilon}},\,\|\gamma_{1,\epsilon}\|_{C^{4,\alpha}_{0,\epsilon}},\,\|\gamma_{2,\epsilon}\|_{C^{4,\alpha}_{0,\epsilon}}\] are uniformly bounded independent of \(\epsilon\). Furthermore, given \(c_0>0\), there exists \(C>0\) independent of \(\epsilon\) such that for all K{\"a}hler potentials \(\varphi\in C^{4,\alpha}(\Bl_pM)\) satisfying \(\|\varphi\|_{C^{4,\alpha}_{2,\epsilon}}\leq c_0\), the following hold: \[\|g_{\epsilon,\varphi}\|_{C^{2,\alpha}_{0,\epsilon}},\,\|g_{\epsilon,\varphi}^{-1}\|_{C^{2,\alpha}_{0,\epsilon}},\,\|\Riem(g_{\epsilon,\varphi})\|_{C^{0,\alpha}_{-2,\epsilon}}\leq C.\] If \(\varphi\) instead satisfies \(\|\varphi\|_{C^{4,\alpha}_{\delta,\epsilon}}< c_0\) for some \(\delta\in\mathbb{R}\), then \[\|g_{\epsilon,\varphi}-g_\epsilon\|_{C^{2,\alpha}_{\delta-2,\epsilon}},\,\|g_{\epsilon,\varphi}^{-1}-g_\epsilon^{-1}\|_{C^{2,\alpha}_{\delta-2,\epsilon}},\,\|\Riem(g_{\epsilon,\varphi})-\Riem(g_\epsilon)\|_{C^{0,\alpha}_{\delta-4,\epsilon}}\leq C\|\varphi\|_{C^{4,\alpha}_{\delta,\epsilon}}.\] 
			\end{lemma}
			
			Finally, we will estimate the lifting function \(\ell_\epsilon:\overline{\mathfrak{h}}\to C^\infty(\Bl_pM)^{T'}\) from Section \ref{sec:deformation}.
			
			\begin{lemma}\label{lem:ell_estimate}
			There exists a constant \(C>0\) independent of \(\epsilon\) such that \[\|\ell_\epsilon(h)\|_{C^{1,\alpha}_{0,\epsilon}}\leq C|h|\] for all \(h\in\overline{\mathfrak{h}}\), where the norm on the right hand side is any choice of fixed norm on \(\overline{\mathfrak{h}}\) independent of \(\epsilon\).
			\end{lemma}
			
			\begin{proof}
			Recalling \(\overline{\mathfrak{h}}=\overline{\mathfrak{t}'}\oplus\mathfrak{h}'\), we treat the cases \(h\in \overline{\mathfrak{t}'}\) and \(h\in \mathfrak{h}'\) separately.
			
			For \(h\in\overline{\mathfrak{t}'}\), the real holomorphic vector field \(\xi_h\) generated by \(h\) vanishes at \(p\), so lifts to \(\widetilde{\xi}_h\) on \(\Bl_pM\). The function \(\ell_\epsilon(h)\) is then the holomorphy potential for \(\widetilde{\xi}_h\) with respect to \(\omega_\epsilon\) that is equal to \(\pi^*h\) on \(\Bl_pM\backslash\widetilde{B}_1\). The norm can be computed as \[\|\ell_\epsilon(h)\|_{C^{1,\alpha}_{0,\epsilon}}=\|\ell_\epsilon(h)\|_{C^0}+\|d\ell_\epsilon(h)\|_{C^{0,\alpha}_{-1,\epsilon}}.\] By the estimates on \(g_\epsilon\) from Lemma \ref{lem:metric_and_gamma_estimates}, note \[\|d\ell_\epsilon(h)\|_{C^{0,\alpha}_{-1,\epsilon}}=\|g_\epsilon g_\epsilon^{-1}d\ell_\epsilon(h)\|_{C^{0,\alpha}_{-1,\epsilon}}\leq C\|g_\epsilon\|_{C^{1,\alpha}_{0,\epsilon}}\|\nabla_\epsilon\ell_\epsilon(h)\|_{C^{0,\alpha}_{-1,\epsilon}}\leq C\|\widetilde{\xi}_h\|_{C^{0,\alpha}_{-1,\epsilon}}.\] Following Remark \ref{rem:coordinates}, we can produce a bound \(\|\widetilde{\xi}_h\|_{C^{0,\alpha}_{-1,\epsilon}}\leq C|h|\) by considering the supremum of \(\|\widetilde{\xi}_h\|_{C^{0,\alpha}_{-1,\epsilon}}\) over the compact unit ball in \(\overline{\mathfrak{t}'}\). Since the \(C^{0,\alpha}_{-1,\epsilon}\)-norm of vector fields decreases in \(\epsilon\), the bound is independent of \(\epsilon\). So it remains to estimate \(\|\ell_\epsilon(h)\|_{C^0}\). This is straightforward however; note that the \(T'\)-action on \(M\) has a moment map \(\mu_{T'}\) with moment polytope \(P'\subset(\mathfrak{t}')^*\), and the lift of this action to \(\Bl_pM\) has a moment map \(\mu_{T',\epsilon}\) whose image is contained in \(P'\). Writing \(\pr(h)\) for the image of \(h\) in \(\mathfrak{t}'\), we have \[\ell_\epsilon(h)=\langle \mu_{T',\epsilon},\pr(h)\rangle+h(q)-\langle \mu_{T'}(q),\pr(h)\rangle\] for any fixed \(q\in M\backslash B_1\). Since the image of \(\mu_{T',\epsilon}\) is contained in \(P'\), this expression gives a uniform bound \(\|\ell_\epsilon(h)\|_{C^0}\leq C|h|\).
			
			Next we take \(h\in\mathfrak{h}'\). Recall that \(h\) vanishes at \(p\), and the lift of \(h\) is defined as \(\ell_\epsilon(h):=\gamma_{1,\epsilon}\pi^*h\). In particular, since this function is supported on \(\Bl_pM\backslash\widetilde{B}_\epsilon\), \[\|\ell_\epsilon(h)\|_{C^{1,\alpha}_{0,\epsilon}}=\|\gamma_{1,\epsilon}h\|_{C^{1,\alpha}_{0,\epsilon}}\leq\|\gamma_{1,\epsilon}\|_{C^{1,\alpha}_{0,\epsilon}}\|h\|_{C^{1,\alpha}_{0,\epsilon}(\Bl_pM\backslash\widetilde{B}_\epsilon)}\leq C\|h\|_{C^{1,\alpha}_0(M_p)},\] where we applied the estimate on \(\gamma_{1,\epsilon}\) from Lemma \ref{lem:metric_and_gamma_estimates}. The norm \(\|h\|_{C^{1,\alpha}_0(M_p)}\) is well-defined, and we get a uniform bound \(\|h\|_{C^{1,\alpha}_0(M_p)}\leq C|h|\) by finite-dimensionality of \(\mathfrak{h}'\).
			\end{proof}

	\section{Moment map estimates}\label{sec:estimates_moment_maps}
	
		In the weighted scalar curvature \(S_{v,w}\) and its derivatives, the moment map appears in several terms. Most importantly are functions of the form \(u(\mu_\epsilon)\) for a fixed smooth function \(u:P\to\mathbb{R}\), and the Laplacian \(\Delta_\epsilon\mu_\epsilon^a\) appearing in \(\Phi_{v,w}(\omega_\epsilon)\). In this section we collect some fundamental estimates on the moment map that are used in many of the later proofs. Throughout this section we will use \(C\) to denote a positive constant that is independent of \(\epsilon\) and may vary from line to line.
		
		\begin{lemma}\label{lem:estimate_epsilon_moment_map}
			Given \(c_0>0\), there exists \(C>0\) independent of \(\epsilon\) such that \[\|\mu_{\epsilon,\varphi}\|_{C^{0,\alpha}_{0,\epsilon}}\leq C\] and \[\|\mu_{\epsilon,\varphi}-\mu_\epsilon\|_{C^{0,\alpha}_{1,\epsilon}}\leq C\|\varphi\|_{C^{4,\alpha}_{2,\epsilon}}\] for all K{\"a}hler potentials \(\varphi\in C^{4,\alpha}(\Bl_pM)^T\) satisfying \(\|\varphi\|_{C^{4,\alpha}_{2,\epsilon}}<c_0\).
		\end{lemma}
		
		\begin{proof}
			We begin by proving \(\|\mu_{\epsilon,\varphi}\|_{C^{0,\alpha}_{0,\epsilon}}\leq C\). By \emph{(6)} of Lemma \ref{lem:weighted_norm_properties} it suffices to estimate the \(C^1\)-norm of \(\mu_{\epsilon,\varphi}\). First, \(\mu_{\epsilon,\varphi}\) clearly has a uniform \(C^0\) bound, since its image is contained in the moment polytope \(P\). To estimate the derivative \(d\mu_{\epsilon,\varphi}\), note we have \begin{align*}
				\|d\mu_{\epsilon,\varphi}\|_{C^0} &\leq C \|g_{\epsilon,\varphi}\|_{C^0}\|g_{\epsilon,\varphi}^{-1}d\mu_{\epsilon,\varphi}\|_{C^0} \\
				&\leq C\|\nabla_{\epsilon,\varphi}\mu_{\epsilon,\varphi}\|_{C^0}.
			\end{align*} Here we have used properties \emph{(3)} and \emph{(4)} of Lemma \ref{lem:weighted_norm_properties} together with the estimate \(\|g_{\epsilon,\varphi}\|_{C^{0,\alpha}_{0,\epsilon}}\leq C\) of Lemma \ref{lem:metric_and_gamma_estimates}. Now, note the component functions \(\mu_{\epsilon,\varphi}^a\) of \(\mu_{\epsilon,\varphi}\) with respect to a basis \(\{\xi_a\}\) of \(\mathfrak{t}\) satisfy \(\nabla_{\epsilon,\varphi}\mu^a_{\epsilon,\varphi}=J\widetilde{\xi}_a\) by definition of the moment map. Hence \(\|\nabla_{\epsilon,\varphi}\mu_{\epsilon,\varphi}\|_{C^0}\) is uniformly bounded, and we have proven the estimate for \(\mu_{\epsilon,\varphi}\).
			
			For \(\|\mu_{\epsilon,\varphi}-\mu_\epsilon\|_{C^{0,\alpha}_{1,\epsilon}}\) note that \[\mu_{\epsilon,\varphi}^a-\mu_\epsilon^a=d^c\varphi(\widetilde{\xi}_a).\] We have \(\|d^c\varphi\|_{C^{0,\alpha}_{1,\epsilon}}\leq C\|\varphi\|_{C^{1,\alpha}_{2,\epsilon}}\leq C\|\varphi\|_{C^{4,\alpha}_{2,\epsilon}}\), and \(\|\widetilde{\xi}_a\|_{C^{0,\alpha}_{0,\epsilon}}\leq C\) by Remark \ref{rem:coordinates}.
		\end{proof}
		
		\begin{lemma}\label{lem:laplacian_order_tau}
			There exists \(C>0\) independent of \(\epsilon\) such that \[\|\Delta_\epsilon\mu_\epsilon\|_{C^{0,\alpha}_{0,\epsilon}}\leq C.\]
		\end{lemma}
		
		\begin{proof}
			We will use the explicit formula for the moment map in Lemma \ref{lem:epsilon_moment_map}. On \(\Bl_pM\backslash\widetilde{B}_{2r_\epsilon}\), since \(\Delta_\epsilon\mu_\epsilon=\Delta\mu\) is independent of \(\epsilon\) there is a uniform bound on this region, by \(\|\Delta\mu\|_{C^1}\) for example. 
			
			On \(\widetilde{B}_{r_\epsilon}\) the metric and moment map are pulled back from \(\widetilde{D}_{R_\epsilon}\subset\Bl_p\mathbb{C}^n\) via \(\iota_\epsilon\), where \(R_\epsilon:=\epsilon^{-1}r_\epsilon\). Since the vector fields \(\widetilde{\xi}_a\) are invariant under pushforward by the dilation map \(\iota_\epsilon\), the component functions \(\mu^a_\epsilon\) of the moment map are also pulled back from \(\widetilde{D}_{R_\epsilon}\). Hence, on \(\widetilde{B}_{r_\epsilon}\), \begin{align*}
				\Delta_\epsilon\mu_\epsilon^a &= \Delta_{\iota_\epsilon^*(\epsilon^2\eta)}\iota_\epsilon^*\left(\epsilon^2\sum_jA_{j}^a |\zeta_j|^2+\epsilon^2d^cg(\widetilde{\xi_a})\right) \\
				&= \iota_\epsilon^*\left(\Delta_\eta\left(\sum_jA_{j}^a |\zeta_j|^2+d^cg(\widetilde{\xi_a})\right)\right) \\
				&=\iota_\epsilon^*(\Delta_\eta\mu_\eta^a).
			\end{align*} On \(\Bl_0\mathbb{C}^n\) we have the expansion \[\eta=i\d\db\left(|\zeta|^2+g(\zeta)\right)\] as \(|\zeta|\to\infty\), where \(g(\zeta)=\O(|\zeta|^{4-2n})\) if \(n>2\) and \(g=\log|\zeta|\) if \(n=2\). It follows \(\Delta_\eta\) has the expansion \(\Delta_\eta=\Delta_{\mathrm{Euc}}+h^{j\bar{k}}\d_j\d_{\bar{k}}\), where \(\Delta_{\mathrm{Euc}}\) is the Euclidean Laplacian in the \(\zeta\)-coordinates and \(h^{j\bar{k}}\) is \(\O(|\zeta|^{2-2n})\). Since \(d^cg(\widetilde{\xi}_a)\) is \(\O(|\zeta|^{4-2n})\) this implies \[\Delta_\eta\mu_\eta^a = \sum_j A_{j}^a+\O(|\zeta|^{2-2n}),\] so \(\Delta_\eta\mu_\eta^a\) is \(\O(1)\) on \(\Bl_0\mathbb{C}^n\). It follows that \[\|\Delta_\epsilon\mu_\epsilon\|_{C^{0,\alpha}_{0,\epsilon}\left(\widetilde{B}_{r_\epsilon}\right)}=\|\Delta_\eta\mu_\eta\|_{C^{0,\alpha}_0\left(\widetilde{D}_{R_\epsilon}\right)}\leq \|\Delta_\eta\mu_\eta\|_{C^{0,\alpha}_0\left(\Bl_0\mathbb{C}^n\right)}.\] Hence \(\|\Delta_\epsilon\mu_\epsilon\|_{C^{0,\alpha}_{0,\epsilon}\left(\widetilde{B}_{r_\epsilon}\right)}\) is uniformly bounded independent of \(\epsilon\).
			
			The only remaining region to estimate is the annulus \(\widetilde{B}_{2r_\epsilon}\backslash\widetilde{B}_{r_\epsilon}\), on which \(\mu_\epsilon\) takes the form \[\mu_\epsilon=\mu_{\mathrm{Euc}}+d^c(\gamma_{1,\epsilon}(z)f(z)+\epsilon^2\gamma_{2,\epsilon}(z)g(\epsilon^{-1}z)).\] We claim that the function \(\gamma_{1,\epsilon}(z)f(z)+\epsilon^2\gamma_{2,\epsilon}(z)g(\epsilon^{-1}z)\) is \(\O(|z|^{2+\tau})\) on the annulus \(\widetilde{B}_{2r_\epsilon}\backslash\widetilde{B}_{r_\epsilon}\), for all \(\tau>0\) sufficiently small. That is to say, for fixed \(\tau>0\) small enough, there is a uniform bound on the \(C^{k,\alpha}_{2+\tau,\epsilon}\)-norm of this function over the annulus, independent of \(\epsilon\). To see this, first note that \(\gamma_{1,\epsilon}f\) is \(\O(|z|^4)\) so it suffices to estimate \(\epsilon^2\gamma_{2,\epsilon}(z)g(\epsilon^{-1}z)\), or equivalently \(\epsilon^2g(\epsilon^{-1}z)\) since \(\gamma_{2,\epsilon}\) is uniformly \(\O(1)\). In the case \(n>2\), \(g\) is \(\O(|\zeta|^{4-2n})\), and \[r_\epsilon^{-3}(\epsilon^2g(\epsilon^{-1}r_\epsilon z))=\epsilon^{-1}R_\epsilon^{1-2n}(R_\epsilon^{2n-4}g(R_\epsilon z)),\] where  \(1\leq|z|\leq 2\). The \(C^{k,\alpha}\)-norm of \(R_\epsilon^{2n-4}g(R_\epsilon z)\) is uniformly bounded on the annulus \(1\leq|z|\leq 2\), hence the \(C^{k,\alpha}\)-norm of this expression is bounded by \[\epsilon^{-1}R_\epsilon^{1-2n}=\epsilon^{\frac{2n-3}{2n+1}}\to0\] as \(\epsilon\to0\), where we used the definition \(r_\epsilon:=\epsilon^{\frac{2n-1}{2n+1}}\). Hence \(\epsilon^2g(\epsilon^{-1})\) is \(\O(|z|^3)\) on \(\widetilde{B}_{2r_\epsilon}\backslash\widetilde{B}_{r_\epsilon}\). In the case \(n=2\) we have \(g=\log|\zeta|\), and \[r_\epsilon^{-2-\tau}(\epsilon^2g(\epsilon^{-1}r_\epsilon z))=\epsilon^{2-\tau}r_\epsilon^{-2}(R_\epsilon^{-\tau}\log|R_\epsilon z|).\] The \(C^{k,\alpha}\)-norm of \(R_\epsilon^{-\tau}\log|R_\epsilon z|\) is uniformly bounded on \(1\leq|z|\leq2\), so the \(C^{k,\alpha}\)-norm of this expression is bounded by \[ \epsilon^{2-\tau}r_\epsilon^{-2}=\epsilon^{2-\tau-2\frac{2n-1}{2n+1}}\to0\] as \(\epsilon\to0\) provided \(\tau<4/5\). It follows that \(\epsilon^2g(\epsilon^{-1}z)\) is \(\O(|z|^{2+\tau})\) on \(\widetilde{B}_{2r_\epsilon}\backslash\widetilde{B}_{r_\epsilon}\) for \(n=2\) and \(0<\tau<4/5\).
			
			We see from this that the Laplacian \(\Delta_\epsilon\) satisfies \(\Delta_\epsilon=\Delta_{\mathrm{Euc}}+h^{j\bar{k}}\d_j\d_{\bar{k}}\) on this region, where \(h^{j\bar{k}}\) is uniformly \(\O(|z|^\tau)\). Hence \[\Delta_\epsilon\mu_\epsilon^a=\sum_jA_{j}^a+\O(|z|^\tau)\] on the annulus. This gives \[\|\Delta_\epsilon\mu_\epsilon^a\|_{C^{0,\alpha}_{0,\epsilon}\left(\widetilde{B}_{2r_\epsilon}\backslash\widetilde{B}_{r_\epsilon}\right)}\leq C(1+r_\epsilon^\tau)\] and the right hand side is uniformly bounded. This completes the estimate on the annulus, so we have a uniform bound on \(\|\Delta_\epsilon\mu_\epsilon\|_{C^{0,\alpha}_{0,\epsilon}}\) on all of \(\Bl_pM\).
		\end{proof}
	
		\begin{lemma}\label{lem:estimates:Ricci_moment_map}
			Given \(c_0>0\) there exists \(C>0\) independent of \(\epsilon\) such that \[\|\Delta_{\epsilon,\varphi}\mu_{\epsilon,\varphi}\|_{C^{0,\alpha}_{0,\epsilon}}\leq C\] and \[\|\Delta_{\epsilon,\varphi}\mu_{\epsilon,\varphi}-\Delta_{\epsilon}\mu_{\epsilon}\|_{C^{0,\alpha}_{0,\epsilon}}\leq C\|\varphi\|_{C^{4,\alpha}_{2,\epsilon}}\] for all \(\varphi\) such that \(\|\varphi\|_{C^{4,\alpha}_{2,\epsilon}}\leq c_0\).
		\end{lemma}
	
		\begin{proof}
			Note it suffices to prove the second estimate, since by the previous lemma, \begin{align*}
				\|\Delta_{\epsilon,\varphi}\mu_{\epsilon,\varphi}\|_{C^{0,\alpha}_{0,\epsilon}} &\leq \|\Delta_{\epsilon,\varphi}\mu_{\epsilon,\varphi}-\Delta_{\epsilon}\mu_{\epsilon}\|_{C^{0,\alpha}_{0,\epsilon}}+\|\Delta_{\epsilon}\mu_{\epsilon}\|_{C^{0,\alpha}_{0,\epsilon}} \\
				&\leq C+\|\Delta_{\epsilon,\varphi}\mu_{\epsilon,\varphi}-\Delta_{\epsilon}\mu_{\epsilon}\|_{C^{0,\alpha}_{0,\epsilon}}.
			\end{align*} Dropping the \(\epsilon\) subscripts, we compute this as \begin{align*}
			\Delta_\varphi\mu_\varphi^a-\Delta\mu^a &= g_\varphi^{-1}i\d\db\mu_\varphi^a-g^{-1}i\d\db\mu^a \\
			&= (g_\varphi^{-1}-g^{-1})i\d\db(\mu^a+d^c\varphi(\widetilde{\xi}_a))+g^{-1}i\d\db(\mu_\varphi^a-\mu^a)\\
			&=(g_\varphi^{-1}-g^{-1})i\d\db(\mu^a+d^c\varphi(\widetilde{\xi}_a))+g^{-1}i\d\db(d^c\varphi(\widetilde{\xi}_a)) \\
			&= (g_\varphi^{-1}-g^{-1})i\d\db\mu^a+g_\varphi^{-1}i\d\db(d^c\varphi(\widetilde{\xi}_a)) .
			\end{align*} From here, \begin{align*}
			\|(g_\varphi^{-1}-g^{-1})i\d\db\mu^a\|_{C^{0,\alpha}_{0,\epsilon}} &\leq C\|g_\varphi^{-1}-g^{-1}\|_{C^{0,\alpha}_{0,\epsilon}}\|\d(gg^{-1}\db\mu^a)\|_{C^{0,\alpha}_{0,\epsilon}} \\
			&\leq C\|\varphi\|_{C^{4,\alpha}_{2,\epsilon}}\|g\|_{C^{1,\alpha}_{0,\epsilon}}\|\widetilde{\xi}_a\|_{C^{1,\alpha}_{1,\epsilon}} \\
			&\leq C\|\varphi\|_{C^{4,\alpha}_{2,\epsilon}},
			\end{align*} where we used the uniform bound on \(\|\widetilde{\xi}_a\|_{C^{k,\alpha}_{1,\epsilon}}\) from Remark \ref{rem:coordinates}. Finally, \begin{align*}
			\|g_\varphi^{-1}i\d\db(d^c\varphi(\widetilde{\xi}_a))\|_{C^{0,\alpha}_{0,\epsilon}} &\leq C\|g_\varphi^{-1}\|_{C^{0,\alpha}_{0,\epsilon}}\| d^c\varphi(\widetilde{\xi}_a)\|_{C^{2,\alpha}_{2,\epsilon}} \\
			&\leq C\|d^c\varphi\|_{C^{2,\alpha}_{1,\epsilon}}\|\widetilde{\xi}_a\|_{C^{2,\alpha}_{1,\epsilon}} \\
			&\leq C\|\varphi\|_{C^{4,\alpha}_{2,\epsilon}}. \qedhere
			\end{align*}	\end{proof}

		\begin{lemma}
			Let \(u:P\to\mathbb{R}\) be a fixed smooth function. Given \(c_0>0\), there is a constant \(C>0\) independent of \(\epsilon\) such that \(\|u(\mu_{\epsilon,\varphi})\|_{C^{0,\alpha}_{0,\epsilon}}<C\) whenever \(\|\varphi\|_{C^{4,\alpha}_{2,\epsilon}}<c_0\).
		\end{lemma}
		
		\begin{proof}
			Again it suffices to bound the \(C^1\)-norm. Note the \(C^0\)-norm is bounded, since the image of \(\mu_{\epsilon,\varphi}\) is contained in \(P\) and the image of \(u:P\to\mathbb{R}\) is compact. For the derivative, by the same argument as in the proof of Lemma \ref{lem:estimate_epsilon_moment_map} it is sufficient to estimate \[\nabla_{\epsilon,\varphi}(u(\mu_{\epsilon,\varphi}))=\sum_a u_{,a}(\mu_{\epsilon,\varphi})\widetilde{\xi}_a.\] By the same reasoning as for \(u(\mu_{\epsilon,\varphi})\) there is a \(C^0\)-bound on \(u_{,a}(\mu_{\epsilon,\varphi})\), and \(\widetilde{\xi}_a\) is a fixed vector field, so also has a \(C^0\) bound. 
		\end{proof}

		\begin{lemma}
			Let \(u:P\to\mathbb{R}\) be a fixed smooth function. Given \(c_0>0\), there is a constant \(C>0\) independent of \(\epsilon\) such that \(\|u(\mu_{\epsilon,\varphi})-u(\mu_\epsilon)\|_{C^{0,\alpha}_{0,\epsilon}}\leq C\|\varphi\|_{C^{4,\alpha}_{2,\epsilon}}\) for all \(\varphi\in C^{4,\alpha}_{2,\epsilon}\) satisfying \(\|\varphi\|_{C^{4,\alpha}_{2,\epsilon}}\leq c_0\).
		\end{lemma}
		
		\begin{proof}
			We again bound the \(C^1\)-norm. The \(C^0\) bound follows from the mean value theorem, since for \(x\in\Bl_pM\), \[u(\mu_{\epsilon,\varphi}(x))-u(\mu_\epsilon(x))=\sum_{a}u_{,a}(p_x)d^c\varphi(\widetilde{\xi}_a)(x)\] for some \(p_x\in P\) on the line joining \(\mu_{\epsilon,\varphi}(x)\) to \(\mu_\epsilon(x)\), and we can bound the terms on the right hand side as follows: \begin{align*}
				\|u_{,a}(p_x)d^c\varphi(\widetilde{\xi}_a)(x)\|_{C^0}
				&\leq C\|u_{,a}\|_{C^0}\|d^c\varphi\|_{C^0}\|\widetilde{\xi}_a\|_{C^0} \\
				&\leq C \|d\varphi\|_{C^0_{0,\epsilon}} \\
				&\leq C\|\varphi\|_{C^{4,\alpha}_{2,\epsilon}}.
			\end{align*} To estimate the derivative, \[d(u(\mu_{\epsilon,\varphi})-u(\mu_\epsilon))=\sum_a u_{,a}(\mu_{\epsilon,\varphi})d\mu_{\epsilon,\varphi}^a-u_{,a}(\mu_\epsilon)d\mu_\epsilon^a,\] and \begin{align*}
				&u_{,a}(\mu_{\epsilon,\varphi})d\mu_{\epsilon,\varphi}^a-u_{,a}(\mu_\epsilon)d\mu_\epsilon^a \\
				=& 	(u_{,a}(\mu_{\epsilon,\varphi})-u_{,a}(\mu_\epsilon))d\mu_{\epsilon,\varphi}^a+u_{,a}(\mu_\epsilon)d(\mu_{\epsilon,\varphi}^a-\mu_\epsilon^a) \\
				=&(u_{,a}(\mu_{\epsilon,\varphi})-u_{,a}(\mu_\epsilon))d\mu_{\epsilon,\varphi}^a+u_{,a}(\mu_\epsilon)d(d^c\varphi(\widetilde{\xi}_a)).
			\end{align*} The same method as for \(u\) proves a \(C^0\) bound by \(C\|\varphi\|_{C^{4,\alpha}_{2,\epsilon}}\) on \(u_{,a}(\mu_{\epsilon,\varphi})-u_{,a}(\mu_\epsilon)\). For the \(C^0\)-bound on \(d\mu_{\epsilon,\varphi}^a\), as in previous proofs it is enough to note \(\nabla_{\epsilon,\varphi}\mu_{\epsilon,\varphi}^a=\widetilde{\xi}_a\) is bounded. Clearly \(u_{,a}(\mu_\epsilon)\) also has a \(C^0\)-bound. Lastly \begin{align*}
				\|d(d^c\varphi(\widetilde{\xi}_a))\|_{C^0} &\leq
				\|d(d^c\varphi(\widetilde{\xi}_a))\|_{C^{0,\alpha}_{0,\epsilon}} \\ 
				&\leq \|d^c\varphi(\widetilde{\xi}_a)\|_{C^{1,\alpha}_{1,\epsilon}} \\ 
				&\leq C\|\widetilde{\xi}_a\|_{C^{1,\alpha}_{0,\epsilon}}\|\varphi\|_{C^{2,\alpha}_{2,\epsilon}}\\
				&\leq C\|\varphi\|_{C^{4,\alpha}_{2,\epsilon}}. \qedhere
			\end{align*} 
		\end{proof}
	
	\section{Estimates for the weighted linearisation}
		
		Throughout the rest of the paper, we will use the notation and set up of Section \ref{sec:deformation}, and our ultimate goal is to prove Proposition \ref{prop:approx_equation}. Let \(\check{L}_{\epsilon,\varphi}\) denote the linearisation of the weighted scalar curvature operator \[\psi\mapsto S_{v,w}(\omega_{\epsilon}+i\d\db\psi)\] at \(\varphi\in\mathcal{H}^T_{\omega_\epsilon}\). Our aim in this section is to prove: \begin{proposition}\label{prop:weighted_Lich_estimate}
			For any \(\delta<0\) there exist \(c_0,C>0\) independent of \(\epsilon\) such that for all \(\varphi\in C^{4,\alpha}(\Bl_pM)^T\) with \(\|\varphi\|_{C^{4,\alpha}_{2,\epsilon}}<c_0\) and all \(\psi\in C^{4,\alpha}(\Bl_pM)^T\), \[\|\check{L}_{\epsilon,\varphi}(\psi)-\check{L}_\epsilon(\psi)\|_{C^{0,\alpha}_{\delta-4,\epsilon}}\leq C\|\varphi\|_{C^{4,\alpha}_2}\|\psi\|_{C^{4,\alpha}_{\delta,\epsilon}}.\]
		\end{proposition}
		
		Let us define \(u:=v/w\). For simplicity, we will now drop \(\epsilon\) from the notation, although the reader should keep in mind the base metric is \(g_\epsilon\) with moment map \(\mu_\epsilon\), even though we write these as \(g\) and \(\mu\). We will also write the weighted norms as \(\|\cdot\|_{C^{k,\alpha}_{\delta}}\), although on \(\Bl_pM\) these always depend on \(\epsilon\). As in the previous section, \(C\) denotes a positive constant that is independent of \(\epsilon\) and may vary from line to line.

		\begin{proof}[Sketch proof.]
		From Lemma \ref{lem:weighted_linearisation} we can write \begin{align}
					&\,\,\check{L}_{\varphi}(\psi)-\check{L}(\psi) \nonumber \\
					=&\,\,u(\mu_{\varphi})L_{\varphi}(\psi) - u(\mu)L(\psi) \label{eq:line1} \\
					-&\,\frac{2}{w(\mu_{\varphi})}(\db_{\varphi}^*\Lich_{\varphi}\psi,\nabla_{\varphi}^{1,0}(v(\mu_{\varphi})))_{\varphi}+\frac{2}{w(\mu)}(\db^*\Lich\psi,\nabla^{1,0}(v(\mu))) \label{eq:line2} \\
					+&\,\frac{1}{w(\mu_{\varphi})}(\Lich_{\varphi}\psi,\Lich_{\varphi}(v(\mu_{\varphi})))_{\varphi}-\frac{1}{w(\mu)}(\Lich\psi,\Lich(v(\mu))) \label{eq:line3} \\
					+&\,\frac{1}{2}S(\omega_{\varphi})\nabla_{\varphi}\left(u(\mu_{\varphi})\right)\cdot\nabla_{\varphi}\psi - \frac{1}{2}S(\omega)\nabla\left(u(\mu)\right)\cdot\nabla\psi \label{eq:line4} \\
					+&\,\frac{1}{2}\nabla_\varphi \Phi_{v,w}(\omega_\varphi)\cdot\nabla_\varphi\psi - \frac{1}{2}\nabla \Phi_{v,w}(\omega)\cdot\nabla\psi, \label{eq:line5}
				\end{align} where \(L\) is the linearisation of the unweighted scalar curvature operator. To prove Proposition \ref{prop:weighted_Lich_estimate}, each of these lines can be estimated separately. The techniques are fairly similar for each line, so we will only estimate some terms to demonstrate how this may be done.
				
		The line \eqref{eq:line1} is equal to \begin{equation}\label{eq:the_trick}
					(u(\mu_\varphi)-u(\mu))L_\varphi(\psi)+u(\mu)(L_\varphi(\psi)-L(\psi)).
				\end{equation} From the estimates in Section \ref{sec:estimates_moment_maps}, we have \[\|u(\mu_\varphi)-u(\mu)\|_{C^{0,\alpha}_0}\leq C\|\varphi\|_{C^{4,\alpha}_2} \] and \[\|u(\mu)\|_{C^{0,\alpha}_0}\leq C.\] It therefore suffices to show \[\|L_\varphi(\psi)\|_{C^{0,\alpha}_{\delta-4}}\leq C\|\psi\|_{C^{4,\alpha}_\delta}\] and \[\|L_\varphi(\psi)-L(\psi)\|_{C^{0,\alpha}_{\delta-4}}\leq C\|\varphi\|_{C^{4,\alpha}_2}\|\psi\|_{C^{4,\alpha}_\delta}.\] The latter inequality is already proven \cite[Proposition 18]{Sze12}. To see the norms of \(L_\varphi:C^{4,\alpha}_\delta\to C^{0,\alpha}_{\delta-4}\) are uniformly bounded, we use the formula \[L_\varphi\psi = \Delta_\varphi^2\psi+\Ric(\omega_\varphi)^{j\bar{k}}\d_j\d_{\bar{k}}\psi.\] The estimates in Lemma \ref{lem:metric_and_gamma_estimates} then give the uniform bound: \begin{align*}
					\|\Delta_\varphi^2\psi\|_{C^{0,\alpha}_{\delta-4}} &= \|g_\varphi^{-1}i\d\db(g_\varphi^{-1}i\d\db\psi)\|_{C^{0,\alpha}_{\delta-4}} \\
					&\leq C\|g_\varphi^{-1}\|_{C^{0,\alpha}_0}\|\d\db(g_\varphi^{-1}i\d\db\psi)\|_{C^{0,\alpha}_{\delta-4}} \\
					&\leq C\|g_\varphi^{-1}i\d\db\psi\|_{C^{2,\alpha}_{\delta-2}} \\
					&\leq C\|\psi\|_{C^{4,\alpha}_{\delta}}.
				\end{align*} Similarly \begin{align*}
					\|\Ric(\omega_\varphi)^{j\bar{k}}\d_j\d_{\bar{k}}\psi\|_{C^{0,\alpha}_{\delta-4}} &\leq C\|g_\varphi\|_{C^{0,\alpha}_0}\|g_\varphi^{-1}\|_{C^{0,\alpha}_0}\|\Riem(\omega_\varphi)\|_{C^{0,\alpha}_{-2}}\|i\d\db\psi\|_{C^{0,\alpha}_{\delta-2}} \\
					&\leq C\|\psi\|_{C^{4,\alpha}_\delta}.
				\end{align*}  
				
				To estimate the other lines of \eqref{eq:line1}-\eqref{eq:line5} we apply a similar principle: namely, wherever we see an expression of the form \(A_\varphi B_\varphi-AB\), we write \[A_\varphi B_\varphi-AB=(A_\varphi-A)B_\varphi+A(B_\varphi-B).\] Applying this trick recursively, we reduce to estimating terms of the form \(A_\varphi-A\) and \(A_\varphi\). In particular, these require estimates \begin{equation}\label{eq:general_estimate}
							\|A_\varphi\|_{C^{0,\alpha}_{-k}}\leq C,\quad\quad \|A_\varphi-A_0\|_{C^{0,\alpha}_{-k}}\leq C\|\varphi\|_{C^{4,\alpha}_2}
				\end{equation} for some \(0\leq k\leq 4\) depending on \(A\). For example, when \(A_\varphi=u(\mu_\varphi)\), we have such estimates on the \(C^{0,\alpha}_0\)-norms of \(A_\varphi\) and \(A_\varphi-A\) from Section \ref{sec:estimates_moment_maps}, and when \(A_\varphi=g_\varphi\) or \(g_\varphi^{-1}\), the \(C^{0,\alpha}_0\)-estimates are from Lemma \ref{lem:metric_and_gamma_estimates}.
			
				In some lines we are required to estimate derivatives of \(v(\mu_\varphi)\), and we only have estimates on \(v(\mu_\varphi)\) itself so far. To estimate these derivatives, we use the chain rule together with the definition of the moment map to reduce these estimates to the \(C^{0,\alpha}_0\)-estimates already proven in Section \ref{sec:estimates_moment_maps}. For example,\footnote{Before now we have written \(\widetilde{\xi}\) for the real holomorphic vector field generated by \(\xi\), whereas here we instead denote the associated \((1,0)\)-vector field by the same symbol. It will always be clear from the context to which vector field we are referring, so we allow this small abuse of notation.} \[\nabla^{1,0}(v(\mu))=\sum_a v_{,a}(\mu)\nabla^{1,0}\mu^a=\sum_a v_{,a}(\mu)\widetilde{\xi}_a,\] and \begin{align*}
					\Lich(v(\mu)) &= \db\nabla^{1,0}(v(\mu)) \\
					&= \db\sum_{a}v_{,a}(\mu)\widetilde{\xi}_a \\
					&= \sum_{a,b}v_{,ab}(\mu)\db\mu^b\otimes\widetilde{\xi}_a \\
					&= \sum_{a,b}v_{,ab}(\mu)(\widetilde{\xi}_b)^\flat\otimes\widetilde{\xi}_a,
				\end{align*} where \(\flat\) is conversion from a \((1,0)\)-vector field to a \((0,1)\)-form via the metric, and in the third line we have used that \(\db\widetilde{\xi}_a=0\). From these expressions, the factors \(\nabla^{1,0}_\varphi v(\mu_\varphi)\) and \(\Lich_\varphi(v(\mu_\varphi))\) satisfy estimates of the form \eqref{eq:general_estimate} with \(k=0\). With this in mind, it is straightforward to estimate lines \eqref{eq:line2} and \eqref{eq:line3}; we note the formula \eqref{eq:db_adjoint} for \(\db^*\) from Section 2 can be used to estimate \eqref{eq:line2}.
			
				Line \eqref{eq:line4} is similarly straightforward; we only note the inequalities \[\|S(\omega_\varphi)-S(\omega)\|_{C^{0,\alpha}_{-2}}\leq C\|\varphi\|_{C^{4,\alpha}_2}\] and \[\|S(\omega_\varphi)\|_{C^{0,\alpha}_{-2}}\leq C\] which follow from Lemma \ref{lem:metric_and_gamma_estimates}.
				
				Finally for line \eqref{eq:line5} we recall Lemma \ref{lem:Phi} which states that \(\Phi_{v,w}(\omega)\) can be written as a linear combination of functions of the form \(u_a(\mu)\Delta\mu^a\) and \(u_{ab}(\mu)g(\widetilde{\xi}_a,\widetilde{\xi}_b)\), for finitely many fixed smooth functions \(u_a\) and \(u_{ab}\) on the moment polytope. Taking the gradients of these gives \[\nabla(u_a(\mu)\Delta\mu^a)=u_a(\mu)\Ric(\widetilde{\xi}_a,-)^\#+\sum_b u_{a,b}(\mu)(\Delta\mu^a)\widetilde{\xi}_b\] and \[\nabla(u_{ab}(\mu)g(\widetilde{\xi}_a,\widetilde{\xi}_b))=\sum_c u_{ab,c}(\mu)g(\widetilde{\xi}_a,\widetilde{\xi}_b)\widetilde{\xi}_c+u_{ab}(\mu)(g(\nabla\widetilde{\xi}_a,\widetilde{\xi}_b)+g(\widetilde{\xi}_a,\nabla\widetilde{\xi}_b)),\] where \(\#\) is conversion from a \(1\)-form to a vector field via the metric; here we have also used Remark \ref{rem:Ricci_moment_map} that \(\Delta\mu\) is a moment map for the Ricci curvature. Given the estimates in Lemma \ref{lem:estimates:Ricci_moment_map} on \(\Delta_\varphi\mu_\varphi\), it is straightforward to bound line \eqref{eq:line5}, which completes the proof of Proposition \ref{prop:weighted_Lich_estimate}.
		\end{proof}
		
		As a corollary, we obtain:
		
		\begin{lemma}\label{lem:Q-estimate}
			Let \(\check{Q}_\epsilon\) be the non-linear part of the weighted scalar curvature operator \(\check{S}:=S_{v,w}\) with respect to \(\omega_\epsilon\), so that \[\check{S}(\omega_\epsilon+i\d\db\psi)=\check{S}(\omega_\epsilon)+\check{L}_\epsilon(\psi)+\check{Q}_\epsilon(\psi).\] Given \(\delta<0\), there exist \(C,c_0>0\) such that if \[\|\varphi\|_{C^{4,\alpha}_2},\,\|\psi\|_{C^{4,\alpha}_2}\leq c_0,\] then \[\|\check{Q}_\epsilon(\varphi)-\check{Q}_\epsilon(\psi)\|_{C^{4,\alpha}_{\delta-4}}\leq C\left(\|\varphi\|_{C^{4,\alpha}_2}+\|\psi\|_{C^{4,\alpha}_2}\right)\|\varphi-\psi\|_{C^{4,\alpha}_{\delta}}\]
		\end{lemma}
		
		\begin{proof}
		The proof is exactly the same as \cite[Lemma 19]{Sze12}. Namely, by the mean value theorem there exists \(\chi\) on the line joining \(\varphi\) and \(\psi\) such that \[\check{Q}_\epsilon(\varphi)-\check{Q}_\epsilon(\psi)=D\check{Q}_\epsilon|_{\chi}(\varphi-\psi)=(\check{L}_{\epsilon,\chi}-\check{L}_\epsilon)(\varphi-\psi).\] Applying Proposition \ref{prop:weighted_Lich_estimate}, \begin{align*}
		\|\check{Q}_\epsilon(\varphi)-\check{Q}_\epsilon(\psi)\|_{C^{4,\alpha}_{\delta-4}} 
		&\leq C\|\chi\|_{C^{4,\alpha}_2}\|\varphi-\psi\|_{C^{4,\alpha}_{\delta}} \\
		&\leq C\left(\|\varphi\|_{C^{4,\alpha}_2}+\|\psi\|_{C^{4,\alpha}_2}\right)\|\varphi-\psi\|_{C^{4,\alpha}_{\delta}}.	\qedhere
		\end{align*}
		\end{proof}

	\section{A right-inverse of the linearised operator}
	
		Recall from Section \ref{sec:deformation} the lifting operator \(\ell_\epsilon:\overline{\mathfrak{h}}\to C^\infty(\Bl_pM)^{T'}\). We also write \(X:=\nabla_\omega S_{v,w}(\omega)\) for the extremal vector field on \(M\), and write \(\widetilde{X}\) for its lift to \(\Bl_pM\). The aim of this section is to prove:
		
		\begin{proposition}\label{prop:right_inverse}
			For \(n>2\), let \(\delta\in(4-2n,0)\). For sufficiently small \(\epsilon>0\), the operator \(G_\epsilon:C^{4,\alpha}_{\delta,\epsilon}(\Bl_pM)^{T'}\times\overline{\mathfrak{h}}\to C^{0,\alpha}_{\delta-4,\epsilon}(\Bl_pM)^{T'}\), \[G_\epsilon(\psi,f):=\check{L}_{\epsilon}\psi-\frac{1}{2}\widetilde{X}(\psi)-\ell_\epsilon(f)\] has a right inverse \(P_\epsilon\), satisfying \(\|P_\epsilon\|\leq C\) for a constant \(C>0\) independent of \(\epsilon\).
			
			In the case \(n=2\), given \(\delta<0\) sufficiently close to \(0\), there is a right inverse \(P_\epsilon\) for \(G_\epsilon\) satisfying \(\|P_\epsilon\|\leq C\epsilon^\delta\).
		\end{proposition}
	
		Following \cite[Proposition 20]{Sze12}, the rough approach is to glue together right-inverses for linearised operators on \(M_p:=M\backslash\{p\}\) and \(\Bl_0\mathbb{C}^n\) to construct an approximate right-inverse, and then to deform this to a genuine right-inverse. However, instead of gluing an inverse for the weighted linearization on \(\Bl_0\mathbb{C}^n\), we can glue in an inverse for the usual unweighted linearization. We will see that near the exceptional divisor, the extra terms from the weighted setting are sufficiently small that we can still deform this to a genuine weighted right-inverse regardless.
		
		Before giving the next proof, we recall the notion of \emph{indicial roots} in the weighted Fredholm theory. On \(\mathbb{R}^m\backslash\{0\}\), the indicial roots of the Laplacian are the growth rates of radially symmetric harmonic functions on \(\mathbb{R}^m\backslash\{0\}\). That is, \(\delta\in\mathbb{R}\) is an \emph{indicial root} of \(\Delta_{\mathbb{R}^n}\) if there exists a non-zero harmonic function \(f(r)\) on \(\mathbb{R}^m\backslash\{0\}\) such that \(f\in C^{k,\alpha}_\delta\) for all \(k\) and \(\alpha\), where \(r\) is the radial coordinate.
		
		\begin{lemma}[{\cite[Theorem 1.7]{Bar86}}]\label{lem:indicial_roots}
		For \(m\geq4\), the indicial roots of the Laplacian on \(\mathbb{R}^m\backslash\{0\}\) are \(\mathbb{Z}\backslash\{-1,-2,\ldots,3-m\}\).
		\end{lemma}
		
		On the manifolds \(M_p\) and \(\Bl_0\mathbb{C}^n\), the K{\"a}hler Laplacians agree asymptotically with the Euclidean Laplacian. Using this, one can prove:
		
		\begin{lemma}[{\cite[Theorem 8.6]{Sze14}}]\label{lem:indicial_roots_Fredholm}
		Let \(n\geq2\). If \(\delta\in\mathbb{R}\) is not an indicial root of the Euclidean Laplacian on \(\mathbb{R}^{2n}\backslash\{0\}\), then the operators \[\Delta_\omega:C^{k,\alpha}_\delta(M_p)\to C^{k-2,\alpha}_{\delta-2}(M_p)\] and \[\Delta_\eta:C^{k,\alpha}_{\delta}(\Bl_0\mathbb{C}^n)\to C^{k-2,\alpha}_{\delta-2}(\Bl_0\mathbb{C}^n)\] are Fredholm.
		
		Hence, for \(n>2\) let \(\delta\in(4-2n,0)\), and for \(n=2\) let \(\delta\in(-1,0)\). Then \(\Delta_\omega^2\) and \(\Delta_\eta^2\) are Fredholm as maps \(C^{k,\alpha}_\delta\to C^{k-4,\alpha}_{\delta-4}\) on the manifolds \(M_p\) and \(\Bl_0\mathbb{C}^n\) respectively. 
		\end{lemma}
		
		Using this information, we can study the properties of the weighted Lichnerowicz operator on weighted H{\"o}lder spaces:
	
		\begin{lemma}\label{lem:M_p_right_inverse}
			For \(n>2\) let \(\delta\in(4-2n,0)\), and for \(n=2\) let \(\delta\in(-1,0)\). The operator \(C^{4,\alpha}_\delta(M_p)^{T'}\times\overline{\mathfrak{h}}\to C^{0,\alpha}_{\delta-4}(M_p)^{T'}\), \[(\varphi,f)\mapsto \frac{v(\mu)}{w(\mu)}\mathcal{D}_v^*\mathcal{D}\varphi-f|_{M_p}\] has a bounded right-inverse.
		\end{lemma}
	
		\begin{proof}
			We first note that the leading order term of \(\check{\mathcal{L}}\) is the linear isomorphism \(v(\mu)/w(\mu)\) composed with \(\Delta^2_\omega\), and \(\Delta_\omega^2\) is Fredholm on \(C^{4,\alpha}_\delta(M_p)\) by Lemma \ref{lem:indicial_roots_Fredholm}. The remaining terms in \(\check{\mathcal{L}}\) define bounded linear maps \(C^{4,\alpha}_\delta(M_p)\to C^{1,\alpha}_{\delta-3}(M_p)\), and the inclusion \(C^{1,\alpha}_{\delta-3}(M_p)\subset C^{0,\alpha}_{\delta-4}(M_p)\) is compact. Hence \(\check{\mathcal{L}}\) is Fredholm on \(C^{4,\alpha}_\delta(M_p)\). Furthermore, \(\check{\mathcal{L}}\) is formally self-adjoint with respect to the \(L^2\)-inner product \(\langle f,g\rangle_w:=\int fg \,w(\mu)\omega^n\). It follows that the image of \(\check{\mathcal{L}}:C^{4,\alpha}_\delta\to C^{0,\alpha}_{\delta-4}\) is the \(L^2\)-orthogonal complement of the kernel of \(\check{\mathcal{L}}:C^{4,\alpha}_{4-2n-\delta}\to C^{0,\alpha}_{-2n-\delta}\).\footnote{ Strictly speaking, the \(L^2\)-inner product does not give a well-defined pairing between the spaces \(C^{0,\alpha}_{\delta-4}\) and \(C^{4,\alpha}_{4-2n-\delta}\), for example since \(|z|^{\delta-4}|z|^{4-2n-\delta}=|z|^{-2n}\) is not integrable. However, since \(4-2n-\delta\in(4-2n,0)\), the weight \(4-2n-\delta\) is not an indicial root of \(\check{\mathcal{L}}\). Hence elements in the kernel of \(\check{\mathcal{L}}:C^{4,\alpha}_{4-2n-\delta}\to C^{0,\alpha}_{-2n-\delta}\) are in fact contained in weighted H{\"o}lder spaces with strictly higher weights \cite[Lemma 12.1.2]{Pac08}, so have a well-defined \(L^2\)-pairing with elements of \(C^{0,\alpha}_{\delta-4}\). I thank Lars Sektnan for explaining this point to me.}
			
			We now claim that the kernel of \(\check{\mathcal{L}}:C^{4,\alpha}_{4-2n-\delta}\to C^{0,\alpha}_{-2n-\delta}\) is precisely \(\overline{\mathfrak{h}}\). Since \(\check{\mathcal{L}}\) has no indicial roots in \((4-2n,0)\) when \(n>2\), any \(f\in\ker(\check{\mathcal{L}})\cap C^{4,\alpha}_{4-2n-\delta}\) lies in \(C^{k,\alpha}_{\delta'}\) for all \(\delta'<0\) \cite[Proposition 12.2.1]{Pac08}. Now, \(0\) is an indicial root of the Laplacian, so there exists \(g\in\ker(\check{\mathcal{L}})\cap C^{4,\alpha}_0\) such that \(f-g\in C^{4,\alpha}_{\delta'}\) for \(\delta'>0\) sufficiently small \cite[Proposition 12.4.1]{Pac08}. Note that elements of \(C^{4,\alpha}_0\) are integrable on \(M_p\), and hence define distributions on \(M\). By elliptic regularity, the kernel of \(\check{\mathcal{L}}\) on \(C^{4,\alpha}_0\) is therefore \(\overline{\mathfrak{h}}\). Hence both \(g\) and \(f-g\) lie in \(\overline{\mathfrak{h}}\), and therefore so does \(f\). In the case of \(n=2\), note \(C^{4,\alpha}_{4-2n-\delta}(M_p)=C^{4,\alpha}_{-\delta}(M_p)\) and \(-\delta>0\), so \(f\) automatically lies in \(\overline{\mathfrak{h}}\) in this case.
		\end{proof}
	
		We also have the following result \cite[Proposition 16]{Sze12}.
	
		\begin{lemma}\label{lem:BS_right_inverse}
			For \(n>2\), let \(\delta>4-2n\). Then the operator \(C^{4,\alpha}_\delta(\Bl_0\mathbb{C}^n)^{T'}\to C^{0,\alpha}_{\delta-4}(\Bl_0\mathbb{C}^n)^{T'}\), \[\varphi\mapsto \Lich_\eta^*\Lich_\eta\varphi\] has a bounded inverse.
			
			When \(n=2\), let \(\delta\in(-1,0)\) and choose a compactly supported smooth \(T'\)-invariant function \(\chi\) on \(\Bl_0\mathbb{C}^2\) with non-zero integral. Then the operator \(C^{4,\alpha}_\delta(\Bl_0\mathbb{C}^2)^{T'}\times\mathbb{R}\to C^{4,\alpha}_{\delta-4}(\Bl_0\mathbb{C}^2)^{T'}\), \[(\varphi,t)\mapsto\Lich_\eta^*\Lich_\eta\varphi+t\chi\] has a bounded inverse.
		\end{lemma}
	
		Recall the functions \(\gamma_{j,\epsilon}(z)\) defined in \eqref{eq:gammas}, for \(j=1,2\). Before beginning the proof of Proposition \ref{prop:right_inverse}, following \cite[p. 16]{Sze12} for each \(j\) we will need to define a function \(\beta_{j,\epsilon}(z)\) that is equal to \(1\) on  \(\mathrm{supp}\,\gamma_{j,\epsilon}(z)\), and has slightly larger support than \(\gamma_{j,\epsilon}(z)\). 
	
		Write \(a:=\frac{2n-1}{2n+1}\), and recall \(r_\epsilon:=\epsilon^a\). Let us choose \(\overline{a}\) such that \(a<\overline{a}<1\), and let \(\chi_1:\mathbb{R}\to\mathbb{R}\) be a smooth function such that \(\chi_1(x)=1\) for \(x\leq a\) and \(\chi_1(x)=0\) for \(x\geq\overline{a}\). With this choice, let \(\beta_{1,\epsilon}\) be the function on \(\Bl_pM\) defined by \[\beta_{1,\epsilon}(z):=\chi_1\left(\frac{\log|z|}{\log\epsilon}\right),\] for \(z\in\widetilde{B}_1\backslash E\), extended to the constant function \(1\) on \(\Bl_pM\backslash\widetilde{B}_1\) and \(0\) on \(E\). Then \(\beta_{1,\epsilon}=0\) on \(\widetilde{B}_{\epsilon^{\overline{a}}}\), \(\beta_{1,\epsilon}=1\) on \(\Bl_pM\backslash \widetilde{B}_{r_\epsilon}\), and there are uniform bounds \[\|\beta_{1,\epsilon}\|_{C^{4,\alpha}_{0,\epsilon}}\leq C,\quad\quad\|\nabla_\epsilon\beta_{1,\epsilon}\|_{C^{3,\alpha}_{-1,\epsilon}}\leq \frac{C}{|\log\epsilon|}\] where \(C>0\) is independent of \(\epsilon\). In particular \(\beta_{1,\epsilon}\) is equal to 1 on \(\mathrm{supp}(\gamma_{1,\epsilon})\).
	
		Similarly, we choose \(\underline{a}\in\mathbb{R}\) such that \(0<\underline{a}<a\), and let \(\chi_2:\mathbb{R}\to\mathbb{R}\) be a smooth function such that \(\chi_2(x)=0\) for \(x<\underline{a}/a\) and \(\chi_2(x)=1\) for \(x>1\). We define the function \[\beta_{2,\epsilon}(z):=\chi_2\left(\frac{\log(|z|/2)}{\log r_\epsilon}\right)\] for \(z\in\widetilde{B}_1\backslash E\), extended similarly to \(\Bl_pM\). Then \(\beta_{2,\epsilon}\) satisfies \(\beta_{2,\epsilon}=1\) on \(\widetilde{B}_{2r_\epsilon}\), \(\beta_{2,\epsilon}=0\) on \(\Bl_pM\backslash\widetilde{B}_{2\epsilon^{\underline{a}}}\), and  there exists \(C>0\) independent of \(\epsilon\) such that \begin{equation}\label{eq:beta2_estimates}
			\|\beta_{2,\epsilon}\|_{C^{4,\alpha}_{0,\epsilon}}\leq C,\quad\quad\|\nabla_\epsilon\beta_{2,\epsilon}\|_{C^{3,\alpha}_{-1,\epsilon}}\leq\frac{C}{|\log\epsilon|}.
		\end{equation}
	
		\begin{proof}[Proof of Proposition \ref{prop:right_inverse}]
			
			We again drop \(\epsilon\) from the notation, writing \(G\) in place of \(G_\epsilon\), \(\mu\) for \(\mu_\epsilon\), and so on. We will prove the case \(n>2\); the case \(n=2\) requires only slight alterations, and we refer to \cite[Proposition 20]{Sze12} for the details. Given \(\varphi\in C^{0,\alpha}_{\delta-4}(\Bl_pM)^{T'}\), let us consider \(\gamma_1\varphi\) as a function on \(M_p\) and write \begin{equation}\label{eq:dummy}
				\widetilde{P}_1(\gamma_1\varphi):=(P_1(\gamma_1\varphi),f_\varphi)
			\end{equation} for the right-inverse operator \(\widetilde{P}_1\) of Lemma \ref{lem:M_p_right_inverse} applied to \(\gamma_1\varphi\). The function \(\beta_1P_1(\gamma_1\varphi)\) can then be considered as a function on \(\Bl_pM\) instead of \(M_p\). 
		
			Similarly, we let \(c_0:=v(p_0)/w(p_0)\), identify \(\gamma_2\varphi\) with the function \(\zeta\mapsto\gamma_2(\iota_\epsilon^{-1}(\zeta))\varphi(\iota_\epsilon^{-1}(\zeta))\) on \(\Bl_0\mathbb{C}^n\), and write \(P_2(\gamma_2\varphi)\) for the inverse operator \(P_2\) of \(c_0L_{\epsilon^2\eta}\) from Lemma \ref{lem:BS_right_inverse} applied to \(\gamma_2\varphi\). We then write \(\beta_2 P_2(\gamma_2\varphi)\) for the function \(\beta_2(z)P_2(\gamma_2\varphi)(\iota_\epsilon(z))\) on \(\Bl_pM\). Note that \(L_{\epsilon^2\eta}=\epsilon^{-4}L_\eta\), so \(P_2\) is \(\epsilon^4\) times a fixed operator.
			
			Using all this information, we now define the operators \(P:C^{0,\alpha}_{\delta-4}(\Bl_pM)^{T'}\to C^{4,\alpha}_{\delta}(\Bl_pM)^{T'}\) and \(\widetilde{P}:C^{0,\alpha}_{\delta-4}(\Bl_pM)^{T'}\to C^{4,\alpha}_{\delta}(\Bl_pM)^{T'}\times\overline{\mathfrak{h}}\) by \[P(\varphi):=\beta_1P_1(\gamma_1\varphi)+\beta_2P_2(\gamma_2\varphi),\] and \[\widetilde{P}(\varphi):=(P\varphi,f_\varphi),\] where \(f_\varphi\) is defined in \eqref{eq:dummy}. Note these operators depend on the parameter \(\epsilon\), and act on the corresponding weighted spaces defined in terms of \(\epsilon\). Our aim is now to prove that for \(\epsilon>0\) sufficiently small, \begin{equation}\label{eq:dummy_aim}
				\|(G\circ \widetilde{P})(\varphi)-\varphi\|_{C^{0,\alpha}_{\delta-4}}\leq\frac{1}{2}\|\varphi\|_{C^{0,\alpha}_{\delta-4}}.
			\end{equation} From this it follows that \(\|G\circ\widetilde{P}-\Id\|\leq \frac{1}{2}\), so the operator \(G\circ\widetilde{P}\) is invertible. Writing \(Q\) for the inverse, we have \(G\circ(\widetilde{P}\circ Q)=\Id\), and so \(\widetilde{P}\circ Q\) is a right inverse for \(G\). Since \(Q\) is constructed via a geometric series, \(\|Q\|\leq 2\). For the norm of \(\widetilde{P}\circ Q\) to be uniformly bounded, it then suffices to show the norm of \(\widetilde{P}\) is uniformly bounded independent of \(\epsilon\): \begin{align*}
				&\quad\,\, \|\widetilde{P}\varphi\|_{C^{4,\alpha}_\delta(\Bl_pM)\times\overline{\mathfrak{h}}} \\
				&= \|(P\varphi,f_\varphi)\|_{C^{4,\alpha}_\delta(\Bl_pM)\times\overline{\mathfrak{h}}} \\
				&= \|P\varphi\|_{C^{4,\alpha}_\delta}+|f_\varphi| \\
				&\leq \|\beta_1P_1(\gamma_1\varphi)\|_{C^{4,\alpha}_\delta}+\|\beta_2P_2(\gamma_2\varphi)\|_{C^{4,\alpha}_\delta}+\|\widetilde{P}_1(\gamma_1\varphi)\|_{C^{4,\alpha}_\delta(M_p)} \\
				&\leq \|\beta_1\|_{C^{4,\alpha}_0}\|P_1(\gamma_1\varphi)\|_{C^{4,\alpha}_\delta(M_p)}+\|\beta_2\|_{C^{4,\alpha}_0}\epsilon^{-\delta}\|P_2(\gamma_2\varphi)\|_{C^{4,\alpha}_\delta(\Bl_0\mathbb{C}^n)}+C\|\gamma_1\varphi\|_{C^{4,\alpha}_{\delta-4}(M_p)} \\
				&\leq C\left(\|\gamma_1\varphi\|_{C^{0,\alpha}_{\delta-4}(M_p)}+\epsilon^{-(\delta-4)}\|\gamma_2\varphi\|_{C^{0,\alpha}_{\delta-4}(\Bl_0\mathbb{C}^n)}\right) \\
				&\leq C\|\varphi\|_{C^{0,\alpha}_{\delta-4}}.
			\end{align*} In the fourth and sixth lines we have used the equivalence of norms \emph{(5)} from Lemma \ref{lem:weighted_norm_properties}, as well as the bounds on the norms of the \(\gamma_j\) from Lemma \ref{lem:metric_and_gamma_estimates}. We also used that \(\|P_2\|=\O(\epsilon^4)\), as we remarked above. Note in particular, \begin{equation}\label{eq:P2_bound}
			\|\beta_2P_2(\gamma_2\varphi)\|_{C^{4,\alpha}_\delta}\leq C\|\varphi\|_{C^{0,\alpha}_{\delta-4}}.
			\end{equation}
			
			It remains to prove \eqref{eq:dummy_aim}. We can write \((G\circ\widetilde{P})(\varphi)-\varphi\) as \begin{align}
				 & \check{L}(\beta_1P_1(\gamma_1\varphi))-\frac{1}{2}\widetilde{X}(\beta_1P_1(\gamma_1\varphi))-\gamma_1\ell(f_\varphi)-\gamma_1\varphi \label{eq:line_a1}\\
				 +&\check{L}(\beta_2P_2(\gamma_2\varphi))-\frac{1}{2}\widetilde{X}(\beta_2P_2(\gamma_2\varphi))-\gamma_2\ell(f_\varphi)-\gamma_2\varphi, \label{eq:line_a2}
			\end{align} and so to prove \eqref{eq:dummy_aim}, it suffices to estimate each of these expressions separately. The line \eqref{eq:line_a1} is supported in \(\Bl_pM\backslash \widetilde{B}_{\epsilon^{\overline{a}}}\), and \eqref{eq:line_a2} is supported in \(\widetilde{B}_{2\epsilon^{\underline{a}}}\). The estimate for \eqref{eq:line_a1} is essentially unchanged from \cite[Proposition 20]{Sze12}, only we use Proposition \ref{prop:weighted_Lich_estimate} instead of the unweighted version \cite[Proposition 18]{Sze12}, so we omit this. However, \eqref{eq:line_a2} introduces new challenges, since we are seeking an inverse of the weighted linearisation but have glued in an inverse for the unweighted linearisation on \(\Bl_0\mathbb{C}^n\). 
		
			Our task is therefore to estimate \eqref{eq:line_a2} on the region \(\widetilde{B}_{2\epsilon^{\underline{a}}}\). Using Lemma \ref{lem:weighted_linearisation} and defining \(\psi:=\beta_2P_2(\gamma_2\varphi)\), \eqref{eq:line_a2} can be written \begin{align}
				&\frac{v(\mu)}{w(\mu)}L\psi-\frac{2}{w(\mu)}(\db^*\Lich\psi,\nabla^{1,0}(v(\mu)))\label{eq:goal}\\
				&+\frac{1}{w(\mu)}(\Lich\psi,\Lich(v(\mu)))+ \frac{1}{2}S(\omega)\nabla\left(\frac{v(\mu)}{w(\mu)}\right)\cdot\nabla\psi \nonumber\\
				& +\frac{1}{2}\nabla\Phi_{v,w}(\omega)\cdot\nabla\psi -\frac{1}{2}\widetilde{X}(\psi) -\gamma_2\ell(f_\varphi)-\gamma_2\varphi. \nonumber
			\end{align} Setting \(c_0:=v(p_0)/w(p_0)\), we have \[\frac{v(\mu)}{w(\mu)}L\psi=\left(\frac{v(\mu)}{w(\mu)}-c_0\right)L\psi+c_0L\psi.\] We first estimate \[\left\|\left(\frac{v(\mu)}{w(\mu)}-c_0\right)L(\beta_2 P_2(\gamma_2\varphi))\right\|_{C^{0,\alpha}_{\delta-4}},\] for which we will show the uniform bound \begin{equation}\label{eq:bound}
			\left\|\frac{v(\mu)}{w(\mu)}-c_0\right\|_{C^{0,\alpha}_0\left(\widetilde{B}_{4\epsilon^{\underline{a}}}\right)}\leq C\epsilon^{2\underline{a}}.
			\end{equation} By the mean value theorem, it suffices to bound \[\left\|\mu-p_0\right\|_{C^{0,\alpha}_0\left(\widetilde{B}_{4\epsilon^{\underline{a}}}\right)}.\] On the region \(\widetilde{B}_{r_\epsilon}\), we have \(\mu-p_0=\epsilon^2\iota_\epsilon^*\mu_\eta\), where \(\iota_\epsilon:\widetilde{B}_{r_\epsilon}\to\widetilde{D}_{R_\epsilon}\subset\Bl_0\mathbb{C}^n\) is the scaling isomorphism. From this, \begin{align*}
				\|\mu-p_0\|_{C^{0,\alpha}_0\left(\widetilde{B}_{r_\epsilon}\right)} &= \epsilon^2\|\mu_\eta\|_{C^{0,\alpha}_0\left(\widetilde{D}_{R_\epsilon}\right)} \\
				&\leq \epsilon^2R_\epsilon^2\|\mu_\eta\|_{C^{0,\alpha}_2\left(\widetilde{D}_{R_\epsilon}\right)} \\
				&\leq C\epsilon^{2\underline{a}},
			\end{align*} where we use that \(\mu_\eta^a=\mu_{\mathrm{Euc}}^a+d^cg(\widetilde{\xi}_a)\) is \(\O(|\zeta|^2)\) to get a uniform bound \(\|\mu_\eta\|_{C^{0,\alpha}_2\left(\widetilde{B}_{R_\epsilon}\right)}\leq C\). Then on \(\widetilde{B}_{4\epsilon^{\underline{a}}}\backslash\widetilde{B}_{r_\epsilon}\), \[\mu-p_0=\mu_{\mathrm{Euc}}+d^c(\gamma_1(z)f(z)+\epsilon^2\gamma_2(z)g(\epsilon^{-1}z))\] is uniformly \(\O(|z|^2)\), so there is a uniform bound \[\|\mu-p_0\|_{C^{0,\alpha}_0\left(\widetilde{B}_{4\epsilon^{\underline{a}}}\backslash\widetilde{B}_{r_\epsilon}\right)}\leq C\epsilon^{2\underline{a}}\] and the bound \eqref{eq:bound} is achieved. Using this, \begin{align*}
				\left\|\left(\frac{v(\mu)}{w(\mu)}-c_0\right)L(\beta_2 P_2(\gamma_2\varphi))\right\|_{C^{0,\alpha}_{\delta-4}} &\leq C\epsilon^{2\underline{a}}\|L(\beta_2 P_2(\gamma_2\varphi))\|_{C^{0,\alpha}_{\delta-4}} \\
				&\leq C\epsilon^{2\underline{a}}\|\varphi\|_{C^{0,\alpha}_{\delta-4}},
			\end{align*} where we used \eqref{eq:P2_bound}.
		
			Next, consider the term \[\frac{1}{w(\mu)}(\Lich(\beta_2P_2(\gamma_2\varphi)),\Lich(v(\mu)))\] from \eqref{eq:goal}. Note that \[\Lich(v(\mu))=\sum_{a,b}v_{,ab}(\mu)(\widetilde{\xi}_b)^\flat\otimes\widetilde{\xi}_a\] and the right hand side has a uniform \(C^{0,\alpha}_0\)-bound on \(\Bl_pM\) independent of \(\epsilon\). Hence \begin{align*}
				&\left\|\frac{1}{w(\mu)}(\Lich(\beta_2P_2(\gamma_2\varphi)),\Lich(v(\mu)))\right\|_{C^{0,\alpha}_{\delta-4}} \\
				\leq & C\|\Lich(\beta_2P_2(\gamma_2\varphi))\|_{C^{0,\alpha}_{\delta-2}}\|\Lich(v(\mu))\|_{C^{0,\alpha}_{-2}\left(\widetilde{B}_{4\epsilon^{\underline{a}}}\right)} \\
				\leq & C\|\beta_2P_2(\gamma_2\varphi)\|_{C^{4,\alpha}_\delta}\left\|\sum_{a,b}v_{,ab}(\mu)(\widetilde{\xi}_b)^\flat\otimes\widetilde{\xi}_a\right\|_{C^{0,\alpha}_{-2}\left(\widetilde{B}_{4\epsilon^{\underline{a}}}\right)} \\
				\leq & C\epsilon^{2\underline{a}}\|\varphi\|_{C^{0,\alpha}_{\delta-4}}.
			\end{align*} Similarly for the second term in \eqref{eq:goal}, \begin{align*}
				&\left\|\frac{2}{w(\mu)}(\db^*\Lich\psi,\nabla^{1,0}(v(\mu)))\right\|_{C^{0,\alpha}_{\delta-4}} \\
				\leq & C\|\db^*\Lich\psi\|_{C^{0,\alpha}_{\delta-3}}\|\nabla^{1,0}(v(\mu))\|_{C^{0,\alpha}_{-1}\left(\widetilde{B}_{4\epsilon^{\underline{a}}}\right)} \\
				\leq & C\|\beta_2P_2(\gamma_2\varphi)\|_{C^{4,\alpha}_\delta}\left\|\sum_a v_{,a}(\mu)\widetilde{\xi}_a\right\|_{C^{0,\alpha}_{-1}\left(\widetilde{B}_{4\epsilon^{\underline{a}}}\right)} \\
				\leq & C\epsilon^{\underline{a}}\|\varphi\|_{C^{0,\alpha}_{\delta-4}},
			\end{align*} where \(\db^*:C^{1,\alpha}_{\delta-2}\to C^{0,\alpha}_{\delta-3}\) is seen to have uniformly bounded norm from \eqref{eq:db_adjoint}.
	
			For the term \(-\frac{1}{2}\widetilde{X}(\psi)\) from \eqref{eq:goal}, \begin{align*}
				&\|\widetilde{X}(\beta_2 P_2(\gamma_2\varphi))\|_{C^{0,\alpha}_{\delta-4}} \\
				\leq & \|\widetilde{X}\|_{C^{0,\alpha}_{-3}\left(\widetilde{B}_{4\epsilon^{\underline{a}}}\right)}\left(\|\nabla\beta_2\|_{C^{0,\alpha}_{-1}}\|P_2(\gamma_2(\varphi))\|_{C^{4,\alpha}_\delta\left(\widetilde{B}_{4\epsilon^{\underline{a}}}\right)}+\|\beta_2\|_{C^{0,\alpha}_0}\|\nabla(P_2(\gamma_2\varphi))\|_{C^{0,\alpha}_{\delta-1}\left(\widetilde{B}_{4\epsilon^{\underline{a}}}\right)}\right) \\
				\leq & C\epsilon^{3\underline{a}}\|\varphi\|_{C^{0,\alpha}_{\delta-4}}.
			\end{align*} 
			
			Next we bound the term \(\frac{1}{2}S(\omega)\nabla\left(\frac{v(\mu)}{w(\mu)}\right)\cdot\nabla\psi\) in \eqref{eq:goal}. Since \(\|S(\omega)\|_{C^{4,\alpha}_{-2}}\leq C\) we have \(\|S(\omega)\|_{C^{4,\alpha}_{-3}\left(\widetilde{B}_{4\epsilon^{\underline{a}}}\right)}\leq C\epsilon^{\underline{a}}\). Writing \(u=v/w\), note \(\nabla(u(\mu))=\sum_a u_{,a}(\mu)\widetilde{\xi}_a\) is uniformly bounded in \(C^{0,\alpha}_{0}\). Hence \begin{align*}
				\left\|S(\omega)\nabla\left(\frac{v(\mu)}{w(\mu)}\right)\cdot\nabla(\beta_2P_2(\gamma_2\varphi))\right\|_{C^{0,\alpha}_{\delta-4}} &\leq C\epsilon^{\underline{a}}\|\nabla(\beta_2P_2(\gamma_2\varphi))\|_{C^{0,\alpha}_{\delta-1}}\\
				&\leq C\epsilon^{\underline{a}}\|\varphi\|_{C^{0,\alpha}_{\delta-4}}.
			\end{align*} 
		
			For the term \(\nabla\Phi_{v,w}(\omega)\cdot\nabla(\beta_2P_2(\gamma_2\varphi))\) in \eqref{eq:goal}, it is enough to show that \[\|\nabla\Phi_{v,w}(\omega)\|_{C^{0,\alpha}_{-3}\left(\widetilde{B}_{4\epsilon^{\underline{a}}}\right)}\leq C\epsilon^{\underline{a}}.\] Lemma \ref{lem:Phi} states that \(\Phi_{v,w}(\omega)\) is a linear combination of terms of two kinds: \(u_{a}(\mu)\Delta\mu^a\), and \(u_{ab}(\mu)g(\widetilde{\xi}_a,\widetilde{\xi}_b)\), where \(u_a\) and \(u_{ab}\) are among finitely many fixed smooth functions on the moment polytope. We first have \[\nabla(u_a(\mu)\Delta\mu^a)=\sum_b u_{a,b}(\mu)\widetilde{\xi}_b\Delta\mu^a+u_a(\mu)\Ric(\widetilde{\xi}_a,-)^{\#}.\] By Lemma \ref{lem:estimates:Ricci_moment_map}, \[\|\Delta\mu^a\|_{C^{0,\alpha}_{-3}\left(\widetilde{B}_{4\epsilon^{\underline{a}}}\right)}\leq C\epsilon^{3\underline{a}}.\] Also, \[\|\Ric(\widetilde{\xi}_a,-)^{\#}\|_{C^{0,\alpha}_{-3}\left(\widetilde{B}_{4\epsilon^{\underline{a}}}\right)}\leq C\epsilon^{\underline{a}}\] by Lemma \ref{lem:metric_and_gamma_estimates}. The remaining factors \(u_{a,b}(\mu)\), \(u_a(\mu)\) and \(\widetilde{\xi}_b\) are uniformly bounded in \(C^{0,\alpha}_0\), hence \[\|\nabla(u_a(\mu)\Delta\mu^a)\|_{C^{0,\alpha}_{-3}\left(\widetilde{B}_{4\epsilon^{\underline{a}}}\right)}\leq C\epsilon^{\underline{a}}.\] Next, we must estimate \[\nabla(u_{ab}(\mu)g(\widetilde{\xi}_a,\widetilde{\xi}_b))=\sum_c u_{ab,c}(\mu)\widetilde{\xi}_c g(\widetilde{\xi}_a,\widetilde{\xi}_b)+u_{ab}(\mu)(g(\nabla\widetilde{\xi}_a,\widetilde{\xi}_b)+g(\widetilde{\xi}_a,\nabla\widetilde{\xi}_b)).\] The factors \(u_{ab,c}(\mu)\), \(\widetilde{\xi}_c\) and \(g_\epsilon(\widetilde{\xi}_a,\widetilde{\xi}_b)\) all have uniform \(C^{0,\alpha}_{0,\epsilon}\) bounds, so \[\|u_{ab,c}(\mu)\widetilde{\xi}_c g(\widetilde{\xi}_a,\widetilde{\xi}_b)\|_{C^{0,\alpha}_{-3}\left(\widetilde{B}_{4\epsilon^{\underline{a}}}\right)}\leq C\epsilon^{3\underline{a}}.\] For the term \(u_{ab}(\mu)(g(\nabla\widetilde{\xi}_a,\widetilde{\xi}_b)+g(\widetilde{\xi}_a,\nabla\widetilde{\xi}_b))\), the factors \(u_{ab}(\mu)\), \(g\), \(\widetilde{\xi}_a\) and \(\widetilde{\xi}_b\) have uniform \(C^{0,\alpha}_0\)-bounds. The covariant derivatives \(\nabla\widetilde{\xi}_a\) are uniformly bounded in \(C^{0,\alpha}_{-1}\) by Remark \ref{rem:coordinates}. Hence \[\|u_{ab}(\mu)(g(\nabla\widetilde{\xi}_a,\widetilde{\xi}_b)+g(\widetilde{\xi}_a,\nabla\widetilde{\xi}_b))\|_{C^{0,\alpha}_{-3}\left(\widetilde{B}_{4\epsilon^{\underline{a}}}\right)}\leq C\epsilon^{2\underline{a}}.\]
			
			Next consider the term \(\gamma_2\ell(f_\varphi)\) from \eqref{eq:goal}. By Lemma \ref{lem:ell_estimate}, \(\|\ell(f_\varphi)\|_{C^{0,\alpha}_0}\leq C|f_\varphi|\) where \(|\cdot|\) is a fixed norm on \(\overline{\mathfrak{h}}\), and by \eqref{eq:dummy} we have \(|f_\varphi|\leq C\|\varphi\|_{C^{0,\alpha}_{\delta-4}}\). It follows that \[\|\gamma_2\ell(f_\varphi)\|_{C^{0,\alpha}_{\delta-4}}\leq C\|\gamma_2\|_{C^{0,\alpha}_{\delta-4}}\|\ell(f_\varphi)\|_{C^{0,\alpha}_{0}}\leq Cr_\epsilon^{4-\delta}\|\varphi\|_{C^{0,\alpha}_{\delta-4}}.\]
			
			In summary, we have reduced estimating \eqref{eq:goal} to estimating \[c_0L(\beta_2P_2(\gamma_2\varphi))-\gamma_2\varphi.\] We apply \cite[Proposition 18]{Sze12}, which we recall is the unweighted analogue of Proposition \ref{prop:weighted_Lich_estimate}, to show the norm of \(L_\omega-L_{\epsilon^2\eta}\) tends to \(0\) over \(\widetilde{B}_{4\epsilon^{\underline{a}}}\) as \(\epsilon\to0\), for any \(k\geq0\) and \(\alpha\in(0,1)\). Here we write \(\epsilon^2\eta\) for what is strictly speaking \(\iota_\epsilon^*(\epsilon^2\eta)\). First note this operator vanishes identically on \(\widetilde{B}_{r_\epsilon}\), since \(\omega=\epsilon^2\eta\) on this region. So it suffices to estimate the norm over the annulus \(\widetilde{B}_{4\epsilon^{\underline{a}}}\backslash\widetilde{B}_{r_\epsilon}\). On the annulus we have \(\epsilon^2\eta-\omega=i\d\db\rho\), where \[\rho:=\gamma_1(z)(\epsilon^2g(\epsilon^{-1}z)-f(z)).\] Using the fact that \(f\) is \(\O(|z|^4)\) on \(M\) near \(p\), and \(g\) is \(\O(|\zeta|^{4-2n})\) on \(\Bl_0\mathbb{C}^n\) near \(\infty\), we can proceed as in the proof of Lemma \ref{lem:laplacian_order_tau} to show that \(\epsilon^2g(\epsilon^{-1}z)-f(z)\) is \(\O(|z|^{3})\) on \(\widetilde{B}_{4\epsilon^{\underline{a}}}\backslash\widetilde{B}_{r_\epsilon}\) for \(\tau>0\) small, so that \[\|\rho\|_{C^{k,\alpha}_2\left(\widetilde{B}_{4\epsilon^{\underline{a}}}\backslash \widetilde{B}_{r_\epsilon}\right)}\leq C\epsilon^{\underline{a}}\to0\] as \(\epsilon\to0\). It follows from \cite[Proposition 18]{Sze12} that the operator norm of \[L_\omega-L_{\epsilon^2\eta}:C^{4,\alpha}_\delta(\widetilde{B}_{4\epsilon^{\underline{a}}})\to C^{0,\alpha}_{\delta-4}(\widetilde{B}_{4\epsilon^{\underline{a}}})\] tends to \(0\) as \(\epsilon\to0\). Note the bound on \(\rho\) also implies the estimates in Lemma \ref{lem:metric_and_gamma_estimates} and \eqref{eq:beta2_estimates} hold with \({\epsilon^2\eta}\) in place of \(g_\epsilon\). 
			
			We have further reduced to estimating \[c_0L_{\epsilon^2\eta}(\beta_2P_2(\gamma_2\varphi))-\gamma_2\varphi.\] Now, \begin{align*}
				c_0L_{\epsilon^2\eta}(\beta_2P_2(\gamma_2\varphi))-\gamma_2\varphi &=  (\nabla^{1,0})^*\db^*\db(P_2(\gamma_2\varphi)\nabla^{1,0}\beta_2) \\
				&+(\nabla^{1,0})^*\db^*(\db\beta_2\otimes\nabla^{1,0}(P_2(\gamma_2\varphi))) \\
				&-2(\db^*\Lich (P_2(\gamma_2\varphi)),\nabla^{1,0}\beta_2)+(\Lich\beta_2,\Lich(P_2(\gamma_2\varphi))),
			\end{align*} where all gradients and adjoints are with respect to \(\epsilon^2\eta\). Note by the estimate \eqref{eq:beta2_estimates}, the right hand side satisfies \(\|\mathrm{RHS}\|_{C^{0,\alpha}_{\delta-4}\left(\widetilde{B}_{4\epsilon^{\underline{a}}}\right)}\leq C\frac{1}{|\log\epsilon|}\|\varphi\|_{C^{0,\alpha}_{\delta-4}}\). For example,  \begin{align*}
				\|(\Lich\beta_2,\Lich(P_2(\gamma_2\varphi)))\|_{C^{0,\alpha}_{\delta-4}\left(\widetilde{B}_{4\epsilon^{\underline{a}}}\right)} &\leq C\|\db \nabla^{1,0}\beta_2\|_{C^{0,\alpha}_{-2}}\|\Lich(P_2(\gamma_2\varphi))\|_{C^{0,\alpha}_{\delta-2}\left(\widetilde{B}_{4\epsilon^{\underline{a}}}\right)} \\
				&\leq C\|\nabla^{1,0}\beta_2\|_{C^{3,\alpha}_{-1}}\|P_2(\gamma_2\varphi)\|_{C^{4,\alpha}_\delta\left(\widetilde{B}_{4\epsilon^{\underline{a}}}\right)} \\
				&\leq C\frac{1}{|\log\epsilon|}\|\varphi\|_{C^{0,\alpha}_{\delta-4}}.
			\end{align*} 
			
			To summarise, we have shown that the \(C^{0,\alpha}_{\delta-4}\)-norm of \eqref{eq:goal} is bounded by \(c_\epsilon\|\varphi\|_{C^{0,\alpha}_{\delta-4}}\), where \(c_\epsilon\to0\) as \(\epsilon\to0\). Since \eqref{eq:goal} was equal to \((G\circ\widetilde{P})(\varphi)-\varphi\), we have shown the inequality \eqref{eq:dummy_aim} holds, which was our aim.
		\end{proof}
	
	\section{Completing the proof}
	
		In this section we will finish the proof of Proposition \ref{prop:approx_equation}, which we recall implies Theorem \ref{thm:main}, our main result. The proof is by a contraction mapping argument, which will deform our approximate solution \(\omega_\epsilon\) to \(\omega_\epsilon+i\d\db\varphi_\epsilon\) solving the equation \[S_{v,w}(\omega_\epsilon+i\d\db\varphi_\epsilon)-\frac{1}{2}\nabla_\epsilon\ell_\epsilon(f_\epsilon)\cdot\nabla_\epsilon\varphi_\epsilon=\ell_\epsilon(f_\epsilon),\] for some \(f_\epsilon\in\overline{\mathfrak{h}}\). We replace \(f_\epsilon\) with \(f_\epsilon+s\) where \(s:=S_{v,w}(\omega)\in\overline{\mathfrak{h}}\) generates the extremal field on \(M\). So we are trying to solve \[S_{v,w}(\omega_\epsilon+i\d\db\varphi_\epsilon)-\frac{1}{2}\nabla_\epsilon\ell_\epsilon(f_\epsilon+s)\cdot\nabla_\epsilon\varphi_\epsilon=\ell_\epsilon(f_\epsilon+s).\] We also require that \(f_\epsilon\) satisfies the expansion \[f_\epsilon=c_n\epsilon^{2n-2}\overline{\mu_H^{\#}(p)}+f_\epsilon'\] from Proposition \ref{prop:approx_equation} where \(c_n\) is a fixed constant and \(|f_\epsilon'|\leq C\epsilon^\kappa\) for some \(\kappa>2n-2\) and \(C>0\) independent of \(\epsilon\). We recall \(\mu_H^{\#}\) is the composition of the moment map \(\mu_H\) with the linear isomorphism \(\mathfrak{h}^*\to\mathfrak{h}\) determined by the inner product \eqref{eq:inner_product}, and \(\overline{\mu_H^{\#}(p)}\) denotes a fixed lift of \(\mu_H^{\#}(p)\) to \(\overline{\mathfrak{h}}\).
		
		As in \cite{APS11} and \cite{Sze12}, we will first modify \(\omega\) so that it matches with the Burns--Simanca metric to higher order near \(p\). To do this, we need:

		\begin{lemma}\label{lem:Gamma}
			There exists a \(T'\)-invariant smooth function \(\Gamma:M_p\to\mathbb{R}\) and an element \(h\in\overline{\mathfrak{h}}\) satisfying the equation \[\frac{v(\mu)}{w(\mu)}\Lich_v^*\Lich\Gamma=h|_{M_p}.\] Moreover, \(h=c_n\overline{\mu_H^{\#}(p)}\in\overline{\mathfrak{h}}\) for some \(c_n\in\mathbb{R}\), and \(\Gamma\) has the asymptotic behaviour \[\Gamma(z)=-|z|^{4-2n}+\O(|z|^{5-2n})\] near \(p\) when \(n>2\), and \[\Gamma(z)=\log|z|+\O(|z|^\tau)\] for all \(0<\tau<4/5\) when \(n=2\).
		\end{lemma}
	
		\begin{proof}
			We first treat the case \(n>2\). Let \(G:M_p\to\mathbb{R}\) be a smooth \(T'\)-invariant function equal to \(|z|^{4-2n}\) on \(B_{1/2}\backslash\{p\}\), and \(0\) on \(M\backslash B_1\). The highest order term of \(\frac{v}{w}\Lich^*_v\Lich\) is \(\frac{v}{w}\Delta^2\), and \(\Delta\) is asymptotic to \(\Delta_{\mathrm{Euc}}\) near \(p\) in the sense that \(\Delta-\Delta_{\mathrm{Euc}}=H^{j\bar{k}}\d_j\d_{\bar{k}}\) where \(H^{j\bar{k}}\) is \(\O(|z|)\); see the proof of Lemma \ref{lem:laplacian_order_tau}. From this, \(\Delta^2-\Delta_{\mathrm{Euc}}^2\) defines a bounded map \(C^{k,\alpha}_{\delta}(M_p)\to C^{k-4,\alpha}_{\delta-3}(M_p)\) for any weight \(\delta\). Since \(|z|^{4-2n}\) is the fundamental solution for \(\Delta_{\mathrm{Euc}}^2\) in Euclidean space, it follows that \[\frac{v}{w}\Lich_v^*\Lich G\in C^{k-4,\alpha}_{1-2n}(M_p).\] By Lemma \ref{lem:M_p_right_inverse}, there exist \(\varphi\in C^{4,\alpha}_{5-2n}(M_p)^{T'}\) and \(h\in\overline{\mathfrak{h}}\) such that \[\frac{v}{w}\Lich_v^*\Lich(\varphi-G)=h|_{M_p}.\] By elliptic regularity \(\varphi\) is smooth, and we have \[\Gamma:=\varphi-G=-|z|^{4-2n}+\O(|z|^{5-2n})\] near \(p\). From this expansion, \(\Gamma\) is a distributional solution to the equation \[\frac{v(\mu)}{w(\mu)}\Lich^*_v\Lich\Gamma=h-c_n\delta_p\] on \(M\), where \(c_n\) is a constant depending only on \(n\) and the weights \(v\), \(w\).
			
			We next treat the \(n=2\) case; let \(G:M_p\to\mathbb{R}\) be a smooth \(T'\)-invariant function equal to \(\log|z|\) on \(B_{1/2}\backslash\{p\}\), and \(0\) on \(M\backslash B_1\). In this case, the difference \(\Delta-\Delta_{\mathrm{Euc}}=H^{j\bar{k}}\d_j\d_{\bar{k}}\), where \(H^{j\bar{k}}\) is now \(\O(|z|^\tau)\) for all \(0<\tau<4/5\); see again the proof of Lemma \ref{lem:laplacian_order_tau}. The function \(-\log|z|\) is the fundamental solution of \(\Delta_{\mathrm{Euc}}^2\) on \(\mathbb{C}^2\), so \[\frac{v}{w}\mathcal{D}_v^*\mathcal{D}G\in C^{k-4,\alpha}_{-4+\tau}(M_p).\] Now, the same proof as for the \(n>2\) case in Lemma \ref{lem:M_p_right_inverse} in fact shows that \[(\varphi,f)\mapsto \frac{v(\mu)}{w(\mu)}\Lich_v^*\Lich\varphi-f\] on \(C^{4,\alpha}_{\delta'}(M_p)^{T'}\times\overline{\mathfrak{h}}\) has a right inverse for any \(\delta'\in(0,1)\) (as opposed to \(\delta'\in(-1,0)\)). Hence there exist \(\varphi\in C^{4,\alpha}_\tau(M_p)^{T'}\) and \(h\in\overline{\mathfrak{h}}\) such that \[\frac{v}{w}\Lich_v^*\Lich(\varphi-G)=h|_{M_p}.\] Defining \(\Gamma:=\varphi-G\) once again, we have by the expansion \(\Gamma=\log|z|+\O(|z|^\tau)\) that \(\Gamma\) is a distributional solution to the equation \[\frac{v(\mu)}{w(\mu)}\Lich^*_v\Lich\Gamma=h-c_2\delta_p\] on \(M\), where \(c_2\) is a constant depending only on the weights \(v\), \(w\).
			
			For any \(n\geq2\), denote by \(\pr:\overline{\mathfrak{h}}\to\mathfrak{h}\) the natural projection; we wish to show that \(\pr(h)=c_n\mu_H^{\#}(p)\). To see this, first note \(\langle\mu_H,\pr(h)\rangle = h+\mathrm{const}\). Recalling \(\mu_H\) is normalised to have integral \(0\) with respect to \(w(\mu)\omega^n\), for any \(\xi\in\mathfrak{h}\), \begin{align*}
			\langle \xi,\pr(h)\rangle_{\mathfrak{h}} 
			&=\int_M\langle\mu_H,\xi\rangle\langle\mu_H,\pr(h)\rangle w(\mu)\omega^n \\
			&= \int_M\langle\mu_H,\xi\rangle h\,w(\mu)\omega^n \\
			&= \int_M\langle\mu_H,\xi\rangle\left(\frac{v(\mu)}{w(\mu)}\Lich^*_v\Lich\Gamma+c_n\delta_p\right)w(\mu)\omega^n \\
			&= c_n\langle\mu_H(p),\xi\rangle \\
			&=c_n\langle \xi,\mu_H^{\#}(p)\rangle_{\mathfrak{h}},
			\end{align*} where in the fourth line we used that \(\langle\mu_H,\xi\rangle\in\ker\mathcal{D}\).
		\end{proof}

		On \(M_p\) define the metric \(\widetilde{\omega}\) by \[\widetilde{\omega}=\omega+\epsilon^{2n-2}i\d\db\Gamma.\] For \(n>2\), \(\widetilde{\omega}\) takes the form \begin{align*}
			\widetilde{\omega} &= i\d\db\left(|z|^2+\epsilon^{2n-2}\Gamma(z)+f(z)\right) \\
			&= i\d\db\left(|z|^2-\epsilon^2|\epsilon^{-1}z|^{4-2n}+\epsilon^{2n-2}\widetilde{\Gamma}(z)+f(z)\right),
		\end{align*} near \(p\), where \(\Gamma(z)=-|z|^{4-2n}+\widetilde{\Gamma}(z)\). Recall the Burns--Simanca metric satisfies \(\eta=i\d\db(|\zeta|^2+g(\zeta))\), where \(g(\zeta)=-|\zeta|^{4-2n}+\O(|\zeta|^{3-2n})\) as \(|\zeta|\to\infty\). Let us write explicitly \[g=-|\zeta|^{4-2n}+\widetilde{g},\] where \(\widetilde{g}\) is \(\O(|\zeta|^{3-2n})\). Gluing \(\widetilde{\omega}\) to the pullback of \(\epsilon^2\eta\) from \(\widetilde{D}_{R_\epsilon}\) to \(\widetilde{B}_{r_\epsilon}\), we produce the metric \begin{align}
		\widetilde{\omega}_\epsilon:= i\d\db\left(\vphantom{\widetilde{\Gamma}}\right.&|z|^2-\epsilon^2|\epsilon^{-1}z|^{4-2n}+\gamma_1(z)(\epsilon^{2n-2}\widetilde{\Gamma}(z)+f(z))\label{eq:modified_omega}\\
		&\left.+\gamma_2(z)\epsilon^2\widetilde{g}(\epsilon^{-1}z)\vphantom{\widetilde{\Gamma}}\right).\nonumber
		\end{align} Outside of \(\widetilde{B}_{2r_\epsilon}\) we have \(\widetilde{\omega}_\epsilon=\widetilde{\omega}\), and on \(\widetilde{B}_{r_\epsilon}\) we have \(\widetilde{\omega}_\epsilon=\iota_\epsilon^*(\epsilon^2\eta)\). Furthermore, \[\widetilde{\omega}_\epsilon=\omega_\epsilon+i\d\db(\epsilon^{2n-2}\gamma_1(z)\Gamma(z)).\] In the case \(n=2\), we define \(\widetilde{\omega}_\epsilon\) in the same way: \begin{equation}\label{eq:modified_omega2}
		\widetilde{\omega}_\epsilon=i\d\db\left(|z|^2+\epsilon^2\log|\epsilon^{-1}z|+\gamma_1(z)(\epsilon^{2}\widetilde{\Gamma}(z)+f(z))\right),
		\end{equation} where \(\Gamma(z)=\log|z|+\widetilde{\Gamma}(z)=\log|z|+\O(|z|^\tau)\) for \(\tau>0\) small, and \(g(\zeta)=\log|\zeta|\) so \(\widetilde{g}=0\).

		The equation we wish to solve is \[S_{v,w}(\omega_\epsilon+i\d\db \varphi_\epsilon)-\nabla_\epsilon\ell_\epsilon(f_\epsilon+s)\cdot\nabla_\epsilon \varphi_\epsilon=\ell_\epsilon(f_\epsilon+s),\] where \(\varphi_\epsilon\) is a \(T'\)-invariant smooth K{\"a}hler potential and \(f_\epsilon\in\overline{\mathfrak{h}}\). As done previously we will now drop the ever cumbersome \(\epsilon\), favouring \[S_{v,w}(\omega+i\d\db \varphi)-\nabla\ell(f+s)\cdot\nabla \varphi=\ell(f+s).\] For the original K{\"a}hler metric on \(M\) we will write \(\omega'\) when needed. Expanding the scalar curvature \(\check{S}:=S_{v,w}\) at \(\omega_\epsilon\), we can rewrite this as \[\check{L} (\varphi)-\frac{1}{2}\widetilde{X}(\varphi)-\ell(f)=\ell(s)-\check{S}(\omega)-\check{Q}(\varphi)+\frac{1}{2}\nabla\ell(f)\cdot\nabla \varphi.\] In order to incorporate the metric \(\widetilde{\omega}_\epsilon\) later on, we define \[\varphi':=\varphi-\epsilon^{2n-2}\gamma_1\Gamma\] and \[f':=f-\epsilon^{2n-2}h,\] where \(\Gamma\) and \(h\) are from Lemma \ref{lem:Gamma}. The equation can then be written \begin{align*}
			\check{L}(\varphi')-\frac{1}{2}\widetilde{X}(\varphi')-\ell(f') &= \ell(s)-\check{S}(\omega)-\check{Q}(\varphi)+\frac{1}{2}\nabla\ell(f)\cdot\nabla \varphi \\
			&-\check{L}(\epsilon^{2n-2}\gamma_1\Gamma)+\frac{1}{2}X(\epsilon^{2n-2}\gamma_1\Gamma)+\epsilon^{2n-2}\ell(h).
		\end{align*} We now use the operator \(P\) from Proposition \ref{prop:right_inverse}, which is a right-inverse for the operator on the left-hand side of this equation, to rewrite this as a fixed point problem \((\varphi',f')=\mathcal{N}(\varphi',f')\), where \begin{align*}
			\mathcal{N}(\varphi',f'):=P\left(\vphantom{\int}\right.&\ell(s)-\check{S}(\omega)-\check{Q}(\epsilon^{2n-2}\gamma_1\Gamma+\varphi') \\
			&+\frac{1}{2}\nabla(\epsilon^{2n-2}\ell(h)+\ell(f'))\cdot\nabla(\epsilon^{2n-2}\gamma_1\Gamma+\varphi') \\
			&\left. -\check{L}(\epsilon^{2n-2}\gamma_1\Gamma)+\frac{1}{2}X(\epsilon^{2n-2}\gamma_1\Gamma)+\epsilon^{2n-2}\ell(h)\vphantom{\int}\right).
		\end{align*}
	
		\begin{lemma}\label{lem:N_contraction}
			For \(n>2\) let \(\delta\in(4-2n,0)\), and for \(n=2\) let \(\delta<0\) be sufficiently close to \(0\). Then there exist constants \(c_0,\epsilon_0>0\) such that for all positive \(\epsilon<\epsilon_0\), if \(\varphi_1',\varphi_2'\in(C^{4,\alpha}_\delta)^{T'}\) satisfy \(\|\varphi_j'\|_{C^{4,\alpha}_2}<c_0\) and \(f_1',f_2'\in\mathfrak{h}\) satisfy \(|f_j'|<c_0\), then \[\|\mathcal{N}(\varphi_1',f_1')-\mathcal{N}(\varphi_2',f_2')\|_{C^{4,\alpha}_\delta}\leq \frac{1}{2}(\|\varphi_1'-\varphi_2'\|_{C^{4,\alpha}_\delta}+|f_1'-f_2'|).\]
		\end{lemma}
	
		\begin{proof}
			For \(n>2\), the operator \(P\) has norm uniformly bounded independent of \(\epsilon\), hence we must estimate the \(C^{0,\alpha}_{\delta-4}\)-norm of \[\check{Q}(\varphi_2)-\check{Q}(\varphi_1)+\frac{1}{2}\left(\nabla\ell(f_1)\cdot\nabla\varphi_1-\nabla\ell(f_2)\cdot\nabla\varphi_2\right).\] By Lemma \ref{lem:Q-estimate}, \[\|\check{Q}(\varphi_2)-\check{Q}(\varphi_1)\|_{C^{0,\alpha}_{\delta-4}}\leq C(\|\varphi_1\|_{C^{4,\alpha}_2}+\|\varphi_2\|_{C^{4,\alpha}_2})\|\varphi_1-\varphi_2\|_{C^{4,\alpha}_\delta}.\] Now, \(\varphi_j=\varphi_j'+\epsilon^{2n-2}\gamma_1\Gamma\), and since \(\Gamma\) is \(\O(|z|^{4-2n})\), \[\|\epsilon^{2n-2}\gamma_1\Gamma\|_{C^{4,\alpha}_2}\leq C\|\epsilon^{2n-2}\Gamma\|_{C^{4,\alpha}_2(M\backslash B_{r_\epsilon})}\leq C\epsilon^{2n-2}r_\epsilon^{4-2n}r_\epsilon^{-2}\to0\] as \(\epsilon\to0\). Hence, choosing \(c_0\) and \(\epsilon_0\) sufficiently small, we can ensure \[C(\|\varphi_1\|_{C^{4,\alpha}_2}+\|\varphi_2\|_{C^{4,\alpha}_2})\leq\frac{1}{4}.\] Noting that \(\|\varphi_1-\varphi_2\|_{C^{4,\alpha}_\delta}=\|\varphi_1'-\varphi_2'\|_{C^{4,\alpha}_\delta}\), this implies \[\|\check{Q}(\varphi_2)-\check{Q}(\varphi_1)\|_{C^{0,\alpha}_{\delta-4}}\leq\frac{1}{4}\|\varphi_1'-\varphi_2'\|_{C^{4,\alpha}_\delta}.\]
			
			For the remaining term, we use the uniform estimate \(\|\ell(f)\|_{C^{1,\alpha}_0}\leq C|f|\) from Lemma \ref{lem:ell_estimate}: \begin{align*}
			&\|\nabla\ell(f_1)\cdot\nabla\varphi_1-\nabla\ell(f_2)\cdot\nabla\varphi_2\|_{C^{0,\alpha}_{\delta-4}} \\
			\leq & \|\nabla\ell(f_1)\cdot(\nabla\varphi_1-\nabla\varphi_2)\|_{C^{0,\alpha}_{\delta-4}}+\|(\nabla\ell(f_1)-\nabla\ell(f_2))\cdot\nabla\varphi_2\|_{C^{0,\alpha}_{\delta-4}}\\
			\leq & C|f_1|\cdot\|\varphi_1-\varphi_0\|_{C^{1,\alpha}_{\delta-2}}+C|f_1-f_2|\cdot\|\varphi_2\|_{C^{1,\alpha}_{\delta-2}}.
			\end{align*} Again, choosing \(c_0\) sufficiently small, \[\|\nabla\ell(f_1)\cdot\nabla\varphi_1-\nabla\ell(f_2)\cdot\nabla\varphi_2\|_{C^{0,\alpha}_{\delta-4}}\leq\frac{1}{4}(\|\varphi_1'-\varphi_2'\|_{C^{4,\alpha}_\delta}+|f_1'-f_2'|).\]
			
			When \(n=2\), the norm of \(P\) is only bounded by \(C\epsilon^\delta\). However, since \(\Gamma\) is \(\O(\log|z|)\) we have \[\epsilon^\delta\|\epsilon^{2}\gamma_1\Gamma\|_{C^{4,\alpha}_2}\leq C\epsilon^{2+\delta}|\log r_\epsilon|r_\epsilon^{-2}\to0\] as \(\epsilon\to0\), so the proof still goes through in this case.
		\end{proof}
	
		\begin{proposition}\label{prop:N_bounded_norm}
			For \(n>2\) let \(\delta\in(4-2n,0)\) be sufficiently close to \(4-2n\), and for \(n=2\) let \(\delta<0\) be sufficiently close to \(0\). Then there exists \(C>0\) independent of \(\epsilon\) such that \[\|\mathcal{N}(0,0)\|_{C^{4,\alpha}_{\delta}}\leq Cr_\epsilon^{4-\delta}\epsilon^\theta,\] where \(\theta:=0\) for \(n>2\) and \(\theta:=\delta\) for \(n=2\).
		\end{proposition}
	
		\begin{proof}
			 Since the norm of \(P\) is uniformly bounded by \(C\epsilon^\theta\), it is enough to bound the \(C^{0,\alpha}_{\delta-4}\)-norm of \begin{align*}
				F :=&\ell(s)-\check{S}(\omega)-\check{Q}(\epsilon^{2n-2}\gamma_1\Gamma)
				+\frac{1}{2}\nabla(\epsilon^{2n-2}\ell(h))\cdot\nabla(\epsilon^{2n-2}\gamma_1\Gamma) \\
				&-\check{L}(\epsilon^{2n-2}\gamma_1\Gamma)+\frac{1}{2}X(\epsilon^{2n-2}\gamma_1\Gamma)+\epsilon^{2n-2}\ell(h)
			\end{align*} by \(Cr_\epsilon^{4-\delta}\).
			
			We first estimate \(F\) in the region \(\widetilde{B}_{r_\epsilon}\). Here the terms involving \(\gamma_1\) vanish, and we are left with \[F=\ell(s)-\check{S}(\omega)+\epsilon^{2n-2}\ell(h).\] By Lemma \ref{lem:ell_estimate}, \[\|\ell(s)\|_{C^{0,\alpha}_0}+\|\ell(h)\|_{C^{0,\alpha}_0}\leq C,\] which gives \[\|\ell(s)\|_{C^{0,\alpha}_{\delta-4}(\widetilde{B}_{r_\epsilon})}+\|\ell(h)\|_{C^{0,\alpha}_{\delta-4}(\widetilde{B}_{r_\epsilon})}\leq Cr_\epsilon^{4-\delta}\to0\] as \(\epsilon\to0\). For the term \(\check{S}(\omega)\), note that \(\omega=\epsilon^2\eta\) is scalar flat in this region, so we only need to estimate \(\Phi_{v,w}(\omega)\). The term \(\Phi_{v,w}(\omega)\) is a linear combination of terms of the form \(u_a(\mu)\Delta\mu^a\) and \(u_{ab}(\mu)g(\widetilde{\xi}_a,\widetilde{\xi}_b)\), where the \(u_a\) and \(u_{ab}\) are among finitely many fixed smooth functions on the moment polytope \(P\). From Sections \ref{sec:weighted_norms} and \ref{sec:estimates_moment_maps} we have \(C^{0,\alpha}_0\)-bounds for \(u_a(\mu)\), \(u_{ab}(\mu)\), \(g\), \(\widetilde{\xi}_a\), and \(\Delta\mu_\epsilon^a\). Using these bounds, we get \[\|\Phi_{v,w}(\omega)\|_{C^{0,\alpha}_{\delta-4}(\widetilde{B}_{r})}\leq Cr_\epsilon^{4-\delta}\to0\] as \(\epsilon\to0\).
			
			Next we estimate \(F\) on the region \(\Bl_pM\backslash\widetilde{B}_{2r_\epsilon}\). Here \(F\) reduces to \[\check{Q}_{\omega'}(\epsilon^{2n-2}\Gamma)+\epsilon^{4n-4}\frac{1}{2}\nabla_{\omega'} h\cdot\nabla_{\omega'}\Gamma,\] where \(\omega'\) is the original metric on \(M\). First note for \(n\geq2\) that \(\Gamma\) is \(\O(|z|^{4-2n-\tau})\) for all \(\tau>0\) small, so \[\|\Gamma\|_{C^{4,\alpha}_2(\Bl_pM\backslash\widetilde{B}_{2r_\epsilon})}\leq Cr_\epsilon^{4-2n-\tau}r_\epsilon^{-2}=Cr_\epsilon^{2-2n-\tau}.\] In particular \(\|\epsilon^{2n-2}\Gamma\|_{C^{4,\alpha}_2(\Bl_pM\backslash\widetilde{B}_{2r_\epsilon})}\leq CR_\epsilon^{2-2n}r_\epsilon^{-\tau}\to0\) as \(\epsilon\to0\) for \(\tau\) sufficiently small, and we can  apply Lemma \ref{lem:Q-estimate} to get \begin{align*}
			&\|\check{Q}_{\omega'}(\epsilon^{2n-2}\Gamma)\|_{C^{0,\alpha}_{\delta-4}(\Bl_pM\backslash\widetilde{B}_{2r_\epsilon})}\\
			\leq & C\|\epsilon^{2n-2}\Gamma\|_{C^{4,\alpha}_2(\Bl_pM\backslash\widetilde{B}_{2r_\epsilon})}\|\epsilon^{2n-2}\Gamma\|_{C^{4,\alpha}_\delta(\Bl_pM\backslash\widetilde{B}_{2r_\epsilon})}\\
			\leq & C\epsilon^{4n-4}r_\epsilon^{8-4n-2\tau}r_\epsilon^{-2}r_\epsilon^{-\delta} \\
			=& Cr_\epsilon^{4-\delta}\epsilon^{4n-4}r_\epsilon^{2-4n-2\tau}\\
			\leq & Cr_\epsilon^{4-\delta},
			\end{align*} where in the last line we used the explicit definition \(r_\epsilon:=\epsilon^{
			\frac{2n-1}{2n+1}}\) to conclude \(\epsilon^{4n-4}r_\epsilon^{2-4n-2\tau}=\epsilon^{\frac{4n-6}{2n+1}-2\tau\frac{2n-1}{2n+1}}\to0\) as \(\epsilon\to0\) for small \(\tau\). For the term \(\epsilon^{4n-4}\nabla_{\omega'} h\cdot\nabla_{\omega'}\Gamma\), \begin{align*}
			\|\epsilon^{4n-4}\nabla_{\omega'} h\cdot\nabla_{\omega'}\Gamma\|_{C^{0,\alpha}_{\delta-4}(\Bl_pM\backslash\widetilde{B}_{2r_\epsilon})} &\leq C\epsilon^{4n-4}\|h\|_{C^{1,\alpha}_0}\|\Gamma\|_{C^{1,\alpha}_{\delta-3}(\Bl_pM\backslash\widetilde{B}_{2r_\epsilon})} \\
			&\leq C\epsilon^{4n-4}r_\epsilon^{4-2n-\tau}r_\epsilon^{3-\delta} \\
			&\leq Cr_\epsilon^{4-\delta}.
			\end{align*} Here we used Lemma \ref{lem:ell_estimate} to bound \(\|h\|_{C^{1,\alpha}_0}\), and the definition of \(r_\epsilon\) to conclude \(\epsilon^{4n-4}r_\epsilon^{7-2n-\tau}\leq r_\epsilon^4\).

			Lastly we estimate \(F\) on the annulus \(A_\epsilon:=\widetilde{B}_{2r_\epsilon}\backslash\widetilde{B}_{r_\epsilon}\). Here, note that since \(\widetilde{\omega}_\epsilon=\omega_\epsilon+i\d\db(\epsilon^{2n-2}\gamma_1(z)\Gamma(z))\), we have \[\check{S}(\omega)+\check{L}(\epsilon^{2n-2}\gamma_1\Gamma)+\check{Q}(\epsilon^{2n-2}\gamma_1\Gamma)=\check{S}(\widetilde{\omega}).\] So \(F\) is given by \[\ell(s)-\check{S}(\widetilde{\omega})+\frac{1}{2}\epsilon^{4n-4}\nabla\ell(h)\cdot\nabla(\gamma_1\Gamma)+\frac{1}{2}X(\epsilon^{2n-2}\gamma_1\Gamma)+\epsilon^{2n-2}\ell(h).\] The same estimates as for the region \(\widetilde{B}_{r_\epsilon}\) will bound the terms  \(\ell(s)\) and \(\epsilon^{2n-2}\ell(h)\), and the terms \(\frac{1}{2}\epsilon^{4n-4}\nabla\ell(h)\cdot\nabla(\gamma_1\Gamma)\) and \(\frac{1}{2}X(\epsilon^{2n-2}\gamma_1\Gamma)\) are bounded similarly as \(\epsilon^{4n-4}\nabla h\cdot\nabla\Gamma\) was on the region \(\Bl_pM\backslash\widetilde{B}_{r_\epsilon}\). 
			
			The only term remaining to estimate is \(\check{S}(\widetilde{\omega})\). On the annulus \(A_\epsilon\) we can write \(\widetilde{\omega}_\epsilon=i\d\db\left(|z|^2+\rho\right)\), where  \[\rho:=-\epsilon^2|\epsilon^{-1}z|^{4-2n}+\gamma_1(z)(\epsilon^{2n-2}\widetilde{\Gamma}(z)+f(z))+\gamma_2(z)\epsilon^2\widetilde{g}(\epsilon^{-1}z)\] for \(n>2\), and \[\rho:=\epsilon^2\log|\epsilon^{-1}z|+\gamma_1(z)(\epsilon^{2}\widetilde{\Gamma}(z)+f(z))\] for \(n=2\). For each \(t\in[0,1]\) we define the metric \(\omega_t:=i\d\db\left(|z|^2+t\rho\right)\) on the annulus, so \(\omega_1=\widetilde{\omega}_\epsilon\) and \(\omega_0\) is the Euclidean metric. To each \(\omega_t\) we associate a moment map as follows. First, for \(\omega_1=\widetilde{\omega}_\epsilon\) we take the moment map \(\mu_1:=\mu_\epsilon+d^c(\epsilon^{2n-2}\gamma_1\Gamma)\). For the Euclidean metric \(\omega_0\), we proceed as follows. First, recall the original metric \(\omega'\) on \(M\) can be written \(\omega'=i\d\db\left(|z|^2+f(z)\right)\) near \(p\), where \(f\in\O(|z|^4)\). Using the function \(\beta_2\) from the proof of Lemma \ref{lem:BS_right_inverse}, we define \[\omega_s':=\omega'-si\d\db(\beta_2(z)f(z))\] on all of \(M\). Since \(f\) is \(\O(|z|^4)\), \(\omega_s'\) is a K{\"a}hler metric in the same class as \(\omega\) for \(\epsilon\) sufficiently small, and \[\mu_s'=\mu'-sd^c(\beta_2(z)f(z))\] is a moment map for \(\omega_s'\) with image in \(P\). Taking \(s=1\) and restricting to \(A_\epsilon\), we produce a moment map \(\mu_0\) for \(\omega_0\) on the annulus whose image lies in \(P\). By taking the convex combination of moment maps for \(\omega_0\) and \(\omega_1\), we have a moment map \(\mu_t\) for every \(\omega_t\) on the annulus with image lying in \(P\). It follows the functions \(v(\mu_t)\) and \(w(\mu_t)\) are well defined, and we can consider the operator \(S_{v,w}(\omega_t)\) and its linearisation for all \(t\). 
			
			By the mean value theorem, there exists \(t\in[0,1]\) such that \[\check{S}(\widetilde{\omega}_\epsilon)=\check{S}(\omega_0)+\check{L}_{\omega_t}\rho\] on \(A_\epsilon\). We estimate each term on the right. First, since \(\omega_0\) is scalar flat, the term \(\check{S}(\omega_0)\) is a sum of terms of the form \(u_\alpha(\mu_0)\Delta_0\mu_0\) and \(u_{ab}(\mu_0)g_0(\xi_a,\xi_b)\). These are fixed smooth functions independent of \(\epsilon\), so their \(C^{0,\alpha}_{\delta-4}\) norms over the region \(A_\epsilon\) are bounded by \(Cr_\epsilon^{4-\delta}\). For \(n>2\), \begin{align}
			\|\rho\|_{C^{4,\alpha}_\delta(A_\epsilon)} \leq & Cr_\epsilon^{-\delta}(\epsilon^{2n-2}r_\epsilon^{4-2n}+\epsilon^{2n-2}r_\epsilon^{5-2n}+r_\epsilon^4+\epsilon^2R_\epsilon^{3-2n}) \label{eq:rho1} \\
			\leq & Cr_\epsilon^{-\delta}(r_\epsilon^3+r_\epsilon^4+r_\epsilon^4+r_\epsilon^4)\nonumber\\
			\leq & Cr_\epsilon^{3-\delta}.\nonumber
			\end{align} Note to obtain sufficiently sharp bounds in the third line, we used the formula \(r_\epsilon:=\epsilon^{\frac{2n-1}{2n+1}}\). For example, writing \(\epsilon^{2n-2}r_\epsilon^{4-2n}=r_\epsilon^a\) for some \(a\), we solve \(a=5-\frac{2n+1}{2n-1}>3\) for all \(n>2\), so \(\epsilon^{2n-2}r_\epsilon^{4-2n}\leq r_\epsilon^3\). For \(n=2\), \begin{align}
						\|\rho\|_{C^{4,\alpha}_\delta(A_\epsilon)}
						\leq & Cr_\epsilon^{-\delta}(\epsilon^{2}|\log R_\epsilon|+\epsilon^{2}r_\epsilon^{\tau}+r_\epsilon^4) \label{eq:rho2} \\
						\leq & Cr_\epsilon^{-\delta}(r_\epsilon^3+r_\epsilon^4+r_\epsilon^4) \nonumber \\
						\leq & Cr_\epsilon^{3-\delta}, \nonumber
						\end{align} where we used \(\epsilon^2=r_\epsilon^{10/3}\) and \(\tau<4/5\) so \(\epsilon^2r_\epsilon^\tau\leq r_\epsilon^4\) for \(\tau\) sufficiently close to \(4/5\).
			
			Similarly, \(\|\rho\|_{C^{4,\alpha}_2(A_\epsilon)}\leq Cr_\epsilon^{3-2}=Cr_\epsilon\) for \(n\geq2\). Writing \(\check{L}_0\) for the linearisation of the weighted scalar curvature operator at \(\omega_0\), \[\check{L}_{\omega_t}\rho=\check{L}_0\rho+(\check{L}_{\omega_t}\rho-\check{L}_0\rho).\] We now claim that the analogue of Proposition \ref{prop:weighted_Lich_estimate} applies for the metric \(\omega_0\) on the region \(A_\epsilon\) in place of \(\omega_\epsilon\). To see this is rather straightforward; in particular one needs estimates \(\|g_0\|_{C^{k,\alpha}_{0,\epsilon}(A_\epsilon)}\leq C\) independent of \(\epsilon\), but this is even easier than the estimates for \(g_\epsilon\), since the metric is fixed independent of \(\epsilon\). The estimates on the metrics \(g_\epsilon\) were used to prove all of the corresponding moment map estimates in Section \ref{sec:estimates_moment_maps}, so these also hold for \(\omega_0\). In turn, the moment map estimates were applied to prove Proposition \ref{prop:weighted_Lich_estimate}, and so the Proposition holds for \(\omega_0\) in place of \(\omega_\epsilon\). Applying this, \[\|\check{L}_{\omega_t}\rho-\check{L}_0\rho\|_{C^{0,\alpha}_{\delta-4}(A_\epsilon)}\leq C \|\rho\|_{C^{4,\alpha}_2(A_\epsilon)}\|\rho\|_{C^{4,\alpha}_\delta(A_\epsilon)}\leq Cr_\epsilon^{4-\delta}\] on \(A_\epsilon\). The final term to estimate is \(\check{L}_0(\rho)\). The leading order term of \(\check{L}_0\) is \(D:=\frac{v(\mu_0)}{w(\mu_0)}\Delta_0^2\), which annihilates the leading order term \(-\epsilon^{2n-2}|z|^{4-2n}\) (or \(\epsilon^2\log|\epsilon^{-1}z|\) when \(n=2\)) of \(\rho\). Writing \(\widetilde{\rho}\) for \(\rho\) minus its leading order term, \[\|\check{L}_0(\rho)\|_{C^{0,\alpha}_{\delta-4}(A_\epsilon)}\leq \|(\check{L}_0-D)\rho\|_{C^{0,\alpha}_{\delta-4}(A_\epsilon)}+\|D\widetilde{\rho}\|_{C^{0,\alpha}_{\delta-4}(A_\epsilon)}.\] Now, from Lemma \ref{lem:weighted_linearisation}, \[(\check{L}_0-D)\rho=-\frac{2}{w(\mu_0)}(\db_0^*\mathcal{D}_0\rho,\nabla^{1,0}_0(v(\mu_0)))_0+D'\rho,\] where \(D'\) is a differential operator of order 2. For the first term on the right-hand side,  \[\nabla^{1,0}_0(v(\mu_0))=\sum_av_{,a}(\mu_0)\xi_a.\] This has uniformly bounded \(C^{0,\alpha}_{1}\)-norm on \(A_\epsilon\), since \(\xi_a\) vanishes at \(p\) and the norm on this region can be computed as the weighted norm on \(M_p\). It follows that \begin{align*}
			&\left\|-\frac{2}{w(\mu_0)}(\db_0^*\mathcal{D}_0\rho,\nabla^{1,0}_0(v(\mu_0)))_0\right\|_{C^{0,\alpha}_{\delta-4}(A_\epsilon)} \\
			\leq &C\|\db_0^*\mathcal{D}_0\rho\|_{C^{0,\alpha}_{\delta-5}(A_\epsilon)}\sum_a\|\xi_a\|_{C^{0,\alpha}_{1}(A_\epsilon)}\\
			\leq & C\|\rho\|_{C^{3,\alpha}_{\delta-2}(A_\epsilon)} \\
			\leq & Cr_\epsilon^{2-\delta}r_\epsilon^3 \\
			\leq&Cr_\epsilon^{4-\delta}.
			\end{align*} Since \(D'\) has order 2 with coefficients having bounded \(C^{k,\alpha}_0\)-norms, \[\|D'\rho\|_{C^{0,\alpha}_{\delta-4}(A_\epsilon)}\leq\|\rho\|_{C^{2,\alpha}_{\delta-2}(A_\epsilon)}\leq Cr_\epsilon^{4-\delta}.\] Lastly, to estimate \(D\widetilde{\rho}\), we use parts of the bounds in \eqref{eq:rho1} and \eqref{eq:rho2} to get \begin{align*}
			\|D\widetilde{\rho}\|_{C^{0,\alpha}_{\delta-4}(A_\epsilon)} & \leq \|\widetilde{\rho}\|_{C^{4,\alpha}_\delta(A_\epsilon)} \\
			&\leq Cr_\epsilon^{-\delta}r_\epsilon^4 \\
			&= Cr_\epsilon^{4-\delta}.
			\end{align*} We have bound all terms in \(F\) by \(Cr_\epsilon^{4-\delta}\), so the proof is complete.
		\end{proof}
		
		We can at last prove Proposition \ref{prop:approx_equation}, which we recall implies Theorem \ref{thm:main}.
		
		\begin{proof}[Proof of Proposition \ref{prop:approx_equation}]
		Given \(\delta\) so that Proposition \ref{prop:N_bounded_norm} holds, we have \[\|\mathcal{N}_\epsilon(0,0)\|_{C^{4,\alpha}_{\delta}}\leq C_1r_\epsilon^{4-\delta}\epsilon^\theta\] for all \(\epsilon>0\) sufficiently small, for a fixed constant \(C_1>0\). Define the set \[S_\epsilon:=\{(\varphi,f)\in(C^{4,\alpha}_\delta)^{T'}\times\overline{\mathfrak{h}}:\|\varphi\|_{C^{4,\alpha}_\delta}+|f|\leq 2C_1r_\epsilon^{4-\delta}\epsilon^\theta\}.\] For \(\epsilon\) sufficiently small, we will have \(2C_1r_\epsilon^{4-\delta}\epsilon^\theta<c_0\), where \(c_0\) is the constant of Lemma \ref{lem:N_contraction}. Hence for \((\varphi,f)\in S_\epsilon\), \begin{align*}
		\|\mathcal{N}_\epsilon(\varphi,f)\|_{C^{4,\alpha}_\delta} &\leq \|\mathcal{N}_\epsilon(0,0)\|_{C^{4,\alpha}_\delta}+\|\mathcal{N}_\epsilon(\varphi,f)-\mathcal{N}_\epsilon(0,0)\|_{C^{4,\alpha}_\delta} \\
		&\leq C_1r_\epsilon^{4-\delta}\epsilon^\theta+\frac{1}{2}(\|\varphi\|_{C^{4,\alpha}_\delta}+|f|) \\
		&\leq 2C_1r_\epsilon^{4-\delta}\epsilon^\theta.
		\end{align*} It follows that \(\mathcal{N}_\epsilon\) maps the set \(S_\epsilon\) to itself, and by Lemma \ref{lem:N_contraction}, \(\mathcal{N}_\epsilon\) is a contraction on \(S_\epsilon\). By the contraction mapping theorem, there exists a unique fixed point \((\varphi_\epsilon,f_\epsilon)\) of \(\mathcal{N}_\epsilon\) on the set \(S\). By construction, this fixed point solves the approximate weighted extremal equation \[S_{v,w}(\omega_\epsilon+i\d\db\varphi_\epsilon)-\nabla_\epsilon\ell_\epsilon(s+f_\epsilon+\epsilon^{2n-2}h)\cdot\nabla_\epsilon\varphi_\epsilon=\ell_\epsilon(s+f_\epsilon+\epsilon^{2n-2}h),\] where we recall \(h\) is the fixed function from Lemma \ref{lem:Gamma}. Thus, relating to Proposition \ref{prop:approx_equation}, we define \(h_{p,\epsilon}:=s+f_\epsilon+\epsilon^{2n-2}h\), and we have solved the required equation; note the solution is smooth, by elliptic regularity. All that remains to be seen is that \(h_{p,\epsilon}\) has the expansion \[h_{p,\epsilon}=s+\epsilon^{2n-2}c_n\overline{\mu_H^{\#}(p)}+h_{p,\epsilon}',\] where \(h_{p,\epsilon}'\) satisfies \(|h_{p,\epsilon}'|\leq C\epsilon^\kappa\) for some \(\kappa>2n-2\). By construction of \(h\) we have \(h=c_n\overline{\mu_H^{\#}(p)}\), and \(h_{p,\epsilon}'=f_\epsilon\) satisfies \[|f_\epsilon|\leq 2C_1r_\epsilon^{4-\delta}\epsilon^\theta\leq C\epsilon^{\frac{2n-1}{2n+1}(4-\delta)+\theta}.\] We can choose \(\delta\) as close to \(4-2n\) as required so that \(\frac{2n-1}{2n+1}(4-\delta)+\theta>2n-2\). This completes the proof.
		\end{proof}
		
		\section{Examples}
		
		We end by applying our theorem to specific choices of weight functions. Perhaps the best new result is in the extremal Sasaki case, where we can genuinely produce new extremal Sasaki metrics.
		
		\begin{enumerate}
		\item {\bf Extremal metrics.} Our first example is that of constant weight functions \(v\) and \(w\). In this setting, our main theorem \ref{thm:main} recovers Sz{\'e}kelyhidi's refinement of the Arezzo--Pacard--Singer theorem, that the blowup of an extremal manifold at a relatively stable fixed point of the extremal field admits an extremal metric in classes making the exceptional divisor small \cite{APS11,Sze12}.
		\item {\bf Extremal Sasaki metrics.} For our next example, we prove Corollary \ref{cor:Sasaki} on extremal Sasaki metrics.
				
				\begin{proof}[Proof of Corollary \ref{cor:Sasaki}]
				Consider the weight functions \[v:=(a+\ell_\xi)^{-n-1},\quad w:=(a+\ell_\xi)^{-n-3}.\] With this choice, a \((v,w)\)-weighted extremal metric on \(M\) in the class \(c_1(L)\) corresponds to an extremal Sasaki metric on the unit sphere bundle \(S\) of \(L^*\); see \cite[Theorem 1]{AC21}. Suppose we blow up \(M\) at a relatively stable fixed point \(p\) of the torus action and the extremal field. For \(\epsilon>0\) sufficiently small, by Theorem \ref{thm:main} there exists a weighted extremal metric in the class \([\pi^*L-\epsilon E]\), where \(E\) is the exceptional divisor of the blowup. If \(\epsilon\) is rational, the class \([\pi^*L-\epsilon E]\) will be that of a \(\mathbb{Q}\)-line bundle. The weighted extremal property is invariant under rescalings of the metric. Thus, if we rescale the class \([\pi^*L-\epsilon E]\) by an integer \(k\) such that \(k\epsilon\in\mathbb{Z}\), there will exist a \((v,w)\)-extremal metric in the class \([k\pi^*L-k\epsilon E]\), which is the first Chern class of an ample line bundle. It follows that the unit sphere bundle \(S_{\epsilon,k}\) of \((k\pi^*L-k\epsilon E)^*\) admits an extremal Sasaki metric. 
				\end{proof}
				\item {\bf K{\"a}hler--Ricci solitons and \(\mu\)-cscK metrics.} Finally, let us consider K{\"a}hler--Ricci solitons. We mentioned in the introduction that our result can never produce a K{\"a}hler--Ricci soliton on the blowup. However, we can produce a weighted extremal metric with weights \(v=w=e^{\langle\xi,-\rangle}\). This is almost a \(\mu\)-cscK metric in the sense of Inoue \cite{Ino22}, the only obstruction is that the extremal field for this metric might not equal \(\xi\). However, it is a small deformation of \(\xi\), since the weighted extremal metric on the blowup is a small deformation of \(\omega_\epsilon\). It would be interesting to know when this extremal field is equal to \(\xi\) itself, so that the blowup is genuinely \(\mu\)-cscK. 
		\end{enumerate}
		
	\bibliographystyle{alpha}
	\bibliography{references}

\begin{thebibliography}{{Hal}23}

\bibitem[AC21]{AC21}
V.~Apostolov and D.~M.~J. Calderbank.
\newblock The {CR} geometry of weighted extremal {K}\"{a}hler and {S}asaki
  metrics.
\newblock {\em Math. Ann.}, 379(3-4):1047--1088, 2021.

\bibitem[ACL21]{ACL21}
V.~Apostolov, D.~M.~J. Calderbank, and E.~Legendre.
\newblock Weighted {K}-stability of polarized varieties and extremality of
  {S}asaki manifolds.
\newblock {\em Adv. Math.}, 391:Paper No. 107969, 63, 2021.

\bibitem[AJL21]{AJL21}
V.~{Apostolov}, S.~{Jubert}, and A.~{Lahdili}.
\newblock {Weighted K-stability and coercivity with applications to extremal
  Kahler and Sasaki metrics}.
\newblock {\em arXiv e-prints}, page arXiv:2104.09709, April 2021.

\bibitem[AM19]{AM19}
V.~Apostolov and G.~Maschler.
\newblock Conformally {K}{\"a}hler {E}instein--{M}axwell geometry.
\newblock {\em J. Eur. Math. Soc}, 21(5):1319--1360, 2019.

\bibitem[AP06]{AP06}
C.~Arezzo and F.~Pacard.
\newblock Blowing up and desingularizing constant scalar curvature {K}\"{a}hler
  manifolds.
\newblock {\em Acta Math.}, 196(2):179--228, 2006.

\bibitem[AP09]{AP09}
C.~Arezzo and F.~Pacard.
\newblock Blowing up {K}\"{a}hler manifolds with constant scalar curvature.
  {II}.
\newblock {\em Ann. of Math. (2)}, 170(2):685--738, 2009.

\bibitem[APS11]{APS11}
C.~Arezzo, F.~Pacard, and M.~Singer.
\newblock Extremal metrics on blowups.
\newblock {\em Duke Math. J.}, 157(1):1--51, 2011.

\bibitem[Ati82]{Ati82}
M.~F. Atiyah.
\newblock Convexity and commuting {H}amiltonians.
\newblock {\em Bull. London Math. Soc.}, 14(1):1--15, 1982.

\bibitem[Bar86]{Bar86}
R.~Bartnik.
\newblock The mass of an asymptotically flat manifold.
\newblock {\em Comm. Pure Appl. Math.}, 39(5):661--693, 1986.

\bibitem[BGS08]{BGS08}
C.~P. Boyer, K.~Galicki, and S.~R. Simanca.
\newblock Canonical {S}asakian metrics.
\newblock {\em Comm. Math. Phys.}, 279(3):705--733, 2008.

\bibitem[BW14]{BWN14}
R.~J. {Berman} and D.~{Witt Nystrom}.
\newblock {Complex optimal transport and the pluripotential theory of
  K{\"a}hler-Ricci solitons}.
\newblock {\em arXiv e-prints}, page arXiv:1401.8264, January 2014.

\bibitem[Cal54]{Cal54}
E.~Calabi.
\newblock The space of {K}{\"a}hler metrics.
\newblock In {\em Proceedings of the {I}nternational {C}ongress of
  {M}athematicians}, volume~2, pages 206--207. E.P. Noordhoff, 1954.

\bibitem[Cal82]{Cal82}
E.~Calabi.
\newblock Extremal {K}\"{a}hler metrics.
\newblock In {\em Seminar on {D}ifferential {G}eometry}, volume 102 of {\em
  Ann. of Math. Stud.}, pages 259--290. Princeton Univ. Press, Princeton, N.J.,
  1982.

\bibitem[CSW18]{CSW18}
X.~Chen, S.~Sun, and B.~Wang.
\newblock K\"{a}hler-{R}icci flow, {K}\"{a}hler-{E}instein metric, and
  {K}-stability.
\newblock {\em Geom. Topol.}, 22(6):3145--3173, 2018.

\bibitem[CW20]{CW20}
X.~Chen and B.~Wang.
\newblock Space of {R}icci flows ({II})---{P}art {B}: {W}eak compactness of the
  flows.
\newblock {\em J. Differential Geom.}, 116(1):1--123, 2020.

\bibitem[Der18]{Der18}
R.~Dervan.
\newblock Relative {K}-stability for {K}\"{a}hler manifolds.
\newblock {\em Math. Ann.}, 372(3-4):859--889, 2018.

\bibitem[Don05]{Don05}
S.~K. Donaldson.
\newblock Lower bounds on the {C}alabi functional.
\newblock {\em J. Differential Geom.}, 70(3):453--472, 2005.

\bibitem[DS21]{DS21}
R.~{Dervan} and L.~M. {Sektnan}.
\newblock {Extremal K{\"a}hler metrics on blowups}.
\newblock {\em arXiv e-prints}, page arXiv:2110.13579, October 2021.

\bibitem[GS82]{GS82}
V.~Guillemin and S.~Sternberg.
\newblock Convexity properties of the moment mapping.
\newblock {\em Invent. Math.}, 67(3):491--513, 1982.

\bibitem[Hal22]{Hal22}
M.~Hallam.
\newblock {\em The geometry and stability of fibrations}.
\newblock PhD thesis, Oxford University, 2022.

\bibitem[{Hal}23]{sequel}
M.~{Hallam}.
\newblock {Stability of weighted extremal manifolds through blowups}.
\newblock {\em arXiv e-prints}, page arXiv:2309.02279, September 2023.

\bibitem[HL20]{HL20}
J.~Han and C.~Li.
\newblock On the {Y}au-{T}ian-{D}onaldson conjecture for generalized
  {K}{\"a}hler-{R}icci soliton equations.
\newblock {\em Communications on Pure and Applied Mathematics}, 2020.

\bibitem[{Ino}20]{Ino20}
E.~{Inoue}.
\newblock {Equivariant calculus on $\mu$-character and $\mu$K-stability of
  polarized schemes}.
\newblock {\em arXiv e-prints}, page arXiv:2004.06393, April 2020.

\bibitem[Ino22]{Ino22}
E.~Inoue.
\newblock Constant {$\mu$}-scalar curvature {K}\"{a}hler metric---formulation
  and foundational results.
\newblock {\em J. Geom. Anal.}, 32(5):Paper No. 145, 53, 2022.

\bibitem[Lah19]{Lah19}
A.~Lahdili.
\newblock K\"{a}hler metrics with constant weighted scalar curvature and
  weighted {K}-stability.
\newblock {\em Proc. Lond. Math. Soc. (3)}, 119(4):1065--1114, 2019.

\bibitem[Lah20]{Lah20}
A.~Lahdili.
\newblock Conformally {K}\"{a}hler, {E}instein-{M}axwell metrics and
  boundedness of the modified {M}abuchi functional.
\newblock {\em Int. Math. Res. Not. IMRN}, (22):8418--8442, 2020.

\bibitem[LeB88]{LeB88}
C.~LeBrun.
\newblock Counter-examples to the generalized positive action conjecture.
\newblock {\em Comm. Math. Phys.}, 118(4):591--596, 1988.

\bibitem[LeB10]{LeB10}
C.~LeBrun.
\newblock The {E}instein-{M}axwell equations, extremal {K}\"{a}hler metrics,
  and {S}eiberg-{W}itten theory.
\newblock In {\em The many facets of geometry}, pages 17--33. Oxford Univ.
  Press, Oxford, 2010.

\bibitem[LS94]{LS94}
C.~LeBrun and S.~R. Simanca.
\newblock Extremal {K}\"{a}hler metrics and complex deformation theory.
\newblock {\em Geom. Funct. Anal.}, 4(3):298--336, 1994.

\bibitem[Mab03]{Mab03}
T.~Mabuchi.
\newblock {M}ultiplier {H}ermitian structures on {K}{\"a}hler manifolds.
\newblock {\em Nagoya Math. J.}, 170:73--115, 2003.

\bibitem[Pac08]{Pac08}
F.~Pacard.
\newblock Connected sum constructions in geometry and nonlinear analysis.
\newblock Lecture notes, 2008.

\bibitem[Sim91]{Sim91}
S.~R. Simanca.
\newblock K\"{a}hler metrics of constant scalar curvature on bundles over
  {${\bf C}{\rm P}_{n-1}$}.
\newblock {\em Math. Ann.}, 291(2):239--246, 1991.

\bibitem[SS11]{SS11}
J.~Stoppa and G.~Sz{\'{e}}kelyhidi.
\newblock Relative {K}-stability of extremal metrics.
\newblock {\em J. Eur. Math. Soc. (JEMS)}, 13(4):899--909, 2011.

\bibitem[Sto09]{Sto09}
J.~Stoppa.
\newblock K-stability of constant scalar curvature {K}\"{a}hler manifolds.
\newblock {\em Adv. Math.}, 221(4):1397--1408, 2009.

\bibitem[{Sto}11]{Sto11}
J.~{Stoppa}.
\newblock {A note on the definition of K-stability}.
\newblock {\em arXiv e-prints}, page arXiv:1111.5826, Nov 2011.

\bibitem[Sz{\'{e}}12]{Sze12}
G.~Sz{\'{e}}kelyhidi.
\newblock On blowing up extremal {K}\"{a}hler manifolds.
\newblock {\em Duke Math. J.}, 161(8):1411--1453, 2012.

\bibitem[Sz{\'{e}}14]{Sze14}
G.~Sz{\'{e}}kelyhidi.
\newblock {\em An introduction to extremal {K}\"{a}hler metrics}, volume 152 of
  {\em Graduate Studies in Mathematics}.
\newblock American Mathematical Society, Providence, RI, 2014.

\bibitem[Sz{\'{e}}15]{Sze15}
G.~Sz{\'{e}}kelyhidi.
\newblock Blowing up extremal {K}\"{a}hler manifolds {II}.
\newblock {\em Invent. Math.}, 200(3):925--977, 2015.

\end{thebibliography}
	
\end{document}